\RequirePackage{filecontents}

\documentclass[reqno, 12pt]{amsart}

\usepackage{a4wide, amsmath, amscd, amsthm, amssymb}

\usepackage{dsfont}
\usepackage[all,cmtip]{xy}
\usepackage{enumerate}
\usepackage{multirow}
\usepackage{tabularx}

\usepackage{color}
\usepackage{ifpdf}
  \usepackage{graphicx}
\ifpdf 
  \DeclareGraphicsExtensions{.pdf,.png,.mps}
  \usepackage{pgf}
  \usepackage{epic}
\else 
  \usepackage{graphicx}
  \DeclareGraphicsExtensions{.eps,.bmp}
  \usepackage{pgf}
\fi
\usepackage{bez123}
\usepackage{wrapfig}

\newtheorem{thm}{Theorem}

\theoremstyle{definition}
\newtheorem{defn}[thm]{Definition}

\newtheorem*{rmk}{Remark}
\newtheorem*{ex}{Example}

\newtheorem{conj}[thm]{Conjecture}
\newcommand{\set}[1]{\left\{#1\right\}}
\newcommand{\abs}[1]{\left|#1\right|}

\def\S{\mathfrak{S}}

\def\RS{\mathcal{RS}}
\def\RSn{\RS_n}
\def\RSnk{\RS_{n,k}}
\def\RSnone{\RS_{n-1}}
\def\RSnkone{\RS_{n-1,k-1}}

\def\E{E}
\def\En{\E_n}
\def\Enk{\E_{n,k}}

\def\EnB{S_n}
\def\EnkB{S_{n,k}}

\def\A{\DU}

\def\Ank{\A_{n,k}}

\def\DU{\Alt}
\def\DUn{\DU_n}
\def\DUnk{\DU_{n,k}}

\def\DUB{\mathcal{S}}
\def\DUnB{\DUB_n}
\def\DUnkB{\DUB_{n,k}}

\def\DA{\mathcal{A}}
\def\DAn{\DA_n}
\def\DAnk{\DA_{n,k}}

\def\DAB{\DA^{(B)}}
\def\DAnB{\DAB_n}
\def\DAnkB{\DAB_{n,k}}

\def\DAH{\DA^{(H)}}
\def\DAnH{\DAH_n}
\def\DAnkH{\DAH_{n,k}}

\def\RSB{\RS^{(B)}}
\def\RSnB{\RSB_n}
\def\RSnkB{\RSB_{n,k}}
\def\RSnoneB{\RSB_{n-1}}
\def\RSnkoneB{\RSB_{n-1,k-1}}

\def\T{\mathcal{T}}
\def\Tn{\T_{n}}
\def\Tnk{\T_{n,k}}

\def\TB{\T^{(B)}}
\def\TnB{\TB_{n}}
\def\TnkB{\TB_{n,k}}

\def\a{\mathtt{a}}
\def\b{\mathtt{b}}
\def\c{\mathtt{c}}
\def\d{\mathtt{d}}

\DeclareMathOperator\Alt{\mathcal D\mathcal U}

\DeclareMathOperator\First{First}
\DeclareMathOperator\Last{Last}
\DeclareMathOperator\Pleaf{Leaf}

\mathchardef\mhyphen="2D
\DeclareMathOperator\les{(31\mhyphen 2)}
\DeclareMathOperator\less{(13\mhyphen 2)}
\DeclareMathOperator\res{(2\mhyphen 13)}
\DeclareMathOperator\ress{(2\mhyphen 31)}

\def\312{\les}
\def\132{\less}
\def\213{\res}
\def\231{\ress}

\def\N{\mathbb{N}}

\mathchardef\mhyphen="2D

\makeatletter
\@namedef{subjclassname@2020}{%
  \textup{2020} Mathematics Subject Classification}
\makeatother

\title{More bijections for Entringer and Arnold families}

\author{Heesung Shin}
\address[Heesung Shin]{Department of Mathematics, Inha University, Incheon 22212, Korea}
\email{shin@inha.ac.kr}

\author{Jiang Zeng}
\address[Jiang Zeng]{Universit\'{e} de Lyon;  Universit\'{e} Lyon 1; UMR 5208 du CNRS; Institut Camille Jordan;
43, boulevard du 11 novembre 1918, F-69622 Villeurbanne Cedex, France}
\email{zeng@math.univ-lyon1.fr}

\date{\today}



\begin{document}

\begin{abstract}
The Euler number $E_n$ (resp. Entringer number $E_{n,k}$)
enumerates the alternating (down-up) permutations of
 $\{1,\dots,n\}$ (resp.  starting with $k$). The Springer number $S_n$ 
 (resp. Arnold number $S_{n,k}$)
enumerates the type $B$ alternating permutations 
 (resp.  starting with $k$).
In  this paper, using bijections
 we first derive the counterparts
in {\em Andr\'e permutations} and {\em Simsun  permutations}
for the Entringer numbers $(E_{n,k})$, and then  the counterparts
in {\em signed Andr\'e permutations} and {\em type $B$ increasing 1-2 
trees} for the Arnold numbers $(S_{n,k})$. 
\end{abstract}

\maketitle

\tableofcontents
\section{Introduction}


The \emph{Euler numbers} $E_n$ are defined by 
the exponential generating function
$$
1 + \sum_{n\ge 1} E_n \frac{x^n}{n!} = \tan x + \sec x.
$$
This is the sequence A000111 in \cite{Slo18}.
In 1877 Seidel \cite{Sei77} defined the triangular array $(E_{n,k})$ by
the recurrence
\begin{align}
E_{n,k} = E_{n,k-1} + E_{n-1,n+1-k} \qquad (n \ge k\ge 2)
\label{eq:rec}
\end{align}
with $E_{1,1}=1$, $E_{n,1}=0$ ($n\ge2$), and proved that $E_n = \sum_k E_{n,k}$, i.e., the Euler number $E_n$ is the sum of the entries of the $n$-th row of the following triangle:

\begin{align}\label{tab:kempner}
\begin{tabular}{ccccccccc}
& &          &&  $E_{1,1}$ \\
&& & $E_{2,1}$ & $\rightarrow$ & $E_{2,2}$ \\
&& $E_{3,3}$ & $\leftarrow$ & $E_{3,2}$ & $\leftarrow$ & $E_{3,1}$ \\
&$E_{4,1}$ & $\rightarrow$ & $E_{4,2}$ & $\rightarrow$ & $E_{4,3}$ & $\rightarrow$ & $E_{4,4}$& \\
& & & & $\cdots$
\end{tabular}
=
\begin{tabular}{ccccccccc}
&&&          &  1 \\
&& & 0 & $\rightarrow$ & 1 \\
&& 1& $\leftarrow$ & 1 & $\leftarrow$ &0  \\
&0 & $\rightarrow$ & 1& $\rightarrow$ & 2& $\rightarrow$ & 2 \\
& & & & $\cdots$ \end{tabular}
\end{align}
Ths first few values of 
$\E_{n,k}$ are given in Table~\ref{tab:EntA}.

\begin{table}[t]
\centering
\begin{tabular}{c|ccccccc|c}
$n\setminus k$ & 1 & 2 & 3 & 4 & 5 & 6 & 7&$E_n$ \\ \hline
1 & 1&&&&&&&1 \\
2 & 0 & 1&&&&&& 1\\
3 & 0 & 1 & 1&&&&&2 \\
4 & 0 & 1 & 2 & 2&&&&5 \\
5 & 0 & 2 & 4 & 5 & 5&&&16\\
6 & 0 & 5 & 10 & 14 & 16 & 16&&61\\
7 & 0 & 16 & 32 & 46 & 56 & 61 & 61&271\\
\end{tabular}
\vspace{1em}
\caption{The Entringer numbers $\Enk$ for $1\leq k \leq n\leq 7$
and Euler numbers $\En=\sum_{k=1}^n \Enk$.}
\label{tab:EntA}
\end{table}

Andr\'e~\cite{And79} showed in  1879 that the Euler number $E_n$ enumerates 
the alternating permutations of $[n]:=\set{1,2,\dots,n}$, i.e., the permutations 
$\sigma_1\sigma_2\dots\sigma_n$ of $12\ldots n$ such that $\sigma_1 > \sigma_2 < \sigma_3 > \sigma_4 < \cdots.$
Let $\DUn$ be the set of (down-up) alternating permutations 
of $[n]$.
For example,
$$
\DU_4=\{2143, 3142, 3241, 4132, 4231\}.
$$
In 1933
Kempener~\cite{Kem33} used the boustrophedon algorithm~\eqref{tab:kempner} to enumerate  alternating permutations without refering to Euler numbers. Since  Entringer~\cite{Ent66} first 
found the combinatorial interpretation of  
Kempener's table $(E_{n,k})$  in terms of  Andr\'e's  model for Euler numbers,
the numbers $E_{n,k}$ are then called \emph{Entringer numbers}.
\begin{thm}[Entringer] 
The number of the (down-up) alternating permutations of $[n]$ with first entry $k$ is $E_{n,k}$,
i.e., $E_{n,k} = \# \DUnk$, where
$$
\DUnk := \set{\sigma\in\DUn : \sigma_1=k}.
$$
\end{thm}

According to Foata-Sch\"utzenberger~\cite{FS71}
a sequence of sets $(X_{n})$ is called an \emph{Andr\'e complex}
if the cardinality of $X_{n}$ is equal to $E_{n}$ for $n\geq 1$. 
Several other  Andr\'e complexs were introduced also in \cite{FS71} such as \emph{Andr\'e permutations of first and second kinds}, \emph{Andr\'e trees} or \emph{increasing 1-2 trees} 
and 
Rodica Simion and Sheila Sundaram~\cite{Sta10} discovered the \emph{Simsun} permutations; see Section~2. 
A sequence of sets $(X_{n,k})$ is called an \emph{Entringer family}
if the cardinality of $X_{n,k}$ is equal to $E_{n,k}$ for $1 \le k \le n$.
During the last two decades of the twentieth century, 
Poupard worked out several Entringer families in a series of papers~\cite{Pou82, Pou89, Pou97}.
Stanley \cite[Conjecture 3.1]{Sta94} and Hetyei \cite{Het96} introduced more Entringer families
by refining of Purtill's result~\cite[Theorem 6.1]{Pur93} about the $cd$-index of Andr\'e and Simsun permutations with fixed last letter.

The \emph{Springer numbers} $S_n$ are 
defined by the exponential generating function~\cite{Spr71}
$$
1 + \sum_{n\ge 1} S_n \frac{x^n}{n!}
= \frac{1}{\cos x - \sin x}.
$$
Arnold~\cite[p.11]{Arn92} showed in 1992 that $S_n$ enumerates a signed-permutation analogue of the alternating permutations. 
Recall that a \emph{signed permutation} of $[n]$ is a sequence $\pi=(\pi_1, \ldots, \pi_n)$ of elements of $[\pm n]:=\{-1,\dots, -n\}\cup \{1,\dots,n\}$ such that $|\pi|=(|\pi_1|, \ldots, |\pi_n|)$ is a permutation of $[n]$. 
We write $\mathcal{B}_n$ for  the set of all signed permutations of $[n]$.  
Clearly the cardinality of $\mathcal{B}_n$ is $2^n n!$.
An \emph{(down-up) alternating permutation of type $B_n$} 
is a signed permutation  $\pi\in \mathcal{B}_n$ such that $\pi_1>\pi_2<\pi_3>\pi_4\cdots$ 
and  a \emph{snake} of type $B_n$ is an alternating permutation of type $B_n$ starting with a positive entry. 
Let $\DUn^{(B)}$ be the set of (down-up) alternating permutations of type $B_n$ 
and $\DUnB$ the set of snakes of type $B_n$.   
Clearly the cardinality of $\DUn^{(B)}$ is $2^n E_n$. 
Arnold~\cite{Arn92} showed that the Springer number $S_n$ enumerates the snakes of type $B_n$. 
For example,
the $11$ snakes of $\DUB_3$ are  as follows:
$$
1\,\bar{2}\,3, 
~1\,\bar{3}\,2, 
~1\,\bar{3}\,\bar{2}, 
~2\,1\,3, 
~2\,\bar{1}\,3, 
~2\,\bar{3}\,1,
~2\,\bar{3}\,\bar{1}, 
~3\,1\,2, 
~3\,\bar{1}\,2, 
~3\,\bar{2}\,1, 
~3\,\bar{2}\,\bar{1},
$$
where we write $\bar{k}$ for $-k$. Arnold~\cite{Arn92} introduced the following
 pair of triangles to compute the Springer numbers:
\begin{align*}
\begin{tabular}{ccccccccc}
&&&&	\scriptsize $S_{1,-1}$ \\
&&&	\scriptsize $S_{2,2}$ & $\leftarrow$  & \scriptsize $S_{2,1}$ \\
&&	\scriptsize $S_{3,-3}$ & $\rightarrow$ & \scriptsize $S_{3,-2}$ & $\rightarrow$ & \scriptsize $S_{3,-1}$  \\
&	\scriptsize $S_{4,4}$ & $\leftarrow$ & \scriptsize $S_{4,3}$ & $\leftarrow$  & \scriptsize $S_{4,2}$ & 
$\leftarrow$  & \scriptsize $S_{4,1}$ \\
& & & & $\cdots$\\
\\
& & & & $\Updownarrow$\\
\\
&&&&	1 \\
&&&	2 & $\leftarrow$  & 1 \\
&&	0 & $\rightarrow$ & 2 & $\rightarrow$ & 3  \\
&	16 & $\leftarrow$ & 16 & $\leftarrow$  & 14 & $\leftarrow$  & 11 \\
& & & & $\cdots$
\end{tabular}
&&&
\begin{tabular}{ccccccccc}
&&&& 	\scriptsize $S_{1,1}$ \\
&&&	\scriptsize $S_{2,-1}$ & $\leftarrow$  & \scriptsize $S_{2,-2}$ \\
&&	\scriptsize $S_{3,1}$ & $\rightarrow$ & \scriptsize $S_{3,2}$ & $\rightarrow$ & \scriptsize $S_{3,3}$  \\
&	\scriptsize $S_{4,-1}$ & $\leftarrow$ & \scriptsize $S_{4,-2}$ & $\leftarrow$  & \scriptsize $S_{4,-3}$ & $\leftarrow$  & \scriptsize $S_{4,-4}$ \\
& & & & $\cdots$\\
\\
& & & & $\Updownarrow$\\
\\
&&&&	1 \\
&&&	1 & $\leftarrow$  & 0 \\
&&	3 & $\rightarrow$ & 4 & $\rightarrow$ & 4  \\
&	11 & $\leftarrow$ & 8 & $\leftarrow$  & 4 & $\leftarrow$  & 0 \\
& & & & $\cdots$
\end{tabular}
\end{align*}
where $S_{n,k}$ is defined by
$S_{1,1}=S_{1,-1}=1$, $S_{n,-n}=0$ ($n \ge 2$), and the recurrence
\begin{align}
S_{n,k} &=
\begin{cases}
S_{n,k-1}+ S_{n-1,-k+1} & \text{if $n \ge k >1$,}\\
S_{n,-1} & \text{if $n>k=1$,}\\
S_{n,k-1} + S_{n-1,-k}& \text{if $-1 \ge k > -n$.}
\end{cases}
\label{eq:recB}
\end{align}
\begin{thm}[Arnold]
For all integers $1\le k\le n$,
the number of the snakes of type $B_n$ starting with $k$ is $S_{n,k}$,
i.e., $S_{n,k} = \# \DUnkB$ with  $S_n = \sum_{k>0} \#\DUnkB$, where
$$
\DUnkB := \set{\sigma\in\DUnB : \sigma_1=k}.
$$
Moreover, for all integers $-n \le k \le -1$, it holds that 
\begin{align*}
S_{n,k} = \#\set{\sigma\in \DUn^{(B)} : \sigma_1=k }.
\end{align*}
\end{thm}

%


Similarly, the numbers $S_{n,k}$ are called \emph{Arnold numbers} and
a sequence of sets $(X_{n,k})$ is called an \emph{Arnold family} if the  cardinality of $X_{n,k}$  is equal to $S_{n,k}$ for $1 \le \abs{k} \le n$.
The first values of Arnold and Springer numbers are given in Table~\ref{tab:Arnold}.
\begin{table}[t]
\centering
\begin{tabular}{c|cccccc|cccccc|c}
$n\setminus k$ & -6 & -5 & -4 & -3 & -2 & -1  & 1 & 2 & 3 & 4 & 5 & 6 & $S_n$ \\
\hline
1 &     &     &     &     &     & 1   & 1   &     &     &     &     &     & $1$    \\
2 &     &     &     &     & 0   & 1   & 1   & 2   &     &     &     &     & $3$    \\
3 &     &     &     & 0   & 2   & 3   & 3   & 4   & 4   &     &     &     & $11$   \\
4 &     &     & 0   & 4   & 8   & 11  & 11  & 14  & 16  & 16  &     &     & $57$   \\
5 &     & 0   & 16  & 32  & 46  & 57  & 57  & 68  & 76  & 80  & 80  &     & $361$  \\
6 & 0   & 80  & 160 & 236 & 304 & 361 & 361 & 418 & 464 & 496 & 512 & 512 & $2763$
\end{tabular}
\vspace{1em}
\caption{The Arnold numbers $\EnkB$ for $1\leq \abs{k} \leq n\leq 6$ and Springer numbers $\EnB=\sum _{k=1}^n\EnkB$.}
\label{tab:Arnold}
\end{table}
The aim of this paper is to provide some new  Entringer families and
new  Arnold families by 
refining  known  combinatorial models for  Euler and Springer numbers.
To this end we shall build 
bijections between these new  Entringer (resp. Arnold) families with the known 
ones. 
We refer the reader to the more recent papers~\cite{Sta10, GSZ11, CEFP12, JNT12, FH14} related to the combinatorics of Euler numbers and Springer numbers.

This paper is organized as follows. 
In Section~\ref{sec:defn}, we shall give the necessary definitions 
 and present our main results. The proof of 
our  theorems will be given in Sections~\ref{sec:proof1}-\ref{sec:proof2}.
In Section~\ref{sec:psi},
we shall give more insightful description of two important  bijections. 
More precisely, Chuang et al.'s constructed a 
$\phi: \T_{n+1} \to \RS_{n}$ in \cite{CEFP12}, we show that  
 $\phi$ can be factorized as the compositions of two of our simpler bijections, and 
 then 
  give a  direct description of
Gelineau et al.'s bijection  $\psi: \DUn \to \Tn$ in \cite{GSZ11}.

\section{Definitions and main results} \label{sec:defn}

Let  $V$ be a finite subset of $\N$.
An \emph{increasing 1-2 tree} on $V$ is a vertex labeled rooted tree
with at most two (upward) branchings at any vertex and vertex labels in increasing order on any (upward) path from the root, see Figure~\ref{fig1}.
In what follows, when one draws an increasing 1-2 tree,
let's designate the \emph{left} child
if its parent has a unique child or it is smaller than the other sibling
and the others, except of the root, are designate the \emph{right} child.

\begin{figure}[t]
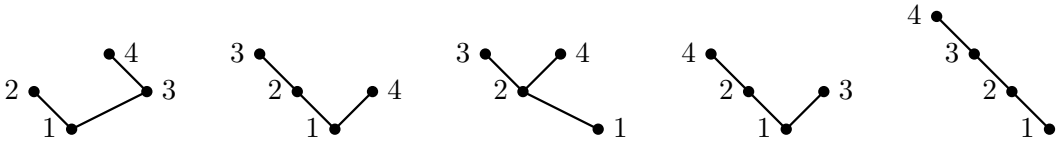
\label{fig1}
\centering
\begin{pgfpicture}{-11.00mm}{-3.86mm}{132.60mm}{19.14mm}
\pgfsetxvec{\pgfpoint{1.00mm}{0mm}}
\pgfsetyvec{\pgfpoint{0mm}{1.00mm}}
\color[rgb]{0,0,0}\pgfsetlinewidth{0.30mm}\pgfsetdash{}{0mm}
\pgfcircle[fill]{\pgfxy(0.00,0.00)}{0.60mm}
\pgfcircle[stroke]{\pgfxy(0.00,0.00)}{0.60mm}
\pgfputat{\pgfxy(-2.00,-1.00)}{\pgfbox[bottom,left]{\fontsize{11.38}{13.66}\selectfont \makebox[0pt][r]{1}}}
\pgfmoveto{\pgfxy(-5.00,5.00)}\pgflineto{\pgfxy(0.00,0.00)}\pgfstroke
\pgfcircle[fill]{\pgfxy(-5.00,5.00)}{0.60mm}
\pgfcircle[stroke]{\pgfxy(-5.00,5.00)}{0.60mm}
\pgfputat{\pgfxy(-7.00,4.00)}{\pgfbox[bottom,left]{\fontsize{11.38}{13.66}\selectfont \makebox[0pt][r]{2}}}
\pgfmoveto{\pgfxy(0.00,0.00)}\pgflineto{\pgfxy(10.00,5.00)}\pgfstroke
\pgfputat{\pgfxy(12.00,4.00)}{\pgfbox[bottom,left]{\fontsize{11.38}{13.66}\selectfont 3}}
\pgfcircle[fill]{\pgfxy(10.00,5.00)}{0.60mm}
\pgfcircle[stroke]{\pgfxy(10.00,5.00)}{0.60mm}
\pgfmoveto{\pgfxy(5.00,10.00)}\pgflineto{\pgfxy(10.00,5.00)}\pgfstroke
\pgfputat{\pgfxy(7.00,9.00)}{\pgfbox[bottom,left]{\fontsize{11.38}{13.66}\selectfont 4}}
\pgfcircle[fill]{\pgfxy(5.00,10.00)}{0.60mm}
\pgfcircle[stroke]{\pgfxy(5.00,10.00)}{0.60mm}
\pgfcircle[fill]{\pgfxy(35.00,0.00)}{0.60mm}
\pgfcircle[stroke]{\pgfxy(35.00,0.00)}{0.60mm}
\pgfputat{\pgfxy(33.00,-1.00)}{\pgfbox[bottom,left]{\fontsize{11.38}{13.66}\selectfont \makebox[0pt][r]{1}}}
\pgfmoveto{\pgfxy(30.00,5.00)}\pgflineto{\pgfxy(35.00,0.00)}\pgfstroke
\pgfcircle[fill]{\pgfxy(30.00,5.00)}{0.60mm}
\pgfcircle[stroke]{\pgfxy(30.00,5.00)}{0.60mm}
\pgfputat{\pgfxy(28.00,4.00)}{\pgfbox[bottom,left]{\fontsize{11.38}{13.66}\selectfont \makebox[0pt][r]{2}}}
\pgfmoveto{\pgfxy(25.00,10.00)}\pgflineto{\pgfxy(30.00,5.00)}\pgfstroke
\pgfputat{\pgfxy(23.00,9.00)}{\pgfbox[bottom,left]{\fontsize{11.38}{13.66}\selectfont \makebox[0pt][r]{3}}}
\pgfcircle[fill]{\pgfxy(25.00,10.00)}{0.60mm}
\pgfcircle[stroke]{\pgfxy(25.00,10.00)}{0.60mm}
\pgfcircle[fill]{\pgfxy(40.00,5.00)}{0.60mm}
\pgfcircle[stroke]{\pgfxy(40.00,5.00)}{0.60mm}
\pgfmoveto{\pgfxy(40.00,5.00)}\pgflineto{\pgfxy(35.00,0.00)}\pgfstroke
\pgfputat{\pgfxy(42.00,4.00)}{\pgfbox[bottom,left]{\fontsize{11.38}{13.66}\selectfont 4}}
\pgfcircle[fill]{\pgfxy(70.00,0.00)}{0.60mm}
\pgfcircle[stroke]{\pgfxy(70.00,0.00)}{0.60mm}
\pgfputat{\pgfxy(72.00,-1.00)}{\pgfbox[bottom,left]{\fontsize{11.38}{13.66}\selectfont 1}}
\pgfmoveto{\pgfxy(60.00,5.00)}\pgflineto{\pgfxy(70.00,0.00)}\pgfstroke
\pgfcircle[fill]{\pgfxy(60.00,5.00)}{0.60mm}
\pgfcircle[stroke]{\pgfxy(60.00,5.00)}{0.60mm}
\pgfputat{\pgfxy(58.00,4.00)}{\pgfbox[bottom,left]{\fontsize{11.38}{13.66}\selectfont \makebox[0pt][r]{2}}}
\pgfmoveto{\pgfxy(55.00,10.00)}\pgflineto{\pgfxy(60.00,5.00)}\pgfstroke
\pgfputat{\pgfxy(53.00,9.00)}{\pgfbox[bottom,left]{\fontsize{11.38}{13.66}\selectfont \makebox[0pt][r]{3}}}
\pgfcircle[fill]{\pgfxy(55.00,10.00)}{0.60mm}
\pgfcircle[stroke]{\pgfxy(55.00,10.00)}{0.60mm}
\pgfcircle[fill]{\pgfxy(65.00,10.00)}{0.60mm}
\pgfcircle[stroke]{\pgfxy(65.00,10.00)}{0.60mm}
\pgfmoveto{\pgfxy(65.00,10.00)}\pgflineto{\pgfxy(60.00,5.00)}\pgfstroke
\pgfputat{\pgfxy(67.00,9.00)}{\pgfbox[bottom,left]{\fontsize{11.38}{13.66}\selectfont 4}}
\pgfcircle[fill]{\pgfxy(95.00,0.00)}{0.60mm}
\pgfcircle[stroke]{\pgfxy(95.00,0.00)}{0.60mm}
\pgfputat{\pgfxy(93.00,-1.00)}{\pgfbox[bottom,left]{\fontsize{11.38}{13.66}\selectfont \makebox[0pt][r]{1}}}
\pgfmoveto{\pgfxy(90.00,5.00)}\pgflineto{\pgfxy(95.00,0.00)}\pgfstroke
\pgfcircle[fill]{\pgfxy(90.00,5.00)}{0.60mm}
\pgfcircle[stroke]{\pgfxy(90.00,5.00)}{0.60mm}
\pgfputat{\pgfxy(88.00,4.00)}{\pgfbox[bottom,left]{\fontsize{11.38}{13.66}\selectfont \makebox[0pt][r]{2}}}
\pgfmoveto{\pgfxy(95.00,0.00)}\pgflineto{\pgfxy(100.00,5.00)}\pgfstroke
\pgfputat{\pgfxy(102.00,4.00)}{\pgfbox[bottom,left]{\fontsize{11.38}{13.66}\selectfont 3}}
\pgfcircle[fill]{\pgfxy(100.00,5.00)}{0.60mm}
\pgfcircle[stroke]{\pgfxy(100.00,5.00)}{0.60mm}
\pgfmoveto{\pgfxy(85.00,10.00)}\pgflineto{\pgfxy(90.00,5.00)}\pgfstroke
\pgfputat{\pgfxy(83.00,9.00)}{\pgfbox[bottom,left]{\fontsize{11.38}{13.66}\selectfont \makebox[0pt][r]{4}}}
\pgfcircle[fill]{\pgfxy(85.00,10.00)}{0.60mm}
\pgfcircle[stroke]{\pgfxy(85.00,10.00)}{0.60mm}
\pgfcircle[fill]{\pgfxy(130.00,0.00)}{0.60mm}
\pgfcircle[stroke]{\pgfxy(130.00,0.00)}{0.60mm}
\pgfputat{\pgfxy(128.00,-1.00)}{\pgfbox[bottom,left]{\fontsize{11.38}{13.66}\selectfont \makebox[0pt][r]{1}}}
\pgfmoveto{\pgfxy(125.00,5.00)}\pgflineto{\pgfxy(130.00,0.00)}\pgfstroke
\pgfcircle[fill]{\pgfxy(125.00,5.00)}{0.60mm}
\pgfcircle[stroke]{\pgfxy(125.00,5.00)}{0.60mm}
\pgfputat{\pgfxy(123.00,4.00)}{\pgfbox[bottom,left]{\fontsize{11.38}{13.66}\selectfont \makebox[0pt][r]{2}}}
\pgfmoveto{\pgfxy(120.00,10.00)}\pgflineto{\pgfxy(125.00,5.00)}\pgfstroke
\pgfputat{\pgfxy(118.00,9.00)}{\pgfbox[bottom,left]{\fontsize{11.38}{13.66}\selectfont \makebox[0pt][r]{3}}}
\pgfcircle[fill]{\pgfxy(120.00,10.00)}{0.60mm}
\pgfcircle[stroke]{\pgfxy(120.00,10.00)}{0.60mm}
\pgfcircle[fill]{\pgfxy(115.00,15.00)}{0.60mm}
\pgfcircle[stroke]{\pgfxy(115.00,15.00)}{0.60mm}
\pgfmoveto{\pgfxy(115.00,15.00)}\pgflineto{\pgfxy(120.00,10.00)}\pgfstroke
\pgfputat{\pgfxy(113.00,14.00)}{\pgfbox[bottom,left]{\fontsize{11.38}{13.66}\selectfont \makebox[0pt][r]{4}}}
\end{pgfpicture}%
\caption{increasing 1-2 trees on $[4]$}
\end{figure}

For each vertex $v$ of a binary tree, by exchanging the left and right subtrees of $v$,  
we obtain another binary tree. This operation  is called a \emph{flip}.
If two binary trees can be connected by a sequence of flips, 
we say that these two trees are \emph{flip equivalent}.
Since flip equivalence is obviously an equivalence relation, 
we are able to define the equivalence classes of binary trees, 
which are corresponding to increasing 1-2 trees;  see \cite[Section 3.2]{Sta10} in detail.

\begin{defn}
Given an increasing 1-2 tree $T$,
the {\em minimal path} of $T$ is the unique sequence $(v_1, \cdots, v_{\ell})$ of vertices where $v_1$ is the root, 
$v_{k+1}$ is the left child of $v_k$ ($1\le k<\ell$), and $v_\ell$ is a leaf. 
The vertex $v_\ell$ is called the \emph{minimal leaf} of $T$ and denoted by  $\Pleaf(T)$. Similarly, the unique path $(v_1, \cdots, v_{\ell})$ from $v_1=v$
 to a leaf $v_{\ell}$ of $T$
is called the \emph{maximal path} from $v$ if
 $v_{k+1}$ is the right child of $v_k$ for $1\le k<\ell$.
\end{defn}
Let $\T_V$ be the set of increasing 1-2 trees on $V$ with 
$\T_{n}: = \T_{[n]}$ and
$$
\T_{n,k} = \set{T\in\T_n : \Pleaf(T)=k}.
$$

Donaghey~\cite{Don75} (see also \cite{Cal05}) proved bijectively that the Euler number $E_n$ enumerates 
the binary increasing trees in $\T_n$ and 
Poupard~\cite{Pou82} showed that
the sequence $(\T_{n,k})$ is an Entringer family.
In a previous work Gelineau et al.~\cite{GSZ11} proved bijectively Poupard's result by
establishing a 
bijection between  $\DUn$ and $\Tn$.
\begin{thm}[Gelineau-Shin-Zeng]
\label{thm:GSZ11}
There is an explicit  bijection $\psi: \DUn \to \Tn$ such that 
$$\Pleaf(\psi(\sigma))=\First(\sigma)$$
 for all $\sigma\in \DUn$, where
$\First(\sigma)$  is the first entry of $\sigma$.
\end{thm}

Let $\S_n$ be the group of permutations on $[n]$.
For a permutation $\sigma = \sigma_1 \dots \sigma_n \in \S_n$,
a \emph{descent} (resp. \emph{ascent}) of $\sigma$ is a pair $(\sigma_i, \sigma_{i+1})$
with $\sigma_i > \sigma_{i+1}$ (resp. $\sigma_i < \sigma_{i+1}$) and $1 \le i \le n-1$, 
a \emph{double descent} of $\sigma$ is a triple $(\sigma_i, \sigma_{i+1}, \sigma_{i+2})$ with $\sigma_i > \sigma_{i+1} > \sigma_{i+2}$ and $1 \le i \le n-2$, and
a \emph{valley} of $\sigma$ is a triple $(\sigma_i, \sigma_{i+1}, \sigma_{i+2})$ with 
$\sigma_i > \sigma_{i+1} < \sigma_{i+2}$ and $1 \le i \le n-2$.

Hetyei \cite[Definition 4]{Het96} defined recursively \emph{Andr\'e permutation of second kind}
if it is empty or satisfies the following:
\begin{enumerate}[(i)]
\item $\sigma$ has no double descents.
\item $(\sigma_{n-1}, \sigma_n)$ is not descent, i.e., $\sigma_{n-1} < \sigma_n$.
\item For all $2\le i \le n-1$, if $(\sigma_{i-1}, \sigma_{i}, \sigma_{i+1})$ is a valley of $\sigma$, then the minimum letter of $w_2$ is larger than the minimum letter of $w_4$ for the $\sigma_i$-factorization $(w_1, w_2, \sigma_i, w_4, w_5)$ of $\sigma$,
where the word $w_1 w_2 \sigma_i w_4 w_5$ is equal to $\sigma$
and $w_2$ and $w_4$ are maximal consecutive subwords of $\sigma$ satisfies its all letters are greater than $\sigma_i$.
\end{enumerate}

It is known that the above definition for Andr\'e permutation of second kind is simply equivalent to the following definition. Let  $\sigma_{[k]}$ denote the \emph{subword} of $\sigma$ consisting of $1,\dots, k$ in the order they appear in $\sigma$.
\begin{defn}
\label{def:andre}
A permutation $\sigma \in \S_n$ is called an \emph{Andr\'e permutation}
if $\sigma_{[k]}$ has no double descents and ends with an ascent for all $1\leq k\leq n$.
\end{defn}
For example, the permutation $\sigma = 43512$ is not Andr\'e
since the subword $\sigma_{[4]}=4312$ contains a double descent $(4,3,1)$,
while the permutation $\tau = 31245$ is Andr\'e since there is no double descent in the subwords:
\begin{align*}
\tau_{[1]} &= 1, &
\tau_{[2]} &= 12, &
\tau_{[3]} &= 312, &
\tau_{[4]} &= 3124, &
\tau_{[5]} &= 31245.
\end{align*}
Foata and Sch{\"u}tzenberger~\cite{FS73} proved that the Euler number $E_n$ enumerates the 
Andr\'e permutations in $\S_n$.
Let $\DAn$ be the set of Andr\'e permutations in $\S_n$.
For example, $$\DA_{4} = \set{1234, 1423, 3124, 3412, 4123}.$$

\begin{rmk}
\label{rem:andre}
Foata and Sch{\"u}tzenberger in \cite{FS73} introduced \emph{augmented Andr\'e permutation}
is a permutation $\sigma \in \S_n$, 
if $\sigma$ has no double descents, $\sigma_n = n$, and, for $1<j<k\le n$ satisfying 
\begin{align*}
\sigma_{j-1} &= \max\set{\sigma_{j-1}, \sigma_{j}, \sigma_{k-1}, \sigma_k} && \text{and}&
\sigma_k     &= \min\set{\sigma_{j-1}, \sigma_{j}, \sigma_{k-1}, \sigma_k},
\end{align*}
there exists $\ell$ such that $j<\ell <k$ and $\sigma_\ell < \sigma_k$.
\end{rmk}

\begin{defn}
\label{def:simsun}
A permutation $\sigma \in \S_n$ is called a \emph{Simsun permutation}
if $\sigma_{[k]}$ has no double descents for all $1\leq k\leq n$.
\end{defn}
By definition, an Andr\'e permutations is always a Simsun permutation, but the reverse is not true.
For example, the permutation $\sigma=25134$ is Simsun but not Andr\'e, because $\tau_{[2]} = 21$ ends with an descent:
\begin{align*}
\tau_{[1]} &= 1, &
\tau_{[2]} &= 21, &
\tau_{[3]} &= 213, &
\tau_{[4]} &= 2134, &
\tau_{[5]} &= 25134.
\end{align*}
Let $\RSn$ be the set of Simsun permutations in $\S_n$.
For example, $$\RS_{3} = \set{123, 132, 213, 231, 312}.$$

As for $\Ank$, we define two similar refinements of Andr\'e permutations and Simsun permutations as
\begin{align*}
\DAnk &:= \set{\sigma\in\DAn : \sigma_n=k}, &
\RSnk &:= \set{\sigma\in\RSn : \sigma_{n}=k}.
\end{align*}
Some examples are shown in Table~\ref{table:DARS}.

\begin{table}[t]
\centering
\begin{tabular}{c||c|c|c}
$k$ & $\DU_{4,k}$        & $\DA_{4,k}$        & $\RS_{3,k-1}$ \\ \hline
$2$ & $\set{2143}$       & $\set{3412}$       & $\set{231}$ \\
$3$ & $\set{3142, 3241}$ & $\set{1423, 4123}$ & $\set{132, 312}$ \\
$4$ & $\set{4132, 4231}$ & $\set{1234, 3124}$ & $\set{123, 213}$
\end{tabular}
\vspace{1em}
\caption{The sets $\DU_{4,k}$, $\DA_{4,k}$, and $\RS_{3,k-1}$ for $2\le k \le 4$}
\label{table:DARS}
\end{table}




Foata and Han \cite[Theorem 1.1 (iii)]{FH14} proved that  $\DAnk$ is an Entringer family by constructing  a
 bijection between $\DUnk$ and $\DAnk$.
We shall give an easier proof of their result 
by  constructing  a simpler bijection $\omega$ between $\Tnk$ and $\DAnk$. 
 Of course, combining  $\psi$ (cf. Theorem~\ref{thm:GSZ11}) and $\omega$ we obtain another bijection 
between $\DUnk$ and $\DAnk$.
\begin{thm}
\label{main2}
For positive integer $n\ge 1$,
there is a bijection $\omega: \Tn \to \DAn$ such that
\begin{align}
\Pleaf(T)=\Last(\omega(T)) \label{eq:Andre}
\end{align}
for all $T \in \Tn$,
where $\Last(\sigma)$  is the last entry of $\sigma$.
In words, for all $1\le k \le n$,
the mapping $\omega$ is a bijection from $\Tnk$ onto $\DAnk$.
\end{thm}


Whereas one can easily show that the cardinality $\#\A_{n,k}$ of (down-up) alternating permutations of length $n$ with first entry $k$ satisfies \eqref{eq:rec},
it seems hard to show that
the cardinality $\#\DAnk$ of Andr\'e permutations of length $n$ with last entry $k$ or
the cardinality $\#\RSnkone$ of Simsun permutations of length $n-1$ with last entry $k-1$ does.
Thus, in order to show \eqref{eq:Andre} and \eqref{eq:Simsun}, we shall construct a bijection between $\DUnk$, $\DAnk$, $\RSnkone$, and other known Entringer families in \cite{GSZ11}.


Stanley \cite[Conjecture 3.1]{Sta94} conjectured a refinement of Purtill's result~\cite[Theorem 6.1]{Pur93} about the $cd$-index of Andr\'e and Simsun permutations with fixed last letter. 
In this conjecture, he mentioned three kinds of Andr\'e permutations:
(i) Augmented Andr\'e permutations in Remark~\ref{rem:andre}, 
(ii) Andr\'e permutations in Definition~\ref{def:andre}, and
(iii) augmented Sundaram permutations,
where the third corresponds to Simsun permutations in Definition~\ref{def:simsun} by removing last letter. 
Hetyei \cite{Het96} proved the conjecture for the second and the third by verifying that both sides satisfy the same recurrence.
In particular, he proves the following result.
\begin{thm}[Hetyei]
For all $1\le k \le n$, we have that two cardinalities of $\DAnk$ and $\RSnkone$ are same, that is,
\begin{align}
\# \DAnk &= \# \RSnkone.
\end{align}
\end{thm}

In the next theorem, we give a bijective
proof of  the conjecture of Stanley by constructing an explicit  bijection.
\begin{thm}\label{main1}
For positive integer $n\ge 1$, there is
a bijection $\varphi: \DAn\to \RSnone$ such that
\begin{align}
\Last(\sigma)-1=\Last(\varphi(\sigma))  \label{eq:Simsun}
\end{align}
for all $\sigma \in \DAn$.
In words,  the mapping  $\varphi$ is a bijection  from $\DAnk$ onto
$\RSnkone$.
Moveover, the bijection $\varphi$ preserves the $cd$-index of Andr\'e and Simsun permutation.
\end{thm}

Given a permutation $\sigma \in B_n$,
denote $\sigma_{[k]}$ the \emph{subword} of $\sigma$ consisting of $k$ smallest entries in the order they appear in $\sigma$.
A \emph{signed Andr\'e permutation} of $[n]$ is a permutation $\sigma \in B_n$
such that $\sigma_{[k]}$ has no double descents and ends with an ascent for all $1\leq k\leq n$.
Let $\DAnB$ be the set of signed Andr\'e permutations of $[n]$
and $\DAnkB$ be the set of signed Andr\'e permutations $\sigma$ in $\DAnB$ ending with entry $k$.
For example, the permutation $\sigma = 2\bar{4}\bar{1}35$ is And\'re due to
\begin{align*}
\sigma_{[1]} &= \bar{4}, &
\sigma_{[2]} &= \bar{4}\bar{1}, &
\sigma_{[3]} &= 2\bar{4}\bar{1}, &
\sigma_{[4]} &= 2\bar{4}\bar{1}3, &
\sigma_{[5]} &= 2\bar{4}\bar{1}35.
\end{align*}
Some examples of $\DAB_{3,k}$ are shown in Table~\ref{table:DARSB}.

\begin{defn}
A \emph{type $B$ increasing 1-2 tree} on $[n]$ is a binary tree with $n$ signed labels 
in $\{\pm1,\pm2,\dots,\pm n\}$  
such that the absolute values of signed labels are distinct and any vertex is greater than its children.
\end{defn}

For example, all type $B$ increasing 1-2 trees on $[3]$ are given in Figure~\ref{fig2}.
\begin{figure}[t]
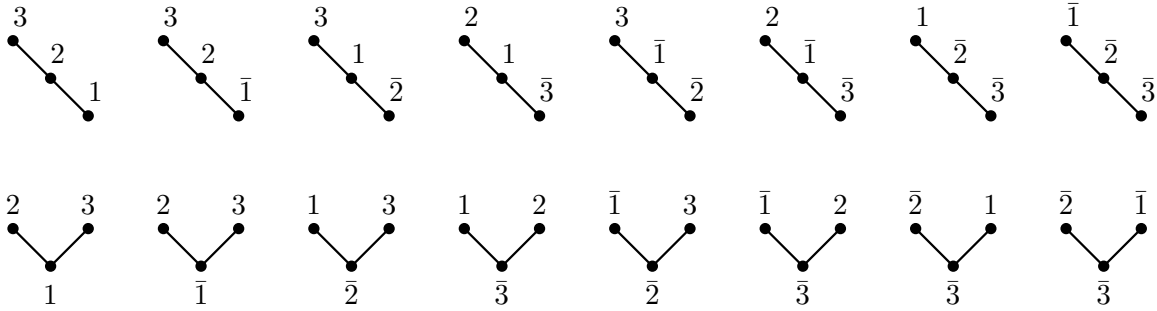
\label{fig2}
\centering
\begin{pgfpicture}{-8.00mm}{-7.86mm}{150.33mm}{37.14mm}
\pgfsetxvec{\pgfpoint{1.00mm}{0mm}}
\pgfsetyvec{\pgfpoint{0mm}{1.00mm}}
\color[rgb]{0,0,0}\pgfsetlinewidth{0.30mm}\pgfsetdash{}{0mm}
\pgfcircle[fill]{\pgfxy(5.00,20.00)}{0.60mm}
\pgfcircle[stroke]{\pgfxy(5.00,20.00)}{0.60mm}
\pgfmoveto{\pgfxy(0.00,25.00)}\pgflineto{\pgfxy(5.00,20.00)}\pgfstroke
\pgfcircle[fill]{\pgfxy(0.00,25.00)}{0.60mm}
\pgfcircle[stroke]{\pgfxy(0.00,25.00)}{0.60mm}
\pgfmoveto{\pgfxy(-5.00,30.00)}\pgflineto{\pgfxy(0.00,25.00)}\pgfstroke
\pgfputat{\pgfxy(5.00,22.00)}{\pgfbox[bottom,left]{\fontsize{11.38}{13.66}\selectfont 1}}
\pgfputat{\pgfxy(0.00,27.00)}{\pgfbox[bottom,left]{\fontsize{11.38}{13.66}\selectfont 2}}
\pgfputat{\pgfxy(-5.00,32.00)}{\pgfbox[bottom,left]{\fontsize{11.38}{13.66}\selectfont 3}}
\pgfcircle[fill]{\pgfxy(-5.00,30.00)}{0.60mm}
\pgfcircle[stroke]{\pgfxy(-5.00,30.00)}{0.60mm}
\pgfcircle[fill]{\pgfxy(0.00,0.00)}{0.60mm}
\pgfcircle[stroke]{\pgfxy(0.00,0.00)}{0.60mm}
\pgfputat{\pgfxy(0.00,-5.00)}{\pgfbox[bottom,left]{\fontsize{11.38}{13.66}\selectfont \makebox[0pt]{1}}}
\pgfmoveto{\pgfxy(-5.00,5.00)}\pgflineto{\pgfxy(0.00,0.00)}\pgfstroke
\pgfcircle[fill]{\pgfxy(-5.00,5.00)}{0.60mm}
\pgfcircle[stroke]{\pgfxy(-5.00,5.00)}{0.60mm}
\pgfputat{\pgfxy(-5.00,7.00)}{\pgfbox[bottom,left]{\fontsize{11.38}{13.66}\selectfont \makebox[0pt]{2}}}
\pgfmoveto{\pgfxy(0.00,0.00)}\pgflineto{\pgfxy(5.00,5.00)}\pgfstroke
\pgfputat{\pgfxy(5.00,7.00)}{\pgfbox[bottom,left]{\fontsize{11.38}{13.66}\selectfont \makebox[0pt]{3}}}
\pgfcircle[fill]{\pgfxy(5.00,5.00)}{0.60mm}
\pgfcircle[stroke]{\pgfxy(5.00,5.00)}{0.60mm}
\pgfcircle[fill]{\pgfxy(145.00,20.00)}{0.60mm}
\pgfcircle[stroke]{\pgfxy(145.00,20.00)}{0.60mm}
\pgfmoveto{\pgfxy(140.00,25.00)}\pgflineto{\pgfxy(145.00,20.00)}\pgfstroke
\pgfcircle[fill]{\pgfxy(140.00,25.00)}{0.60mm}
\pgfcircle[stroke]{\pgfxy(140.00,25.00)}{0.60mm}
\pgfmoveto{\pgfxy(135.00,30.00)}\pgflineto{\pgfxy(140.00,25.00)}\pgfstroke
\pgfputat{\pgfxy(145.00,22.00)}{\pgfbox[bottom,left]{\fontsize{11.38}{13.66}\selectfont $\bar{3}$}}
\pgfputat{\pgfxy(140.00,27.00)}{\pgfbox[bottom,left]{\fontsize{11.38}{13.66}\selectfont $\bar{2}$}}
\pgfputat{\pgfxy(135.00,32.00)}{\pgfbox[bottom,left]{\fontsize{11.38}{13.66}\selectfont $\bar{1}$}}
\pgfcircle[fill]{\pgfxy(135.00,30.00)}{0.60mm}
\pgfcircle[stroke]{\pgfxy(135.00,30.00)}{0.60mm}
\pgfcircle[fill]{\pgfxy(140.00,0.00)}{0.60mm}
\pgfcircle[stroke]{\pgfxy(140.00,0.00)}{0.60mm}
\pgfputat{\pgfxy(140.00,-5.00)}{\pgfbox[bottom,left]{\fontsize{11.38}{13.66}\selectfont \makebox[0pt]{$\bar{3}$}}}
\pgfmoveto{\pgfxy(135.00,5.00)}\pgflineto{\pgfxy(140.00,0.00)}\pgfstroke
\pgfcircle[fill]{\pgfxy(135.00,5.00)}{0.60mm}
\pgfcircle[stroke]{\pgfxy(135.00,5.00)}{0.60mm}
\pgfputat{\pgfxy(135.00,7.00)}{\pgfbox[bottom,left]{\fontsize{11.38}{13.66}\selectfont \makebox[0pt]{$\bar{2}$}}}
\pgfmoveto{\pgfxy(140.00,0.00)}\pgflineto{\pgfxy(145.00,5.00)}\pgfstroke
\pgfputat{\pgfxy(145.00,7.00)}{\pgfbox[bottom,left]{\fontsize{11.38}{13.66}\selectfont \makebox[0pt]{$\bar{1}$}}}
\pgfcircle[fill]{\pgfxy(145.00,5.00)}{0.60mm}
\pgfcircle[stroke]{\pgfxy(145.00,5.00)}{0.60mm}
\pgfcircle[fill]{\pgfxy(125.00,20.00)}{0.60mm}
\pgfcircle[stroke]{\pgfxy(125.00,20.00)}{0.60mm}
\pgfmoveto{\pgfxy(120.00,25.00)}\pgflineto{\pgfxy(125.00,20.00)}\pgfstroke
\pgfcircle[fill]{\pgfxy(120.00,25.00)}{0.60mm}
\pgfcircle[stroke]{\pgfxy(120.00,25.00)}{0.60mm}
\pgfmoveto{\pgfxy(115.00,30.00)}\pgflineto{\pgfxy(120.00,25.00)}\pgfstroke
\pgfputat{\pgfxy(125.00,22.00)}{\pgfbox[bottom,left]{\fontsize{11.38}{13.66}\selectfont $\bar{3}$}}
\pgfputat{\pgfxy(120.00,27.00)}{\pgfbox[bottom,left]{\fontsize{11.38}{13.66}\selectfont $\bar{2}$}}
\pgfputat{\pgfxy(115.00,32.00)}{\pgfbox[bottom,left]{\fontsize{11.38}{13.66}\selectfont 1}}
\pgfcircle[fill]{\pgfxy(115.00,30.00)}{0.60mm}
\pgfcircle[stroke]{\pgfxy(115.00,30.00)}{0.60mm}
\pgfcircle[fill]{\pgfxy(120.00,0.00)}{0.60mm}
\pgfcircle[stroke]{\pgfxy(120.00,0.00)}{0.60mm}
\pgfputat{\pgfxy(120.00,-5.00)}{\pgfbox[bottom,left]{\fontsize{11.38}{13.66}\selectfont \makebox[0pt]{$\bar{3}$}}}
\pgfmoveto{\pgfxy(115.00,5.00)}\pgflineto{\pgfxy(120.00,0.00)}\pgfstroke
\pgfcircle[fill]{\pgfxy(115.00,5.00)}{0.60mm}
\pgfcircle[stroke]{\pgfxy(115.00,5.00)}{0.60mm}
\pgfputat{\pgfxy(115.00,7.00)}{\pgfbox[bottom,left]{\fontsize{11.38}{13.66}\selectfont \makebox[0pt]{$\bar{2}$}}}
\pgfmoveto{\pgfxy(120.00,0.00)}\pgflineto{\pgfxy(125.00,5.00)}\pgfstroke
\pgfputat{\pgfxy(125.00,7.00)}{\pgfbox[bottom,left]{\fontsize{11.38}{13.66}\selectfont \makebox[0pt]{1}}}
\pgfcircle[fill]{\pgfxy(125.00,5.00)}{0.60mm}
\pgfcircle[stroke]{\pgfxy(125.00,5.00)}{0.60mm}
\pgfcircle[fill]{\pgfxy(105.00,20.00)}{0.60mm}
\pgfcircle[stroke]{\pgfxy(105.00,20.00)}{0.60mm}
\pgfmoveto{\pgfxy(100.00,25.00)}\pgflineto{\pgfxy(105.00,20.00)}\pgfstroke
\pgfcircle[fill]{\pgfxy(100.00,25.00)}{0.60mm}
\pgfcircle[stroke]{\pgfxy(100.00,25.00)}{0.60mm}
\pgfmoveto{\pgfxy(95.00,30.00)}\pgflineto{\pgfxy(100.00,25.00)}\pgfstroke
\pgfputat{\pgfxy(105.00,22.00)}{\pgfbox[bottom,left]{\fontsize{11.38}{13.66}\selectfont $\bar{3}$}}
\pgfputat{\pgfxy(100.00,27.00)}{\pgfbox[bottom,left]{\fontsize{11.38}{13.66}\selectfont $\bar{1}$}}
\pgfputat{\pgfxy(95.00,32.00)}{\pgfbox[bottom,left]{\fontsize{11.38}{13.66}\selectfont 2}}
\pgfcircle[fill]{\pgfxy(95.00,30.00)}{0.60mm}
\pgfcircle[stroke]{\pgfxy(95.00,30.00)}{0.60mm}
\pgfcircle[fill]{\pgfxy(100.00,0.00)}{0.60mm}
\pgfcircle[stroke]{\pgfxy(100.00,0.00)}{0.60mm}
\pgfputat{\pgfxy(100.00,-5.00)}{\pgfbox[bottom,left]{\fontsize{11.38}{13.66}\selectfont \makebox[0pt]{$\bar{3}$}}}
\pgfmoveto{\pgfxy(95.00,5.00)}\pgflineto{\pgfxy(100.00,0.00)}\pgfstroke
\pgfcircle[fill]{\pgfxy(95.00,5.00)}{0.60mm}
\pgfcircle[stroke]{\pgfxy(95.00,5.00)}{0.60mm}
\pgfputat{\pgfxy(95.00,7.00)}{\pgfbox[bottom,left]{\fontsize{11.38}{13.66}\selectfont \makebox[0pt]{$\bar{1}$}}}
\pgfmoveto{\pgfxy(100.00,0.00)}\pgflineto{\pgfxy(105.00,5.00)}\pgfstroke
\pgfputat{\pgfxy(105.00,7.00)}{\pgfbox[bottom,left]{\fontsize{11.38}{13.66}\selectfont \makebox[0pt]{2}}}
\pgfcircle[fill]{\pgfxy(105.00,5.00)}{0.60mm}
\pgfcircle[stroke]{\pgfxy(105.00,5.00)}{0.60mm}
\pgfcircle[fill]{\pgfxy(85.00,20.00)}{0.60mm}
\pgfcircle[stroke]{\pgfxy(85.00,20.00)}{0.60mm}
\pgfmoveto{\pgfxy(80.00,25.00)}\pgflineto{\pgfxy(85.00,20.00)}\pgfstroke
\pgfcircle[fill]{\pgfxy(80.00,25.00)}{0.60mm}
\pgfcircle[stroke]{\pgfxy(80.00,25.00)}{0.60mm}
\pgfmoveto{\pgfxy(75.00,30.00)}\pgflineto{\pgfxy(80.00,25.00)}\pgfstroke
\pgfputat{\pgfxy(85.00,22.00)}{\pgfbox[bottom,left]{\fontsize{11.38}{13.66}\selectfont $\bar{2}$}}
\pgfputat{\pgfxy(80.00,27.00)}{\pgfbox[bottom,left]{\fontsize{11.38}{13.66}\selectfont $\bar{1}$}}
\pgfputat{\pgfxy(75.00,32.00)}{\pgfbox[bottom,left]{\fontsize{11.38}{13.66}\selectfont 3}}
\pgfcircle[fill]{\pgfxy(75.00,30.00)}{0.60mm}
\pgfcircle[stroke]{\pgfxy(75.00,30.00)}{0.60mm}
\pgfcircle[fill]{\pgfxy(80.00,0.00)}{0.60mm}
\pgfcircle[stroke]{\pgfxy(80.00,0.00)}{0.60mm}
\pgfputat{\pgfxy(80.00,-5.00)}{\pgfbox[bottom,left]{\fontsize{11.38}{13.66}\selectfont \makebox[0pt]{$\bar{2}$}}}
\pgfmoveto{\pgfxy(75.00,5.00)}\pgflineto{\pgfxy(80.00,0.00)}\pgfstroke
\pgfcircle[fill]{\pgfxy(75.00,5.00)}{0.60mm}
\pgfcircle[stroke]{\pgfxy(75.00,5.00)}{0.60mm}
\pgfputat{\pgfxy(75.00,7.00)}{\pgfbox[bottom,left]{\fontsize{11.38}{13.66}\selectfont \makebox[0pt]{$\bar{1}$}}}
\pgfmoveto{\pgfxy(80.00,0.00)}\pgflineto{\pgfxy(85.00,5.00)}\pgfstroke
\pgfputat{\pgfxy(85.00,7.00)}{\pgfbox[bottom,left]{\fontsize{11.38}{13.66}\selectfont \makebox[0pt]{3}}}
\pgfcircle[fill]{\pgfxy(85.00,5.00)}{0.60mm}
\pgfcircle[stroke]{\pgfxy(85.00,5.00)}{0.60mm}
\pgfcircle[fill]{\pgfxy(65.00,20.00)}{0.60mm}
\pgfcircle[stroke]{\pgfxy(65.00,20.00)}{0.60mm}
\pgfmoveto{\pgfxy(60.00,25.00)}\pgflineto{\pgfxy(65.00,20.00)}\pgfstroke
\pgfcircle[fill]{\pgfxy(60.00,25.00)}{0.60mm}
\pgfcircle[stroke]{\pgfxy(60.00,25.00)}{0.60mm}
\pgfmoveto{\pgfxy(55.00,30.00)}\pgflineto{\pgfxy(60.00,25.00)}\pgfstroke
\pgfputat{\pgfxy(65.00,22.00)}{\pgfbox[bottom,left]{\fontsize{11.38}{13.66}\selectfont $\bar{3}$}}
\pgfputat{\pgfxy(60.00,27.00)}{\pgfbox[bottom,left]{\fontsize{11.38}{13.66}\selectfont 1}}
\pgfputat{\pgfxy(55.00,32.00)}{\pgfbox[bottom,left]{\fontsize{11.38}{13.66}\selectfont 2}}
\pgfcircle[fill]{\pgfxy(55.00,30.00)}{0.60mm}
\pgfcircle[stroke]{\pgfxy(55.00,30.00)}{0.60mm}
\pgfcircle[fill]{\pgfxy(60.00,0.00)}{0.60mm}
\pgfcircle[stroke]{\pgfxy(60.00,0.00)}{0.60mm}
\pgfputat{\pgfxy(60.00,-5.00)}{\pgfbox[bottom,left]{\fontsize{11.38}{13.66}\selectfont \makebox[0pt]{$\bar{3}$}}}
\pgfmoveto{\pgfxy(55.00,5.00)}\pgflineto{\pgfxy(60.00,0.00)}\pgfstroke
\pgfcircle[fill]{\pgfxy(55.00,5.00)}{0.60mm}
\pgfcircle[stroke]{\pgfxy(55.00,5.00)}{0.60mm}
\pgfputat{\pgfxy(55.00,7.00)}{\pgfbox[bottom,left]{\fontsize{11.38}{13.66}\selectfont \makebox[0pt]{1}}}
\pgfmoveto{\pgfxy(60.00,0.00)}\pgflineto{\pgfxy(65.00,5.00)}\pgfstroke
\pgfputat{\pgfxy(65.00,7.00)}{\pgfbox[bottom,left]{\fontsize{11.38}{13.66}\selectfont \makebox[0pt]{2}}}
\pgfcircle[fill]{\pgfxy(65.00,5.00)}{0.60mm}
\pgfcircle[stroke]{\pgfxy(65.00,5.00)}{0.60mm}
\pgfcircle[fill]{\pgfxy(45.00,20.00)}{0.60mm}
\pgfcircle[stroke]{\pgfxy(45.00,20.00)}{0.60mm}
\pgfmoveto{\pgfxy(40.00,25.00)}\pgflineto{\pgfxy(45.00,20.00)}\pgfstroke
\pgfcircle[fill]{\pgfxy(40.00,25.00)}{0.60mm}
\pgfcircle[stroke]{\pgfxy(40.00,25.00)}{0.60mm}
\pgfmoveto{\pgfxy(35.00,30.00)}\pgflineto{\pgfxy(40.00,25.00)}\pgfstroke
\pgfputat{\pgfxy(45.00,22.00)}{\pgfbox[bottom,left]{\fontsize{11.38}{13.66}\selectfont $\bar{2}$}}
\pgfputat{\pgfxy(40.00,27.00)}{\pgfbox[bottom,left]{\fontsize{11.38}{13.66}\selectfont 1}}
\pgfputat{\pgfxy(35.00,32.00)}{\pgfbox[bottom,left]{\fontsize{11.38}{13.66}\selectfont 3}}
\pgfcircle[fill]{\pgfxy(35.00,30.00)}{0.60mm}
\pgfcircle[stroke]{\pgfxy(35.00,30.00)}{0.60mm}
\pgfcircle[fill]{\pgfxy(40.00,0.00)}{0.60mm}
\pgfcircle[stroke]{\pgfxy(40.00,0.00)}{0.60mm}
\pgfputat{\pgfxy(40.00,-5.00)}{\pgfbox[bottom,left]{\fontsize{11.38}{13.66}\selectfont \makebox[0pt]{$\bar{2}$}}}
\pgfmoveto{\pgfxy(35.00,5.00)}\pgflineto{\pgfxy(40.00,0.00)}\pgfstroke
\pgfcircle[fill]{\pgfxy(35.00,5.00)}{0.60mm}
\pgfcircle[stroke]{\pgfxy(35.00,5.00)}{0.60mm}
\pgfputat{\pgfxy(35.00,7.00)}{\pgfbox[bottom,left]{\fontsize{11.38}{13.66}\selectfont \makebox[0pt]{1}}}
\pgfmoveto{\pgfxy(40.00,0.00)}\pgflineto{\pgfxy(45.00,5.00)}\pgfstroke
\pgfputat{\pgfxy(45.00,7.00)}{\pgfbox[bottom,left]{\fontsize{11.38}{13.66}\selectfont \makebox[0pt]{3}}}
\pgfcircle[fill]{\pgfxy(45.00,5.00)}{0.60mm}
\pgfcircle[stroke]{\pgfxy(45.00,5.00)}{0.60mm}
\pgfcircle[fill]{\pgfxy(25.00,20.00)}{0.60mm}
\pgfcircle[stroke]{\pgfxy(25.00,20.00)}{0.60mm}
\pgfmoveto{\pgfxy(20.00,25.00)}\pgflineto{\pgfxy(25.00,20.00)}\pgfstroke
\pgfcircle[fill]{\pgfxy(20.00,25.00)}{0.60mm}
\pgfcircle[stroke]{\pgfxy(20.00,25.00)}{0.60mm}
\pgfmoveto{\pgfxy(15.00,30.00)}\pgflineto{\pgfxy(20.00,25.00)}\pgfstroke
\pgfputat{\pgfxy(25.00,22.00)}{\pgfbox[bottom,left]{\fontsize{11.38}{13.66}\selectfont $\bar{1}$}}
\pgfputat{\pgfxy(20.00,27.00)}{\pgfbox[bottom,left]{\fontsize{11.38}{13.66}\selectfont 2}}
\pgfputat{\pgfxy(15.00,32.00)}{\pgfbox[bottom,left]{\fontsize{11.38}{13.66}\selectfont 3}}
\pgfcircle[fill]{\pgfxy(15.00,30.00)}{0.60mm}
\pgfcircle[stroke]{\pgfxy(15.00,30.00)}{0.60mm}
\pgfcircle[fill]{\pgfxy(20.00,0.00)}{0.60mm}
\pgfcircle[stroke]{\pgfxy(20.00,0.00)}{0.60mm}
\pgfputat{\pgfxy(20.00,-5.00)}{\pgfbox[bottom,left]{\fontsize{11.38}{13.66}\selectfont \makebox[0pt]{$\bar{1}$}}}
\pgfmoveto{\pgfxy(15.00,5.00)}\pgflineto{\pgfxy(20.00,0.00)}\pgfstroke
\pgfcircle[fill]{\pgfxy(15.00,5.00)}{0.60mm}
\pgfcircle[stroke]{\pgfxy(15.00,5.00)}{0.60mm}
\pgfputat{\pgfxy(15.00,7.00)}{\pgfbox[bottom,left]{\fontsize{11.38}{13.66}\selectfont \makebox[0pt]{2}}}
\pgfmoveto{\pgfxy(20.00,0.00)}\pgflineto{\pgfxy(25.00,5.00)}\pgfstroke
\pgfputat{\pgfxy(25.00,7.00)}{\pgfbox[bottom,left]{\fontsize{11.38}{13.66}\selectfont \makebox[0pt]{3}}}
\pgfcircle[fill]{\pgfxy(25.00,5.00)}{0.60mm}
\pgfcircle[stroke]{\pgfxy(25.00,5.00)}{0.60mm}
\end{pgfpicture}%
\caption{Sixteen type $B$ increasing 1-2 trees on $[3]$}
\end{figure}
Let $\TnB$ be the set of type $B$ increasing 1-2 trees on $n$ vertices and $\TnkB$ be the set of trees $T$ in $\TnB$ with leaf $k$ as the end of minimal path.
Clearly we have $\TnB = \bigcup_{\abs{k}>0} \TnkB$.

Our second aim is to show that these two refinements are new Arnold families.
Recall that the sequence $\DUnkB$ is an Arnold family for $1\leq \abs{k} \leq n$ as
$$\DUnkB := \set{\sigma\in \DUn^{(B)} : \sigma_1=k }.$$
\begin{thm}\label{thm:main2}
For all $1\le \abs{k} \le n$, there are two bijections
\begin{align}
\psi^B&: \DUnkB \to \TnkB, \label{eq:TreeB}\\
\omega^B&: \TnkB \to \DAnkB. \label{eq:AndreB}
\end{align}
Thus, for all $1\le \abs{k} \le n$,
\begin{align}
S_{n,k} &= \# \DAnkB = \# \TnkB.
\end{align}
In particular,  the  two sequences $\DAnkB$ and $\TnkB$ are Arnold families
for $1\leq \abs{k} \leq n$.
\end{thm}

Hetyei\cite[Definition 8]{Het96} defined  another class of 
signed Andr\'e permutations.
\begin{defn}[Hetyei]
A   \emph{signed Andr\'e permutation} is  a pair $(\varepsilon, \pi)$,
where $\pi$ is an Andr\'e permutation such that  $\varepsilon(i) = 1$
if  $\pi_i = \min \set{\pi_i, \pi_{i+1}, \dots, \pi_{n}}$.
\end{defn}
We write  $\DAnH$ (resp. $\DAnkH$) for the set of the signed Andr\'e permutations (resp. ending with entry $k$) in $\mathcal{B}_n$.
Some examples of $\DAH_{4,k}$ are shown in Table~\ref{table:DARSB}.
We have the following conjecture.
\begin{conj}
For all $1 \le k \le n$, we have
\begin{align*}
S_{n,k} = \# \DAH_{n+1,n+2-k}.
\end{align*}
\label{conj:main}
\end{conj}
Since the last entry of any permutation in the family $\DAnH$ is always positive,
even if Conjecture~\ref{conj:main} is true, it covers only the half of Table \ref{tab:Arnold}.
Now we define signed Simsun permutations corresponding to Heytei's signed Andr\'e permutations.
\begin{defn}
A permutation $\pi$ in $\mathcal{B}_n$ is a \emph{signed Simsun permutation}
if $\abs{\pi_1} \abs{\pi_2} \dots \abs{\pi_{n}}$ is a Simsun permutation
and $\pi_{i}>0$
for all $\abs{\pi_i} = \min \set{\abs{\pi_i}, \abs{\pi_{i+1}}, \dots, \abs{\pi_{n}}}$.
\end{defn}
Let $\RSnB$ be the set of signed Simsun permutations in $\mathcal{B}_n$
and $\RSnkB$ the set of signed Simsun permutations in $\RSnB$ with last entry $k$.
Some examples of $\RSB_{3,k}$ are shown in Table~\ref{table:DARSB}.

\begin{thm}
\label{thm:main3}
For positive integer $n\ge 1$, there is a bijection $\varphi^{(B)}: \DAnH \to \RSnoneB$
such that
\begin{align}
\Last(\sigma)-1=\Last(\varphi^{(B)}(\sigma))  \label{eq:SimsunB}
\end{align}
for all $\sigma \in \DAnH$.
In words, the mapping $\varphi^{(B)}$ is a bijection from $\DAnkH$ onto $\RSnkoneB$.
\end{thm}

\begin{table}[t]
\centering
\renewcommand{\arraystretch}{1.2}
\begin{tabular}{c||c|c|c|c}
$k$
& $\DUB_{3,k}$
& $\DAB_{3,k}$
& $\DAH_{4,5-k}$
& $\RSB_{3,4-k}$
\\ \hline
$1$
& $\set{1\bar{2}3, 1\bar{3}2, 1\bar{3}\bar{2}}$
& $\set{3\bar{2}1, \bar{3}\bar{2}1, 2\bar{3}1}$
& $\set{1234, 3124, \bar{3}124}$
& $\set{123, 213, \bar{2}13}$
\\
$2$
& $\set{213, 2\bar{1}3, 2\bar{3}1, 2\bar{3}\bar{1}}$
& $\set{312, \bar{3}12, 3\bar{1}2, \bar{3}\bar{1}2}$
& $\set{1423, 1\bar{4}23, 4123, \bar{4}123}$
& $\set{132, 1\bar{3}2, 312, \bar{3}12}$
\\
$3$
& $\set{312, 3\bar{1}2, 3\bar{2}1, 3\bar{2}\bar{1}}$
& $\set{\bar{2}13, \bar{2}\bar{1}3, 123, \bar{1}23}$
& $\set{3412, \bar{3}412, 3\bar{4}12, \bar{3}\bar{4}12}$
& $\set{231, \bar{2}31, 2\bar{3}1, \bar{2}\bar{3}1}$
\\
\end{tabular}\\
\vspace{1em}
\caption{The sets $\DUB_{3,k}$, $\DAB_{3,k}$, $\DAH_{4,5-k}$, and $\RSB_{3, 4-k}$ for $1\le k \le 3$}
\label{table:DARSB}
\end{table}

\begin{rmk}
Ehrenborg and Readdy \cite[Section 7]{ER98} gave a different definition of
\emph{signed Simsun permutation} as follows:
A signed permutation $\sigma$ of length $n$ is a Simsun permutation if $\sigma_{[k]}$ have no double descents for all $1\le k\le n$, where $\sigma_{[k]}$ is obtained by removing the $(n-k)$ entries $\pm (k+1), \dots, \pm n$ from $\sigma$.
Beacuse all eight signed permutations of length $2$ 
$$
12,~21,
~\bar{1}2, ~ 2\bar{1},
~1\bar{2}, ~ \bar{2}1,
~\bar{1}\bar{2}, ~ \bar{2}\bar{1}
$$
are Simsun permutations,
we note that it is not an Arnold family.
\end{rmk}








\section{Proof of Theorems~\ref{main2}, \ref{main1}, and \ref{thm:main3}}
\label{sec:proof1}
First of all, we prove Theorem~\ref{main2},
in order to show that $(\DAnk)_{1\le k \le n}$ is a Entringer family,
that is, $E_{n,k} = \# \DAnk$.
We construct a bijection $\omega$ between $\Tnk$ and $\DAnk$ in Section~\ref{sec:omega}.
Hence the map $\omega \circ \psi$ is a bijection from the set $\DUnk$ of (down-up) alternating permutations with first entry $k$ to the set $\DAnk$ of Andr\'e permutations with last entry $k$.
Also, in order to show that $(\RSnkone)_{1\le k \le n}$ is a Entringer family, that is $E_{n,k} = \# \RSnkone$,
we construct a bijection between $\DAnk$ and $\RSnkone$ in Section~\ref{sec:varphi} and then two sets $\DAnk$ and $\RSnkone$ have the same cardinality.



\subsection{Proof of Theorem~\ref{main2}: Bijection $\omega : \Tnk \to \DAnk$.}
\label{sec:omega}
Given $T \in \Tnk$,
write down the word $\sigma$ of the vertices of the tree $T$ in \emph{inorder}, namely,
for any vertex $v$ in $T$, the left child of $v$ and its descendants precede the vertex $v$ and the vertex $v$ precede the right child of $v$ and its descendants.
Since $T$ is an increasing tree, we can recover $T$ from $\sigma$ by finding minimum in subwords in $\sigma$ successively.
The word $\sigma$ has no double ascents because no vertex in $T$ has only a right child.
The leaf $p(T)$ of the minimal path of $T$ is $\sigma_1$ and the parent of $\sigma_1$ is $\sigma_2$, so $\sigma$ starts with a decent, that is, $\sigma_1 > \sigma_2$.
Similarly, since the subgraph of $T$ consisting of $1, \dots, k$, for any $k$, is also well-defined an increasing 1-2 subtree, the subwords of $\sigma$ consisting of $1, \dots, k$ has also no double ascents and starts with a decent.
Thus the word $\sigma^R$, which is the reverse word of $\sigma$, is an Andr\'e permutation of $n$ and let $\omega(T) = \sigma^R$.

For example, if the tree $T$ is given by the following figure with a corresponding word $\sigma$,
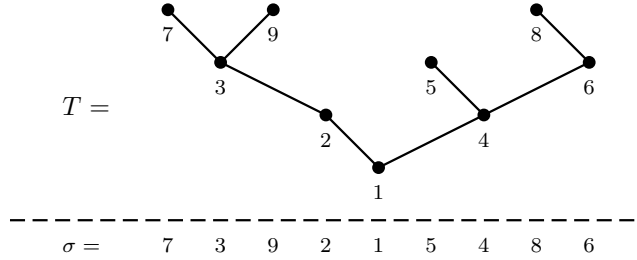
$$
\centering
\begin{pgfpicture}{-23.00mm}{-20.80mm}{65.00mm}{16.70mm}
\pgfsetxvec{\pgfpoint{0.70mm}{0mm}}
\pgfsetyvec{\pgfpoint{0mm}{0.70mm}}
\color[rgb]{0,0,0}\pgfsetlinewidth{0.30mm}\pgfsetdash{}{0mm}
\pgfcircle[fill]{\pgfxy(40.00,-10.00)}{0.70mm}
\pgfcircle[stroke]{\pgfxy(40.00,-10.00)}{0.70mm}
\pgfputat{\pgfxy(-20.00,0.00)}{\pgfbox[bottom,left]{\fontsize{9.96}{11.95}\selectfont $T=$}}
\pgfcircle[fill]{\pgfxy(60.00,0.00)}{0.70mm}
\pgfcircle[stroke]{\pgfxy(60.00,0.00)}{0.70mm}
\pgfcircle[fill]{\pgfxy(70.00,20.00)}{0.70mm}
\pgfcircle[stroke]{\pgfxy(70.00,20.00)}{0.70mm}
\pgfcircle[fill]{\pgfxy(80.00,10.00)}{0.70mm}
\pgfcircle[stroke]{\pgfxy(80.00,10.00)}{0.70mm}
\pgfcircle[fill]{\pgfxy(30.00,0.00)}{0.70mm}
\pgfcircle[stroke]{\pgfxy(30.00,0.00)}{0.70mm}
\pgfcircle[fill]{\pgfxy(50.00,10.00)}{0.70mm}
\pgfcircle[stroke]{\pgfxy(50.00,10.00)}{0.70mm}
\pgfcircle[fill]{\pgfxy(10.00,10.00)}{0.70mm}
\pgfcircle[stroke]{\pgfxy(10.00,10.00)}{0.70mm}
\pgfcircle[fill]{\pgfxy(0.00,20.00)}{0.70mm}
\pgfcircle[stroke]{\pgfxy(0.00,20.00)}{0.70mm}
\pgfcircle[fill]{\pgfxy(20.00,20.00)}{0.70mm}
\pgfcircle[stroke]{\pgfxy(20.00,20.00)}{0.70mm}
\pgfmoveto{\pgfxy(40.00,-10.00)}\pgflineto{\pgfxy(30.00,0.00)}\pgfstroke
\pgfmoveto{\pgfxy(30.00,0.00)}\pgflineto{\pgfxy(10.00,10.00)}\pgfstroke
\pgfmoveto{\pgfxy(10.00,10.00)}\pgflineto{\pgfxy(0.00,20.00)}\pgfstroke
\pgfmoveto{\pgfxy(10.00,10.00)}\pgflineto{\pgfxy(20.00,20.00)}\pgfstroke
\pgfmoveto{\pgfxy(40.00,-10.00)}\pgflineto{\pgfxy(60.00,0.00)}\pgfstroke
\pgfmoveto{\pgfxy(60.00,0.00)}\pgflineto{\pgfxy(50.00,10.00)}\pgfstroke
\pgfmoveto{\pgfxy(60.00,0.00)}\pgflineto{\pgfxy(80.00,10.00)}\pgfstroke
\pgfmoveto{\pgfxy(80.00,10.00)}\pgflineto{\pgfxy(70.00,20.00)}\pgfstroke
\pgfputat{\pgfxy(40.00,-16.00)}{\pgfbox[bottom,left]{\fontsize{7.97}{9.56}\selectfont \makebox[0pt]{$1$}}}
\pgfputat{\pgfxy(30.00,-6.00)}{\pgfbox[bottom,left]{\fontsize{7.97}{9.56}\selectfont \makebox[0pt]{$2$}}}
\pgfputat{\pgfxy(10.00,4.00)}{\pgfbox[bottom,left]{\fontsize{7.97}{9.56}\selectfont \makebox[0pt]{$3$}}}
\pgfputat{\pgfxy(60.00,-6.00)}{\pgfbox[bottom,left]{\fontsize{7.97}{9.56}\selectfont \makebox[0pt]{$4$}}}
\pgfputat{\pgfxy(50.00,4.00)}{\pgfbox[bottom,left]{\fontsize{7.97}{9.56}\selectfont \makebox[0pt]{$5$}}}
\pgfputat{\pgfxy(80.00,4.00)}{\pgfbox[bottom,left]{\fontsize{7.97}{9.56}\selectfont \makebox[0pt]{$6$}}}
\pgfputat{\pgfxy(0.00,14.00)}{\pgfbox[bottom,left]{\fontsize{7.97}{9.56}\selectfont \makebox[0pt]{$7$}}}
\pgfputat{\pgfxy(70.00,14.00)}{\pgfbox[bottom,left]{\fontsize{7.97}{9.56}\selectfont \makebox[0pt]{$8$}}}
\pgfputat{\pgfxy(20.00,14.00)}{\pgfbox[bottom,left]{\fontsize{7.97}{9.56}\selectfont \makebox[0pt]{$9$}}}
\pgfputat{\pgfxy(-20.00,-26.00)}{\pgfbox[bottom,left]{\fontsize{7.97}{9.56}\selectfont $\sigma=$}}
\pgfputat{\pgfxy(40.00,-26.00)}{\pgfbox[bottom,left]{\fontsize{7.97}{9.56}\selectfont \makebox[0pt]{$1$}}}
\pgfputat{\pgfxy(30.00,-26.00)}{\pgfbox[bottom,left]{\fontsize{7.97}{9.56}\selectfont \makebox[0pt]{$2$}}}
\pgfputat{\pgfxy(10.00,-26.00)}{\pgfbox[bottom,left]{\fontsize{7.97}{9.56}\selectfont \makebox[0pt]{$3$}}}
\pgfputat{\pgfxy(60.00,-26.00)}{\pgfbox[bottom,left]{\fontsize{7.97}{9.56}\selectfont \makebox[0pt]{$4$}}}
\pgfputat{\pgfxy(50.00,-26.00)}{\pgfbox[bottom,left]{\fontsize{7.97}{9.56}\selectfont \makebox[0pt]{$5$}}}
\pgfputat{\pgfxy(80.00,-26.00)}{\pgfbox[bottom,left]{\fontsize{7.97}{9.56}\selectfont \makebox[0pt]{$6$}}}
\pgfputat{\pgfxy(0.00,-26.00)}{\pgfbox[bottom,left]{\fontsize{7.97}{9.56}\selectfont \makebox[0pt]{$7$}}}
\pgfputat{\pgfxy(70.00,-26.00)}{\pgfbox[bottom,left]{\fontsize{7.97}{9.56}\selectfont \makebox[0pt]{$8$}}}
\pgfputat{\pgfxy(20.00,-26.00)}{\pgfbox[bottom,left]{\fontsize{7.97}{9.56}\selectfont \makebox[0pt]{$9$}}}
\pgfsetdash{{2.00mm}{1.00mm}}{0mm}\pgfmoveto{\pgfxy(-30.00,-20.00)}\pgflineto{\pgfxy(90.00,-20.00)}\pgfstroke
\end{pgfpicture}%
$$
then $\omega(T) = \sigma^R = 684512937$ as reading reversely vertices of $T$ in inorder.

%

\subsection{Proof of Theorem~\ref{main1}: Bijection $\varphi : \DAnk \to \RS_{n-1,k-1}$.}
\label{sec:varphi}
Given $\sigma:=\sigma_1\ldots \sigma_n \in \DAnk$, let
$i_1<\ldots< i_\ell$ be the positions of the right-to-left minima of $\sigma$.
Clearly
$\sigma_{i_1}=1$,  $i_{\ell-1}=n-1$ and  $i_\ell=n$ with $\sigma_{n}=k$.
Let $\varphi(\sigma)=\pi$, where $\pi=\pi_1 \pi_2 \dots \pi_{n-1}$ is defined by
\begin{align}
\pi_i=\begin{cases}
\sigma_i-1& \text{if}\quad i\notin\{i_1, \ldots, i_\ell\}, \\
\sigma_{i_k}-1& \text{if}\quad  i=i_{k-1} \quad \text{for}\; k=2, \ldots, \ell.
\end{cases}
\label{eq:move}
\end{align}
We show that $\pi \in \RS_{n-1,k-1}$.
Suppose $\pi_{[i]}$ has a double descent for some $1\le i \le n$.
There exists a triple $(a,b,c)$ such that
$1\le a < b < c \le n$, $i \ge \pi_a > \pi_b > \pi_c$, and
$\pi_{a+1}, \ldots, \pi_{b-1}$, $\pi_{b+1}, \ldots, \pi_{c-1}$ are greater than $i$,
which yields $\pi_c = \min\set{\pi_j : a\le j \le c}$.
Then $\pi_{c}$ could be a right-to-left minimum in $\pi$ and the others $\pi_a, \pi_{a+1}, \dots, \pi_{c-1}$ are not in $\pi$. As $\varphi(\sigma) = \pi$, we have
$$
\sigma_a = \pi_a+1,~ \sigma_{a+1} = \pi_{a+1}+1,~ \ldots, \sigma_{c-1} = \pi_{c-1}+1
\text{, and }
\sigma_{c} \le \pi_{c}+1.
$$
Hence a triple $(\sigma_{a}, \sigma_{b}, \sigma_{c})$ is a double descent in $\sigma_{[i]}$,
which contradicts that $\sigma$ is an Andr\'e permutation.
As this procedure is clearly reversible, the mapping $\varphi$ is a bijection.

Consider the running example $\sigma=684512937$. The right-to-left minimums of $\sigma$ are $1$, $2$, $3$, $7$.
So after removing $1$ from $\sigma$,
the entries $2$, $3$, $7$ are moved to the positions of $1$, $2$, $3$, respectively,
and we get the permutation $\hat\pi = 68452397$,
and then $\varphi(\sigma) = \pi = 57341286$,
which is a Simsun permutation of length $8$ with last entry $6$.


\begin{rmk}
Considering the bijection $\psi$ in Theorem~\ref{thm:GSZ11},
the map $\varphi \circ \omega \circ \psi$ is a bijection from the set of (down-up) alternating permutations of length $n$ with first entry $k$ to the set of Simsun permutations of length $n-1$ with last entry $k-1$. Namely we have the diagram in Figure~\ref{fig:diagram1}.
For example,
if $\tau = 739154826 \in \DU_{9,7}$, then 
\begin{align*}
\psi(\tau) &= T  \in T_{9,7}, \\
\omega(T) &= \sigma = 684512937 \in \DA_{9,7}, \\
\varphi(\sigma) &= \pi = 57341286 \in \RS_{8,6}, 
\end{align*}
where $T$ is the increasing 1-2 tree given in Section~\ref{sec:omega}.

    



\end{rmk}

\subsection{Proof of Theorem~\ref{thm:main3}}
One can extend the above mapping $\varphi$ defined on $\DAn$ to a mapping $\varphi^{(B)}$ on $\DAnH$.
It is also bijective between $\DAnkH$ and $\RSnkoneB$,
but the description of $\varphi^{(B)}$ and a proof of bijectivity are omitted,
because it is very similar in Section~\ref{sec:varphi}.

\begin{rmk}
This bijection preserves the $cd$-indices between Andr\'e permutations and Simsun permutations.
The \emph{variation} of a permutation $\pi=\pi_1 \dots \pi_n$ is given by $\a\b$-monomial $u_1\dots u_{n-1}$
such that $u_i=\a$ if $\pi_i < \pi_{i+1}$ and $u_i=\b$ if $\pi_i > \pi_{i+1}$.
The \emph{reduced variation} of Andr\'e permuation is defined by replacing each $\b\a$ with $\d$ and then replacing each remaining $\a$ by $\c$.
For example, the variation and reduced variation of Andr\'e permutation $\sigma=684512937$ is
$$\texttt{ababaaba} = \texttt{cddcd}.$$

For the cd-index of a Simsun permutation $\sigma$,
we consider \emph{augmented} Simsun permutation by adding $\sigma(0)=0$.
Here, the \emph{reduced variation} of augmented Simsun permuations is defined by replacing each $\a\b$ with $\d$ and then replacing each remaining $\a$ by $\c$.
So the variation and reduced variation of augmented Simsun permutation $057341286$ by adding $0$ to $\varphi(\sigma) = 57341286$ is
$$\texttt{aababaab} = \texttt{cddcd}.$$

\end{rmk}

\section{Proof of Theorem~\ref{thm:main2}}
\label{sec:proof2}
Given a $\sigma \in \DUnkB$,
there is a unique order-preserving map $\pi_\sigma$, say just $\pi$, from $\set{\sigma_1, \dots, \sigma_n}$ to $[n]$.
In other words, $\pi$ replaces the $i$-th smallest entry in $\sigma$ by $i$.
The permutation $\tau = \pi \sigma$ belongs to $\Ank$ and $\psi(\tau) = \psi(\pi \sigma)$ in $\Tnk$. Then $\pi^{-1}(\psi(\tau))$ means the tree with vertex labelings $\set{\sigma_1, \dots, \sigma_n}$ instead of $[n]$ and it should belong to $\TnkB$.
So we construct the bijection $\psi^{B}$ from $\DUnkB$ to $\TnkB$ by
$$\psi^{B}(\sigma) = \pi^{-1}(\psi(\pi \sigma))$$
through the unique order-preserving map $\pi$.
Hence, it yields \eqref{eq:TreeB}.

For example, in the case of $\sigma = 6\bar{3}9\bar{8}2\bar{1}7\bar{4}5$,
the order-preserving map $\pi_\sigma$ is 
$$\pi = {\bar{8}\bar{4}\bar{3}\bar{1}25679 \choose 123456789}.$$
So we have $\tau = \pi \sigma = 739154826$ and $\psi(\tau) = \psi(\pi \sigma)$ and $\psi^{B}(\sigma) = \pi^{-1}(\psi(\pi \sigma))$ are illustrated as
$$
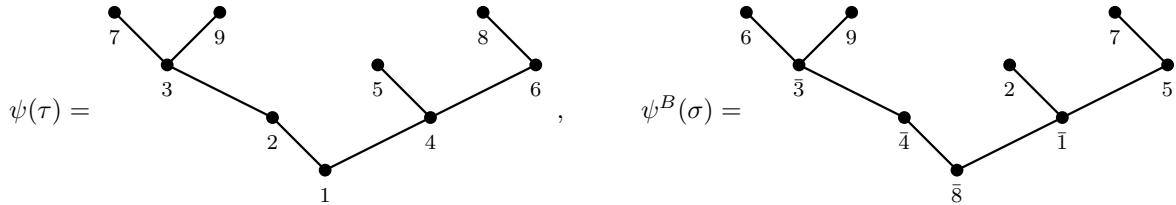
\centering
\begin{pgfpicture}{-16.00mm}{-13.80mm}{142.70mm}{16.70mm}
\pgfsetxvec{\pgfpoint{0.70mm}{0mm}}
\pgfsetyvec{\pgfpoint{0mm}{0.70mm}}
\color[rgb]{0,0,0}\pgfsetlinewidth{0.30mm}\pgfsetdash{}{0mm}
\pgfputat{\pgfxy(20.00,14.00)}{\pgfbox[bottom,left]{\fontsize{7.97}{9.56}\selectfont \makebox[0pt]{$9$}}}
\pgfputat{\pgfxy(0.00,14.00)}{\pgfbox[bottom,left]{\fontsize{7.97}{9.56}\selectfont \makebox[0pt]{$7$}}}
\pgfputat{\pgfxy(80.00,4.00)}{\pgfbox[bottom,left]{\fontsize{7.97}{9.56}\selectfont \makebox[0pt]{$6$}}}
\pgfputat{\pgfxy(50.00,4.00)}{\pgfbox[bottom,left]{\fontsize{7.97}{9.56}\selectfont \makebox[0pt]{$5$}}}
\pgfputat{\pgfxy(60.00,-6.00)}{\pgfbox[bottom,left]{\fontsize{7.97}{9.56}\selectfont \makebox[0pt]{$4$}}}
\pgfputat{\pgfxy(10.00,4.00)}{\pgfbox[bottom,left]{\fontsize{7.97}{9.56}\selectfont \makebox[0pt]{$3$}}}
\pgfputat{\pgfxy(30.00,-6.00)}{\pgfbox[bottom,left]{\fontsize{7.97}{9.56}\selectfont \makebox[0pt]{$2$}}}
\pgfputat{\pgfxy(40.00,-16.00)}{\pgfbox[bottom,left]{\fontsize{7.97}{9.56}\selectfont \makebox[0pt]{$1$}}}
\pgfmoveto{\pgfxy(80.00,10.00)}\pgflineto{\pgfxy(70.00,20.00)}\pgfstroke
\pgfmoveto{\pgfxy(60.00,0.00)}\pgflineto{\pgfxy(80.00,10.00)}\pgfstroke
\pgfmoveto{\pgfxy(60.00,0.00)}\pgflineto{\pgfxy(50.00,10.00)}\pgfstroke
\pgfmoveto{\pgfxy(40.00,-10.00)}\pgflineto{\pgfxy(60.00,0.00)}\pgfstroke
\pgfmoveto{\pgfxy(10.00,10.00)}\pgflineto{\pgfxy(20.00,20.00)}\pgfstroke
\pgfmoveto{\pgfxy(10.00,10.00)}\pgflineto{\pgfxy(0.00,20.00)}\pgfstroke
\pgfmoveto{\pgfxy(30.00,0.00)}\pgflineto{\pgfxy(10.00,10.00)}\pgfstroke
\pgfmoveto{\pgfxy(40.00,-10.00)}\pgflineto{\pgfxy(30.00,0.00)}\pgfstroke
\pgfcircle[fill]{\pgfxy(20.00,20.00)}{0.70mm}
\pgfcircle[stroke]{\pgfxy(20.00,20.00)}{0.70mm}
\pgfcircle[fill]{\pgfxy(0.00,20.00)}{0.70mm}
\pgfcircle[stroke]{\pgfxy(0.00,20.00)}{0.70mm}
\pgfcircle[fill]{\pgfxy(10.00,10.00)}{0.70mm}
\pgfcircle[stroke]{\pgfxy(10.00,10.00)}{0.70mm}
\pgfcircle[fill]{\pgfxy(50.00,10.00)}{0.70mm}
\pgfcircle[stroke]{\pgfxy(50.00,10.00)}{0.70mm}
\pgfcircle[fill]{\pgfxy(30.00,0.00)}{0.70mm}
\pgfcircle[stroke]{\pgfxy(30.00,0.00)}{0.70mm}
\pgfcircle[fill]{\pgfxy(80.00,10.00)}{0.70mm}
\pgfcircle[stroke]{\pgfxy(80.00,10.00)}{0.70mm}
\pgfcircle[fill]{\pgfxy(70.00,20.00)}{0.70mm}
\pgfcircle[stroke]{\pgfxy(70.00,20.00)}{0.70mm}
\pgfcircle[fill]{\pgfxy(60.00,0.00)}{0.70mm}
\pgfcircle[stroke]{\pgfxy(60.00,0.00)}{0.70mm}
\pgfputat{\pgfxy(-20.00,0.00)}{\pgfbox[bottom,left]{\fontsize{9.96}{11.95}\selectfont $\psi(\tau)=$}}
\pgfcircle[fill]{\pgfxy(40.00,-10.00)}{0.70mm}
\pgfcircle[stroke]{\pgfxy(40.00,-10.00)}{0.70mm}
\pgfputat{\pgfxy(70.00,14.00)}{\pgfbox[bottom,left]{\fontsize{7.97}{9.56}\selectfont \makebox[0pt]{$8$}}}
\pgfputat{\pgfxy(140.00,14.00)}{\pgfbox[bottom,left]{\fontsize{7.97}{9.56}\selectfont \makebox[0pt]{$9$}}}
\pgfputat{\pgfxy(120.00,14.00)}{\pgfbox[bottom,left]{\fontsize{7.97}{9.56}\selectfont \makebox[0pt]{$6$}}}
\pgfputat{\pgfxy(200.00,4.00)}{\pgfbox[bottom,left]{\fontsize{7.97}{9.56}\selectfont \makebox[0pt]{$5$}}}
\pgfputat{\pgfxy(170.00,4.00)}{\pgfbox[bottom,left]{\fontsize{7.97}{9.56}\selectfont \makebox[0pt]{$2$}}}
\pgfputat{\pgfxy(180.00,-6.00)}{\pgfbox[bottom,left]{\fontsize{7.97}{9.56}\selectfont \makebox[0pt]{$\bar{1}$}}}
\pgfputat{\pgfxy(130.00,4.00)}{\pgfbox[bottom,left]{\fontsize{7.97}{9.56}\selectfont \makebox[0pt]{$\bar{3}$}}}
\pgfputat{\pgfxy(150.00,-6.00)}{\pgfbox[bottom,left]{\fontsize{7.97}{9.56}\selectfont \makebox[0pt]{$\bar{4}$}}}
\pgfputat{\pgfxy(160.00,-16.00)}{\pgfbox[bottom,left]{\fontsize{7.97}{9.56}\selectfont \makebox[0pt]{$\bar{8}$}}}
\pgfmoveto{\pgfxy(200.00,10.00)}\pgflineto{\pgfxy(190.00,20.00)}\pgfstroke
\pgfmoveto{\pgfxy(180.00,0.00)}\pgflineto{\pgfxy(200.00,10.00)}\pgfstroke
\pgfmoveto{\pgfxy(180.00,0.00)}\pgflineto{\pgfxy(170.00,10.00)}\pgfstroke
\pgfmoveto{\pgfxy(160.00,-10.00)}\pgflineto{\pgfxy(180.00,0.00)}\pgfstroke
\pgfmoveto{\pgfxy(130.00,10.00)}\pgflineto{\pgfxy(140.00,20.00)}\pgfstroke
\pgfmoveto{\pgfxy(130.00,10.00)}\pgflineto{\pgfxy(120.00,20.00)}\pgfstroke
\pgfmoveto{\pgfxy(150.00,0.00)}\pgflineto{\pgfxy(130.00,10.00)}\pgfstroke
\pgfmoveto{\pgfxy(160.00,-10.00)}\pgflineto{\pgfxy(150.00,0.00)}\pgfstroke
\pgfcircle[fill]{\pgfxy(140.00,20.00)}{0.70mm}
\pgfcircle[stroke]{\pgfxy(140.00,20.00)}{0.70mm}
\pgfcircle[fill]{\pgfxy(120.00,20.00)}{0.70mm}
\pgfcircle[stroke]{\pgfxy(120.00,20.00)}{0.70mm}
\pgfcircle[fill]{\pgfxy(130.00,10.00)}{0.70mm}
\pgfcircle[stroke]{\pgfxy(130.00,10.00)}{0.70mm}
\pgfcircle[fill]{\pgfxy(170.00,10.00)}{0.70mm}
\pgfcircle[stroke]{\pgfxy(170.00,10.00)}{0.70mm}
\pgfcircle[fill]{\pgfxy(150.00,0.00)}{0.70mm}
\pgfcircle[stroke]{\pgfxy(150.00,0.00)}{0.70mm}
\pgfcircle[fill]{\pgfxy(200.00,10.00)}{0.70mm}
\pgfcircle[stroke]{\pgfxy(200.00,10.00)}{0.70mm}
\pgfcircle[fill]{\pgfxy(190.00,20.00)}{0.70mm}
\pgfcircle[stroke]{\pgfxy(190.00,20.00)}{0.70mm}
\pgfcircle[fill]{\pgfxy(180.00,0.00)}{0.70mm}
\pgfcircle[stroke]{\pgfxy(180.00,0.00)}{0.70mm}
\pgfputat{\pgfxy(100.00,0.00)}{\pgfbox[bottom,left]{\fontsize{9.96}{11.95}\selectfont $\psi^B(\sigma)=$}}
\pgfcircle[fill]{\pgfxy(160.00,-10.00)}{0.70mm}
\pgfcircle[stroke]{\pgfxy(160.00,-10.00)}{0.70mm}
\pgfputat{\pgfxy(190.00,14.00)}{\pgfbox[bottom,left]{\fontsize{7.97}{9.56}\selectfont \makebox[0pt]{$7$}}}
\pgfputat{\pgfxy(84.00,0.00)}{\pgfbox[bottom,left]{\fontsize{9.96}{11.95}\selectfont ,}}
\end{pgfpicture}%
$$

In Subsection \ref{sec:omega}, we define the bijection $\omega$ from $\Tnk$ to $\DAnk$. Given a tree $T$ in $\TnkB$, there is a unique order-preserving map $\pi_T$, say just $\pi$, from $V(T)$ to $[n]$.
In other words, $\pi$ replaces the $i$-th smallest $V(T)$ by $i$.
After relabeling on vertices of $T$ by $\pi$, we obtain the tree $\pi(T)$ which belongs to $\Tnk$ and $\omega(\pi(T))$ is in $\DAnk$. Then $\pi^{-1}(\omega(\pi(T)))$ should belong to $\DAnkB$.
So we construct the bijection $\omega^{B}$ from $\TnkB$ to $\DAnkB$ by
$$\omega^{B}(T) = \pi^{-1}(\omega(\pi(T)))$$
through the unique order-preserving map $\pi$.
Such the map $\omega^{B}$ can be described simply, as same as $\omega$, by reading reversely vertices of $T$ in inorder.
Hence, it yields \eqref{eq:AndreB}.

For example, in the case of $T$ illustrated as
$$
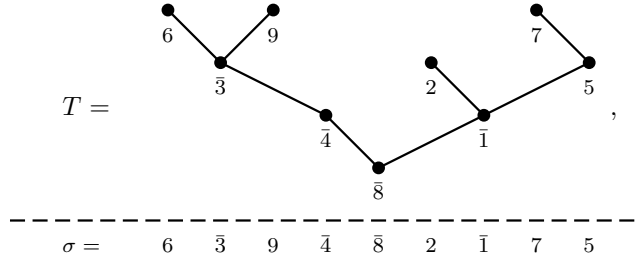
\centering
\begin{pgfpicture}{-23.00mm}{-20.80mm}{65.00mm}{16.70mm}
\pgfsetxvec{\pgfpoint{0.70mm}{0mm}}
\pgfsetyvec{\pgfpoint{0mm}{0.70mm}}
\color[rgb]{0,0,0}\pgfsetlinewidth{0.30mm}\pgfsetdash{}{0mm}
\pgfputat{\pgfxy(-20.00,0.00)}{\pgfbox[bottom,left]{\fontsize{9.96}{11.95}\selectfont $T=$}}
\pgfputat{\pgfxy(-20.00,-26.00)}{\pgfbox[bottom,left]{\fontsize{7.97}{9.56}\selectfont $\sigma=$}}
\pgfsetdash{{2.00mm}{1.00mm}}{0mm}\pgfmoveto{\pgfxy(-30.00,-20.00)}\pgflineto{\pgfxy(90.00,-20.00)}\pgfstroke
\pgfputat{\pgfxy(20.00,14.00)}{\pgfbox[bottom,left]{\fontsize{7.97}{9.56}\selectfont \makebox[0pt]{$9$}}}
\pgfputat{\pgfxy(0.00,14.00)}{\pgfbox[bottom,left]{\fontsize{7.97}{9.56}\selectfont \makebox[0pt]{$6$}}}
\pgfputat{\pgfxy(80.00,4.00)}{\pgfbox[bottom,left]{\fontsize{7.97}{9.56}\selectfont \makebox[0pt]{$5$}}}
\pgfputat{\pgfxy(50.00,4.00)}{\pgfbox[bottom,left]{\fontsize{7.97}{9.56}\selectfont \makebox[0pt]{$2$}}}
\pgfputat{\pgfxy(60.00,-6.00)}{\pgfbox[bottom,left]{\fontsize{7.97}{9.56}\selectfont \makebox[0pt]{$\bar{1}$}}}
\pgfputat{\pgfxy(10.00,4.00)}{\pgfbox[bottom,left]{\fontsize{7.97}{9.56}\selectfont \makebox[0pt]{$\bar{3}$}}}
\pgfputat{\pgfxy(30.00,-6.00)}{\pgfbox[bottom,left]{\fontsize{7.97}{9.56}\selectfont \makebox[0pt]{$\bar{4}$}}}
\pgfputat{\pgfxy(40.00,-16.00)}{\pgfbox[bottom,left]{\fontsize{7.97}{9.56}\selectfont \makebox[0pt]{$\bar{8}$}}}
\pgfsetdash{}{0mm}\pgfmoveto{\pgfxy(80.00,10.00)}\pgflineto{\pgfxy(70.00,20.00)}\pgfstroke
\pgfmoveto{\pgfxy(60.00,0.00)}\pgflineto{\pgfxy(80.00,10.00)}\pgfstroke
\pgfmoveto{\pgfxy(60.00,0.00)}\pgflineto{\pgfxy(50.00,10.00)}\pgfstroke
\pgfmoveto{\pgfxy(40.00,-10.00)}\pgflineto{\pgfxy(60.00,0.00)}\pgfstroke
\pgfmoveto{\pgfxy(10.00,10.00)}\pgflineto{\pgfxy(20.00,20.00)}\pgfstroke
\pgfmoveto{\pgfxy(10.00,10.00)}\pgflineto{\pgfxy(0.00,20.00)}\pgfstroke
\pgfmoveto{\pgfxy(30.00,0.00)}\pgflineto{\pgfxy(10.00,10.00)}\pgfstroke
\pgfmoveto{\pgfxy(40.00,-10.00)}\pgflineto{\pgfxy(30.00,0.00)}\pgfstroke
\pgfcircle[fill]{\pgfxy(20.00,20.00)}{0.70mm}
\pgfcircle[stroke]{\pgfxy(20.00,20.00)}{0.70mm}
\pgfcircle[fill]{\pgfxy(0.00,20.00)}{0.70mm}
\pgfcircle[stroke]{\pgfxy(0.00,20.00)}{0.70mm}
\pgfcircle[fill]{\pgfxy(10.00,10.00)}{0.70mm}
\pgfcircle[stroke]{\pgfxy(10.00,10.00)}{0.70mm}
\pgfcircle[fill]{\pgfxy(50.00,10.00)}{0.70mm}
\pgfcircle[stroke]{\pgfxy(50.00,10.00)}{0.70mm}
\pgfcircle[fill]{\pgfxy(30.00,0.00)}{0.70mm}
\pgfcircle[stroke]{\pgfxy(30.00,0.00)}{0.70mm}
\pgfcircle[fill]{\pgfxy(80.00,10.00)}{0.70mm}
\pgfcircle[stroke]{\pgfxy(80.00,10.00)}{0.70mm}
\pgfcircle[fill]{\pgfxy(70.00,20.00)}{0.70mm}
\pgfcircle[stroke]{\pgfxy(70.00,20.00)}{0.70mm}
\pgfcircle[fill]{\pgfxy(60.00,0.00)}{0.70mm}
\pgfcircle[stroke]{\pgfxy(60.00,0.00)}{0.70mm}
\pgfcircle[fill]{\pgfxy(40.00,-10.00)}{0.70mm}
\pgfcircle[stroke]{\pgfxy(40.00,-10.00)}{0.70mm}
\pgfputat{\pgfxy(70.00,14.00)}{\pgfbox[bottom,left]{\fontsize{7.97}{9.56}\selectfont \makebox[0pt]{$7$}}}
\pgfputat{\pgfxy(84.00,0.00)}{\pgfbox[bottom,left]{\fontsize{9.96}{11.95}\selectfont ,}}
\pgfputat{\pgfxy(40.00,-26.00)}{\pgfbox[bottom,left]{\fontsize{7.97}{9.56}\selectfont \makebox[0pt]{$\bar{8}$}}}
\pgfputat{\pgfxy(60.00,-26.00)}{\pgfbox[bottom,left]{\fontsize{7.97}{9.56}\selectfont \makebox[0pt]{$\bar{1}$}}}
\pgfputat{\pgfxy(50.00,-26.00)}{\pgfbox[bottom,left]{\fontsize{7.97}{9.56}\selectfont \makebox[0pt]{$2$}}}
\pgfputat{\pgfxy(80.00,-26.00)}{\pgfbox[bottom,left]{\fontsize{7.97}{9.56}\selectfont \makebox[0pt]{$5$}}}
\pgfputat{\pgfxy(70.00,-26.00)}{\pgfbox[bottom,left]{\fontsize{7.97}{9.56}\selectfont \makebox[0pt]{$7$}}}
\pgfputat{\pgfxy(30.00,-26.00)}{\pgfbox[bottom,left]{\fontsize{7.97}{9.56}\selectfont \makebox[0pt]{$\bar{4}$}}}
\pgfputat{\pgfxy(10.00,-26.00)}{\pgfbox[bottom,left]{\fontsize{7.97}{9.56}\selectfont \makebox[0pt]{$\bar{3}$}}}
\pgfputat{\pgfxy(20.00,-26.00)}{\pgfbox[bottom,left]{\fontsize{7.97}{9.56}\selectfont \makebox[0pt]{$9$}}}
\pgfputat{\pgfxy(0.00,-26.00)}{\pgfbox[bottom,left]{\fontsize{7.97}{9.56}\selectfont \makebox[0pt]{$6$}}}
\end{pgfpicture}%
$$
we obtain $\omega^B(T) = \sigma^R = 57\bar{1}2\bar{8}\bar{4}9\bar{3}6$ by reading reversely vertices of $T$ in inorder.
So bijections for type A and type B commute in the diagram of Figure~\ref{fig:diagram1}.
\begin{figure}[t]
\[
\xymatrix{
\DUnk    \ar[r]^{\psi}
&\Tnk    \ar[r]^{\omega}
&\DAnk  \ar[rr]^-{\varphi}
&&\RSnkone \\
\DUnkB   \ar[u]^-{\pi} \ar[r]^-{\psi^{B}}   
&\TnkB   \ar[u]^-{\pi} \ar[r]^-{\omega^{B}}
& \DAnkB \ar[u]^-{\pi} \ar@{.}[r]^-{?}
& \DAH_{n+1, n+2-k} \ar[r]^-{\varphi^{(B)}}
& \RSB_{n, n+1-k}
}
\]
\caption{Bijections between Entringer families and Arnold families}
\label{fig:diagram1}
\end{figure}
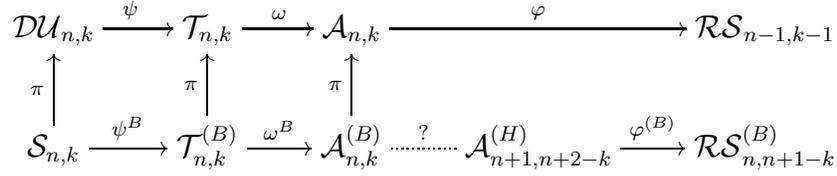

We summarize four interpretations for Entringer numbers $E_{4,k}$, $k \in \set{2,3,4}$ 
in Table~\ref{tab:A} and 
left three interpretations for Arnold number $S_{3,k}$, $k \in \set{1,2,3}$ in Table~\ref{tab:B}. 
In every column, the corresponding elements are described via the different bijections mentioned in the paper.

\begin{table}[t]
\centering
\begin{tabular}{c||c|>{\centering}m{1in}|c|c}
$k$ & $\tau \in \DU_{4,k}$ & $\psi(\tau) \in \T_{4,k}$ & $\omega(\psi(\tau)) \in \DA_{4,k}$ & $\varphi(\omega(\psi(\tau))) \in \RS_{3,k-1}$ \\ \hline
\multirow{1}{*}{$2$} & $2143$ & 
\centering
\begin{pgfpicture}{-8.30mm}{-3.30mm}{11.80mm}{10.50mm}
\pgfsetxvec{\pgfpoint{0.70mm}{0mm}}
\pgfsetyvec{\pgfpoint{0mm}{0.70mm}}
\color[rgb]{0,0,0}\pgfsetlinewidth{0.30mm}\pgfsetdash{}{0mm}
\pgfcircle[fill]{\pgfxy(0.00,0.00)}{0.42mm}
\pgfcircle[stroke]{\pgfxy(0.00,0.00)}{0.42mm}
\pgfputat{\pgfxy(-2.00,-1.00)}{\pgfbox[bottom,left]{\fontsize{7.97}{9.56}\selectfont \makebox[0pt][r]{1}}}
\pgfmoveto{\pgfxy(-5.00,5.00)}\pgflineto{\pgfxy(0.00,0.00)}\pgfstroke
\pgfcircle[fill]{\pgfxy(-5.00,5.00)}{0.42mm}
\pgfcircle[stroke]{\pgfxy(-5.00,5.00)}{0.42mm}
\pgfputat{\pgfxy(-7.00,4.00)}{\pgfbox[bottom,left]{\fontsize{7.97}{9.56}\selectfont \makebox[0pt][r]{2}}}
\pgfmoveto{\pgfxy(0.00,0.00)}\pgflineto{\pgfxy(10.00,5.00)}\pgfstroke
\pgfputat{\pgfxy(12.00,4.00)}{\pgfbox[bottom,left]{\fontsize{7.97}{9.56}\selectfont 3}}
\pgfcircle[fill]{\pgfxy(10.00,5.00)}{0.42mm}
\pgfcircle[stroke]{\pgfxy(10.00,5.00)}{0.42mm}
\pgfmoveto{\pgfxy(5.00,10.00)}\pgflineto{\pgfxy(10.00,5.00)}\pgfstroke
\pgfputat{\pgfxy(7.00,9.00)}{\pgfbox[bottom,left]{\fontsize{7.97}{9.56}\selectfont 4}}
\pgfcircle[fill]{\pgfxy(5.00,10.00)}{0.42mm}
\pgfcircle[stroke]{\pgfxy(5.00,10.00)}{0.42mm}
\end{pgfpicture}%
												  	& $3412$ & $231$ \\ \hline
\multirow{2}{*}{$3$} & $3241$ & 
\centering
\begin{pgfpicture}{33.70mm}{-3.30mm}{53.80mm}{10.50mm}
\pgfsetxvec{\pgfpoint{0.70mm}{0mm}}
\pgfsetyvec{\pgfpoint{0mm}{0.70mm}}
\color[rgb]{0,0,0}\pgfsetlinewidth{0.30mm}\pgfsetdash{}{0mm}
\pgfcircle[fill]{\pgfxy(70.00,0.00)}{0.42mm}
\pgfcircle[stroke]{\pgfxy(70.00,0.00)}{0.42mm}
\pgfputat{\pgfxy(72.00,-1.00)}{\pgfbox[bottom,left]{\fontsize{7.97}{9.56}\selectfont 1}}
\pgfmoveto{\pgfxy(60.00,5.00)}\pgflineto{\pgfxy(70.00,0.00)}\pgfstroke
\pgfcircle[fill]{\pgfxy(60.00,5.00)}{0.42mm}
\pgfcircle[stroke]{\pgfxy(60.00,5.00)}{0.42mm}
\pgfputat{\pgfxy(58.00,4.00)}{\pgfbox[bottom,left]{\fontsize{7.97}{9.56}\selectfont \makebox[0pt][r]{2}}}
\pgfmoveto{\pgfxy(55.00,10.00)}\pgflineto{\pgfxy(60.00,5.00)}\pgfstroke
\pgfputat{\pgfxy(53.00,9.00)}{\pgfbox[bottom,left]{\fontsize{7.97}{9.56}\selectfont \makebox[0pt][r]{3}}}
\pgfcircle[fill]{\pgfxy(55.00,10.00)}{0.42mm}
\pgfcircle[stroke]{\pgfxy(55.00,10.00)}{0.42mm}
\pgfcircle[fill]{\pgfxy(65.00,10.00)}{0.42mm}
\pgfcircle[stroke]{\pgfxy(65.00,10.00)}{0.42mm}
\pgfmoveto{\pgfxy(65.00,10.00)}\pgflineto{\pgfxy(60.00,5.00)}\pgfstroke
\pgfputat{\pgfxy(67.00,9.00)}{\pgfbox[bottom,left]{\fontsize{7.97}{9.56}\selectfont 4}}
\end{pgfpicture}%
												 	& $1423$ & $132$ \\ \cline{2-5}
                     & $3142$ & 
\centering
\begin{pgfpicture}{12.70mm}{-3.30mm}{32.80mm}{10.50mm}
\pgfsetxvec{\pgfpoint{0.70mm}{0mm}}
\pgfsetyvec{\pgfpoint{0mm}{0.70mm}}
\color[rgb]{0,0,0}\pgfsetlinewidth{0.30mm}\pgfsetdash{}{0mm}
\pgfcircle[fill]{\pgfxy(35.00,0.00)}{0.42mm}
\pgfcircle[stroke]{\pgfxy(35.00,0.00)}{0.42mm}
\pgfputat{\pgfxy(33.00,-1.00)}{\pgfbox[bottom,left]{\fontsize{7.97}{9.56}\selectfont \makebox[0pt][r]{1}}}
\pgfmoveto{\pgfxy(30.00,5.00)}\pgflineto{\pgfxy(35.00,0.00)}\pgfstroke
\pgfcircle[fill]{\pgfxy(30.00,5.00)}{0.42mm}
\pgfcircle[stroke]{\pgfxy(30.00,5.00)}{0.42mm}
\pgfputat{\pgfxy(28.00,4.00)}{\pgfbox[bottom,left]{\fontsize{7.97}{9.56}\selectfont \makebox[0pt][r]{2}}}
\pgfmoveto{\pgfxy(25.00,10.00)}\pgflineto{\pgfxy(30.00,5.00)}\pgfstroke
\pgfputat{\pgfxy(23.00,9.00)}{\pgfbox[bottom,left]{\fontsize{7.97}{9.56}\selectfont \makebox[0pt][r]{3}}}
\pgfcircle[fill]{\pgfxy(25.00,10.00)}{0.42mm}
\pgfcircle[stroke]{\pgfxy(25.00,10.00)}{0.42mm}
\pgfcircle[fill]{\pgfxy(40.00,5.00)}{0.42mm}
\pgfcircle[stroke]{\pgfxy(40.00,5.00)}{0.42mm}
\pgfmoveto{\pgfxy(40.00,5.00)}\pgflineto{\pgfxy(35.00,0.00)}\pgfstroke
\pgfputat{\pgfxy(42.00,4.00)}{\pgfbox[bottom,left]{\fontsize{7.97}{9.56}\selectfont 4}}
\end{pgfpicture}%
                     							  	& $4123$ & $312$ \\ \hline
\multirow{2}{*}{$4$} & $4231$ & 
\centering
\begin{pgfpicture}{75.70mm}{-3.30mm}{93.42mm}{14.00mm}
\pgfsetxvec{\pgfpoint{0.70mm}{0mm}}
\pgfsetyvec{\pgfpoint{0mm}{0.70mm}}
\color[rgb]{0,0,0}\pgfsetlinewidth{0.30mm}\pgfsetdash{}{0mm}
\pgfcircle[fill]{\pgfxy(130.00,0.00)}{0.42mm}
\pgfcircle[stroke]{\pgfxy(130.00,0.00)}{0.42mm}
\pgfputat{\pgfxy(128.00,-1.00)}{\pgfbox[bottom,left]{\fontsize{7.97}{9.56}\selectfont \makebox[0pt][r]{1}}}
\pgfmoveto{\pgfxy(125.00,5.00)}\pgflineto{\pgfxy(130.00,0.00)}\pgfstroke
\pgfcircle[fill]{\pgfxy(125.00,5.00)}{0.42mm}
\pgfcircle[stroke]{\pgfxy(125.00,5.00)}{0.42mm}
\pgfputat{\pgfxy(123.00,4.00)}{\pgfbox[bottom,left]{\fontsize{7.97}{9.56}\selectfont \makebox[0pt][r]{2}}}
\pgfmoveto{\pgfxy(120.00,10.00)}\pgflineto{\pgfxy(125.00,5.00)}\pgfstroke
\pgfputat{\pgfxy(118.00,9.00)}{\pgfbox[bottom,left]{\fontsize{7.97}{9.56}\selectfont \makebox[0pt][r]{3}}}
\pgfcircle[fill]{\pgfxy(120.00,10.00)}{0.42mm}
\pgfcircle[stroke]{\pgfxy(120.00,10.00)}{0.42mm}
\pgfcircle[fill]{\pgfxy(115.00,15.00)}{0.42mm}
\pgfcircle[stroke]{\pgfxy(115.00,15.00)}{0.42mm}
\pgfmoveto{\pgfxy(115.00,15.00)}\pgflineto{\pgfxy(120.00,10.00)}\pgfstroke
\pgfputat{\pgfxy(113.00,14.00)}{\pgfbox[bottom,left]{\fontsize{7.97}{9.56}\selectfont \makebox[0pt][r]{4}}}
\end{pgfpicture}%
													& $1234$ & $123$ \\ \cline{2-5}
                     & $4132$ & 
\centering
\begin{pgfpicture}{54.70mm}{-3.30mm}{74.80mm}{10.50mm}
\pgfsetxvec{\pgfpoint{0.70mm}{0mm}}
\pgfsetyvec{\pgfpoint{0mm}{0.70mm}}
\color[rgb]{0,0,0}\pgfsetlinewidth{0.30mm}\pgfsetdash{}{0mm}
\pgfcircle[fill]{\pgfxy(95.00,0.00)}{0.42mm}
\pgfcircle[stroke]{\pgfxy(95.00,0.00)}{0.42mm}
\pgfputat{\pgfxy(93.00,-1.00)}{\pgfbox[bottom,left]{\fontsize{7.97}{9.56}\selectfont \makebox[0pt][r]{1}}}
\pgfmoveto{\pgfxy(90.00,5.00)}\pgflineto{\pgfxy(95.00,0.00)}\pgfstroke
\pgfcircle[fill]{\pgfxy(90.00,5.00)}{0.42mm}
\pgfcircle[stroke]{\pgfxy(90.00,5.00)}{0.42mm}
\pgfputat{\pgfxy(88.00,4.00)}{\pgfbox[bottom,left]{\fontsize{7.97}{9.56}\selectfont \makebox[0pt][r]{2}}}
\pgfmoveto{\pgfxy(95.00,0.00)}\pgflineto{\pgfxy(100.00,5.00)}\pgfstroke
\pgfputat{\pgfxy(102.00,4.00)}{\pgfbox[bottom,left]{\fontsize{7.97}{9.56}\selectfont 3}}
\pgfcircle[fill]{\pgfxy(100.00,5.00)}{0.42mm}
\pgfcircle[stroke]{\pgfxy(100.00,5.00)}{0.42mm}
\pgfmoveto{\pgfxy(85.00,10.00)}\pgflineto{\pgfxy(90.00,5.00)}\pgfstroke
\pgfputat{\pgfxy(83.00,9.00)}{\pgfbox[bottom,left]{\fontsize{7.97}{9.56}\selectfont \makebox[0pt][r]{4}}}
\pgfcircle[fill]{\pgfxy(85.00,10.00)}{0.42mm}
\pgfcircle[stroke]{\pgfxy(85.00,10.00)}{0.42mm}
\end{pgfpicture}%
                     								& $3124$ & $213$ \\ 
\end{tabular}
\vspace{1em}
\caption{Three bijections between Entringer families with $n=4$ and $2\le k \le 4$}
\label{tab:A}
\end{table}

\begin{table}[t]
\centering
\renewcommand{\arraystretch}{0}
\begin{tabular}{c|>{\centering}m{1.2in}|c||c||c|c}
$\tau \in \DUB_{3,k}$ & $\psi^{(B)}(\tau) \in \TB_{3,k}$ & $\omega^{(B)}(\psi^{(B)}(\tau)) \in \DAB_{3,k}$ 
& $k$ & $\sigma \in \DAH_{4,5-k}$ & $\varphi^{(B)}(\sigma) \in \RSB_{3,4-k}$ \\ \hline
$1\bar{2}3$       & 
\centering
\begin{pgfpicture}{21.80mm}{-6.10mm}{34.20mm}{9.10mm}
\pgfsetxvec{\pgfpoint{0.70mm}{0mm}}
\pgfsetyvec{\pgfpoint{0mm}{0.70mm}}
\color[rgb]{0,0,0}\pgfsetlinewidth{0.30mm}\pgfsetdash{}{0mm}
\pgfcircle[fill]{\pgfxy(40.00,0.00)}{0.42mm}
\pgfcircle[stroke]{\pgfxy(40.00,0.00)}{0.42mm}
\pgfputat{\pgfxy(40.00,-5.00)}{\pgfbox[bottom,left]{\fontsize{7.97}{9.56}\selectfont \makebox[0pt]{$\bar{2}$}}}
\pgfmoveto{\pgfxy(35.00,5.00)}\pgflineto{\pgfxy(40.00,0.00)}\pgfstroke
\pgfcircle[fill]{\pgfxy(35.00,5.00)}{0.42mm}
\pgfcircle[stroke]{\pgfxy(35.00,5.00)}{0.42mm}
\pgfputat{\pgfxy(35.00,7.00)}{\pgfbox[bottom,left]{\fontsize{7.97}{9.56}\selectfont \makebox[0pt]{1}}}
\pgfmoveto{\pgfxy(40.00,0.00)}\pgflineto{\pgfxy(45.00,5.00)}\pgfstroke
\pgfputat{\pgfxy(45.00,7.00)}{\pgfbox[bottom,left]{\fontsize{7.97}{9.56}\selectfont \makebox[0pt]{3}}}
\pgfcircle[fill]{\pgfxy(45.00,5.00)}{0.42mm}
\pgfcircle[stroke]{\pgfxy(45.00,5.00)}{0.42mm}
\end{pgfpicture}%
& $3\bar{2}1$       & \multirow{3}{*}{$1$} & $1234$             & $123$             \\ \cline{1-3} \cline{5-6} 
$1\bar{3}2$       & 
\centering
\begin{pgfpicture}{35.80mm}{-6.10mm}{48.20mm}{9.10mm}
\pgfsetxvec{\pgfpoint{0.70mm}{0mm}}
\pgfsetyvec{\pgfpoint{0mm}{0.70mm}}
\color[rgb]{0,0,0}\pgfsetlinewidth{0.30mm}\pgfsetdash{}{0mm}
\pgfcircle[fill]{\pgfxy(60.00,0.00)}{0.42mm}
\pgfcircle[stroke]{\pgfxy(60.00,0.00)}{0.42mm}
\pgfputat{\pgfxy(60.00,-5.00)}{\pgfbox[bottom,left]{\fontsize{7.97}{9.56}\selectfont \makebox[0pt]{$\bar{3}$}}}
\pgfmoveto{\pgfxy(55.00,5.00)}\pgflineto{\pgfxy(60.00,0.00)}\pgfstroke
\pgfcircle[fill]{\pgfxy(55.00,5.00)}{0.42mm}
\pgfcircle[stroke]{\pgfxy(55.00,5.00)}{0.42mm}
\pgfputat{\pgfxy(55.00,7.00)}{\pgfbox[bottom,left]{\fontsize{7.97}{9.56}\selectfont \makebox[0pt]{1}}}
\pgfmoveto{\pgfxy(60.00,0.00)}\pgflineto{\pgfxy(65.00,5.00)}\pgfstroke
\pgfputat{\pgfxy(65.00,7.00)}{\pgfbox[bottom,left]{\fontsize{7.97}{9.56}\selectfont \makebox[0pt]{2}}}
\pgfcircle[fill]{\pgfxy(65.00,5.00)}{0.42mm}
\pgfcircle[stroke]{\pgfxy(65.00,5.00)}{0.42mm}
\end{pgfpicture}%
& $2\bar{3}1$       &                      & $3124$             & $213$             \\ \cline{1-3} \cline{5-6} 
$1\bar{3}\bar{2}$ & 
\centering
\begin{pgfpicture}{78.08mm}{11.58mm}{91.83mm}{26.60mm}
\pgfsetxvec{\pgfpoint{0.70mm}{0mm}}
\pgfsetyvec{\pgfpoint{0mm}{0.70mm}}
\color[rgb]{0,0,0}\pgfsetlinewidth{0.30mm}\pgfsetdash{}{0mm}
\pgfcircle[fill]{\pgfxy(125.00,20.00)}{0.42mm}
\pgfcircle[stroke]{\pgfxy(125.00,20.00)}{0.42mm}
\pgfmoveto{\pgfxy(120.00,25.00)}\pgflineto{\pgfxy(125.00,20.00)}\pgfstroke
\pgfcircle[fill]{\pgfxy(120.00,25.00)}{0.42mm}
\pgfcircle[stroke]{\pgfxy(120.00,25.00)}{0.42mm}
\pgfmoveto{\pgfxy(115.00,30.00)}\pgflineto{\pgfxy(120.00,25.00)}\pgfstroke
\pgfputat{\pgfxy(125.00,22.00)}{\pgfbox[bottom,left]{\fontsize{7.97}{9.56}\selectfont $\bar{3}$}}
\pgfputat{\pgfxy(120.00,27.00)}{\pgfbox[bottom,left]{\fontsize{7.97}{9.56}\selectfont $\bar{2}$}}
\pgfputat{\pgfxy(115.00,32.00)}{\pgfbox[bottom,left]{\fontsize{7.97}{9.56}\selectfont 1}}
\pgfcircle[fill]{\pgfxy(115.00,30.00)}{0.42mm}
\pgfcircle[stroke]{\pgfxy(115.00,30.00)}{0.42mm}
\end{pgfpicture}%
& $\bar{3}\bar{2}1$ &                      & $\bar{3}124$       & $\bar{2}13$       \\ \hline
$213$             & 
\centering
\begin{pgfpicture}{-6.20mm}{-6.10mm}{6.20mm}{9.10mm}
\pgfsetxvec{\pgfpoint{0.70mm}{0mm}}
\pgfsetyvec{\pgfpoint{0mm}{0.70mm}}
\color[rgb]{0,0,0}\pgfsetlinewidth{0.30mm}\pgfsetdash{}{0mm}
\pgfcircle[fill]{\pgfxy(0.00,0.00)}{0.42mm}
\pgfcircle[stroke]{\pgfxy(0.00,0.00)}{0.42mm}
\pgfputat{\pgfxy(0.00,-5.00)}{\pgfbox[bottom,left]{\fontsize{7.97}{9.56}\selectfont \makebox[0pt]{1}}}
\pgfmoveto{\pgfxy(-5.00,5.00)}\pgflineto{\pgfxy(0.00,0.00)}\pgfstroke
\pgfcircle[fill]{\pgfxy(-5.00,5.00)}{0.42mm}
\pgfcircle[stroke]{\pgfxy(-5.00,5.00)}{0.42mm}
\pgfputat{\pgfxy(-5.00,7.00)}{\pgfbox[bottom,left]{\fontsize{7.97}{9.56}\selectfont \makebox[0pt]{2}}}
\pgfmoveto{\pgfxy(0.00,0.00)}\pgflineto{\pgfxy(5.00,5.00)}\pgfstroke
\pgfputat{\pgfxy(5.00,7.00)}{\pgfbox[bottom,left]{\fontsize{7.97}{9.56}\selectfont \makebox[0pt]{3}}}
\pgfcircle[fill]{\pgfxy(5.00,5.00)}{0.42mm}
\pgfcircle[stroke]{\pgfxy(5.00,5.00)}{0.42mm}
\end{pgfpicture}%
& $312$             & \multirow{4}{*}{$2$} & $1423$             & $132$             \\ \cline{1-3} \cline{5-6} 
$2\bar{1}3$       & 
\centering
\begin{pgfpicture}{7.80mm}{-6.10mm}{20.20mm}{9.10mm}
\pgfsetxvec{\pgfpoint{0.70mm}{0mm}}
\pgfsetyvec{\pgfpoint{0mm}{0.70mm}}
\color[rgb]{0,0,0}\pgfsetlinewidth{0.30mm}\pgfsetdash{}{0mm}
\pgfcircle[fill]{\pgfxy(20.00,0.00)}{0.42mm}
\pgfcircle[stroke]{\pgfxy(20.00,0.00)}{0.42mm}
\pgfputat{\pgfxy(20.00,-5.00)}{\pgfbox[bottom,left]{\fontsize{7.97}{9.56}\selectfont \makebox[0pt]{$\bar{1}$}}}
\pgfmoveto{\pgfxy(15.00,5.00)}\pgflineto{\pgfxy(20.00,0.00)}\pgfstroke
\pgfcircle[fill]{\pgfxy(15.00,5.00)}{0.42mm}
\pgfcircle[stroke]{\pgfxy(15.00,5.00)}{0.42mm}
\pgfputat{\pgfxy(15.00,7.00)}{\pgfbox[bottom,left]{\fontsize{7.97}{9.56}\selectfont \makebox[0pt]{2}}}
\pgfmoveto{\pgfxy(20.00,0.00)}\pgflineto{\pgfxy(25.00,5.00)}\pgfstroke
\pgfputat{\pgfxy(25.00,7.00)}{\pgfbox[bottom,left]{\fontsize{7.97}{9.56}\selectfont \makebox[0pt]{3}}}
\pgfcircle[fill]{\pgfxy(25.00,5.00)}{0.42mm}
\pgfcircle[stroke]{\pgfxy(25.00,5.00)}{0.42mm}
\end{pgfpicture}%
& $3\bar{1}2$       &                      & $1\bar{4}23$       & $1\bar{3}2$       \\ \cline{1-3} \cline{5-6} 
$2\bar{3}1$       & 
\centering
\begin{pgfpicture}{36.08mm}{11.58mm}{49.83mm}{26.60mm}
\pgfsetxvec{\pgfpoint{0.70mm}{0mm}}
\pgfsetyvec{\pgfpoint{0mm}{0.70mm}}
\color[rgb]{0,0,0}\pgfsetlinewidth{0.30mm}\pgfsetdash{}{0mm}
\pgfcircle[fill]{\pgfxy(65.00,20.00)}{0.42mm}
\pgfcircle[stroke]{\pgfxy(65.00,20.00)}{0.42mm}
\pgfmoveto{\pgfxy(60.00,25.00)}\pgflineto{\pgfxy(65.00,20.00)}\pgfstroke
\pgfcircle[fill]{\pgfxy(60.00,25.00)}{0.42mm}
\pgfcircle[stroke]{\pgfxy(60.00,25.00)}{0.42mm}
\pgfmoveto{\pgfxy(55.00,30.00)}\pgflineto{\pgfxy(60.00,25.00)}\pgfstroke
\pgfputat{\pgfxy(65.00,22.00)}{\pgfbox[bottom,left]{\fontsize{7.97}{9.56}\selectfont $\bar{3}$}}
\pgfputat{\pgfxy(60.00,27.00)}{\pgfbox[bottom,left]{\fontsize{7.97}{9.56}\selectfont 1}}
\pgfputat{\pgfxy(55.00,32.00)}{\pgfbox[bottom,left]{\fontsize{7.97}{9.56}\selectfont 2}}
\pgfcircle[fill]{\pgfxy(55.00,30.00)}{0.42mm}
\pgfcircle[stroke]{\pgfxy(55.00,30.00)}{0.42mm}
\end{pgfpicture}%
& $\bar{3}12$       &                      & $4123$             & $312$             \\ \cline{1-3} \cline{5-6} 
$2\bar{3}\bar{1}$ & 
\centering
\begin{pgfpicture}{64.08mm}{11.58mm}{77.83mm}{26.60mm}
\pgfsetxvec{\pgfpoint{0.70mm}{0mm}}
\pgfsetyvec{\pgfpoint{0mm}{0.70mm}}
\color[rgb]{0,0,0}\pgfsetlinewidth{0.30mm}\pgfsetdash{}{0mm}
\pgfcircle[fill]{\pgfxy(105.00,20.00)}{0.42mm}
\pgfcircle[stroke]{\pgfxy(105.00,20.00)}{0.42mm}
\pgfmoveto{\pgfxy(100.00,25.00)}\pgflineto{\pgfxy(105.00,20.00)}\pgfstroke
\pgfcircle[fill]{\pgfxy(100.00,25.00)}{0.42mm}
\pgfcircle[stroke]{\pgfxy(100.00,25.00)}{0.42mm}
\pgfmoveto{\pgfxy(95.00,30.00)}\pgflineto{\pgfxy(100.00,25.00)}\pgfstroke
\pgfputat{\pgfxy(105.00,22.00)}{\pgfbox[bottom,left]{\fontsize{7.97}{9.56}\selectfont $\bar{3}$}}
\pgfputat{\pgfxy(100.00,27.00)}{\pgfbox[bottom,left]{\fontsize{7.97}{9.56}\selectfont $\bar{1}$}}
\pgfputat{\pgfxy(95.00,32.00)}{\pgfbox[bottom,left]{\fontsize{7.97}{9.56}\selectfont 2}}
\pgfcircle[fill]{\pgfxy(95.00,30.00)}{0.42mm}
\pgfcircle[stroke]{\pgfxy(95.00,30.00)}{0.42mm}
\end{pgfpicture}%
& $\bar{3}\bar{1}2$ &                      & $\bar{4}123$       & $\bar{3}12$       \\ \hline
$312$             & 
\centering
\begin{pgfpicture}{-5.92mm}{11.58mm}{6.90mm}{26.60mm}
\pgfsetxvec{\pgfpoint{0.70mm}{0mm}}
\pgfsetyvec{\pgfpoint{0mm}{0.70mm}}
\color[rgb]{0,0,0}\pgfsetlinewidth{0.30mm}\pgfsetdash{}{0mm}
\pgfcircle[fill]{\pgfxy(5.00,20.00)}{0.42mm}
\pgfcircle[stroke]{\pgfxy(5.00,20.00)}{0.42mm}
\pgfmoveto{\pgfxy(0.00,25.00)}\pgflineto{\pgfxy(5.00,20.00)}\pgfstroke
\pgfcircle[fill]{\pgfxy(0.00,25.00)}{0.42mm}
\pgfcircle[stroke]{\pgfxy(0.00,25.00)}{0.42mm}
\pgfmoveto{\pgfxy(-5.00,30.00)}\pgflineto{\pgfxy(0.00,25.00)}\pgfstroke
\pgfputat{\pgfxy(5.00,22.00)}{\pgfbox[bottom,left]{\fontsize{7.97}{9.56}\selectfont 1}}
\pgfputat{\pgfxy(0.00,27.00)}{\pgfbox[bottom,left]{\fontsize{7.97}{9.56}\selectfont 2}}
\pgfputat{\pgfxy(-5.00,32.00)}{\pgfbox[bottom,left]{\fontsize{7.97}{9.56}\selectfont 3}}
\pgfcircle[fill]{\pgfxy(-5.00,30.00)}{0.42mm}
\pgfcircle[stroke]{\pgfxy(-5.00,30.00)}{0.42mm}
\end{pgfpicture}%
& $123$             & \multirow{4}{*}{$3$} & $3412$             & $231$             \\ \cline{1-3} \cline{5-6} 
$3\bar{1}2$       & 
\centering
\begin{pgfpicture}{8.08mm}{11.58mm}{21.83mm}{26.60mm}
\pgfsetxvec{\pgfpoint{0.70mm}{0mm}}
\pgfsetyvec{\pgfpoint{0mm}{0.70mm}}
\color[rgb]{0,0,0}\pgfsetlinewidth{0.30mm}\pgfsetdash{}{0mm}
\pgfcircle[fill]{\pgfxy(25.00,20.00)}{0.42mm}
\pgfcircle[stroke]{\pgfxy(25.00,20.00)}{0.42mm}
\pgfmoveto{\pgfxy(20.00,25.00)}\pgflineto{\pgfxy(25.00,20.00)}\pgfstroke
\pgfcircle[fill]{\pgfxy(20.00,25.00)}{0.42mm}
\pgfcircle[stroke]{\pgfxy(20.00,25.00)}{0.42mm}
\pgfmoveto{\pgfxy(15.00,30.00)}\pgflineto{\pgfxy(20.00,25.00)}\pgfstroke
\pgfputat{\pgfxy(25.00,22.00)}{\pgfbox[bottom,left]{\fontsize{7.97}{9.56}\selectfont $\bar{1}$}}
\pgfputat{\pgfxy(20.00,27.00)}{\pgfbox[bottom,left]{\fontsize{7.97}{9.56}\selectfont 2}}
\pgfputat{\pgfxy(15.00,32.00)}{\pgfbox[bottom,left]{\fontsize{7.97}{9.56}\selectfont 3}}
\pgfcircle[fill]{\pgfxy(15.00,30.00)}{0.42mm}
\pgfcircle[stroke]{\pgfxy(15.00,30.00)}{0.42mm}
\end{pgfpicture}%
& $\bar{1}23$       &                      & $\bar{3}412$       & $\bar{2}31$       \\ \cline{1-3} \cline{5-6} 
$3\bar{2}1$       & 
\centering
\begin{pgfpicture}{22.08mm}{11.58mm}{35.83mm}{26.60mm}
\pgfsetxvec{\pgfpoint{0.70mm}{0mm}}
\pgfsetyvec{\pgfpoint{0mm}{0.70mm}}
\color[rgb]{0,0,0}\pgfsetlinewidth{0.30mm}\pgfsetdash{}{0mm}
\pgfcircle[fill]{\pgfxy(45.00,20.00)}{0.42mm}
\pgfcircle[stroke]{\pgfxy(45.00,20.00)}{0.42mm}
\pgfmoveto{\pgfxy(40.00,25.00)}\pgflineto{\pgfxy(45.00,20.00)}\pgfstroke
\pgfcircle[fill]{\pgfxy(40.00,25.00)}{0.42mm}
\pgfcircle[stroke]{\pgfxy(40.00,25.00)}{0.42mm}
\pgfmoveto{\pgfxy(35.00,30.00)}\pgflineto{\pgfxy(40.00,25.00)}\pgfstroke
\pgfputat{\pgfxy(45.00,22.00)}{\pgfbox[bottom,left]{\fontsize{7.97}{9.56}\selectfont $\bar{2}$}}
\pgfputat{\pgfxy(40.00,27.00)}{\pgfbox[bottom,left]{\fontsize{7.97}{9.56}\selectfont 1}}
\pgfputat{\pgfxy(35.00,32.00)}{\pgfbox[bottom,left]{\fontsize{7.97}{9.56}\selectfont 3}}
\pgfcircle[fill]{\pgfxy(35.00,30.00)}{0.42mm}
\pgfcircle[stroke]{\pgfxy(35.00,30.00)}{0.42mm}
\end{pgfpicture}%
& $\bar{2}13$       &                      & $3\bar{4}12$       & $2\bar{3}1$       \\ \cline{1-3} \cline{5-6} 
$3\bar{2}\bar{1}$ & 
\centering
\begin{pgfpicture}{50.08mm}{11.58mm}{63.83mm}{26.60mm}
\pgfsetxvec{\pgfpoint{0.70mm}{0mm}}
\pgfsetyvec{\pgfpoint{0mm}{0.70mm}}
\color[rgb]{0,0,0}\pgfsetlinewidth{0.30mm}\pgfsetdash{}{0mm}
\pgfcircle[fill]{\pgfxy(85.00,20.00)}{0.42mm}
\pgfcircle[stroke]{\pgfxy(85.00,20.00)}{0.42mm}
\pgfmoveto{\pgfxy(80.00,25.00)}\pgflineto{\pgfxy(85.00,20.00)}\pgfstroke
\pgfcircle[fill]{\pgfxy(80.00,25.00)}{0.42mm}
\pgfcircle[stroke]{\pgfxy(80.00,25.00)}{0.42mm}
\pgfmoveto{\pgfxy(75.00,30.00)}\pgflineto{\pgfxy(80.00,25.00)}\pgfstroke
\pgfputat{\pgfxy(85.00,22.00)}{\pgfbox[bottom,left]{\fontsize{7.97}{9.56}\selectfont $\bar{2}$}}
\pgfputat{\pgfxy(80.00,27.00)}{\pgfbox[bottom,left]{\fontsize{7.97}{9.56}\selectfont $\bar{1}$}}
\pgfputat{\pgfxy(75.00,32.00)}{\pgfbox[bottom,left]{\fontsize{7.97}{9.56}\selectfont 3}}
\pgfcircle[fill]{\pgfxy(75.00,30.00)}{0.42mm}
\pgfcircle[stroke]{\pgfxy(75.00,30.00)}{0.42mm}
\end{pgfpicture}%
& $\bar{2}\bar{1}3$ &                      & $\bar{3}\bar{4}12$ & $\bar{2}\bar{3}1$ \\ \hline
\end{tabular}
\vspace{1em}
\caption{Three bijections between Arnold families with $n=3$ and $1\le k \le 3$}
\label{tab:B}
\end{table}

\section{New description of two known bijections}
\label{sec:psi}

\subsection{A decomposition  of  Chuang et al.'s bijection $\phi: \T_{n+1} \to \RS_{n}$}
In 2012, Chuang et al.~\cite{CEFP12} construct a bijection $\phi: \T_{n+1} \to \RS_{n}$.
If $x$ has only one child then it is the right child of $x$.
They described the bijection $\phi$ between increasing 1-2 trees and Simsun permutations using 
the following algorithm.

\subsection*{Algorithm~A}
\begin{itemize}
\item[(A1)] If $T$ consists of the root vertex then $T$ is associated with an empty word.
\item[(A2)] Otherwise, the word $\rho(T)$ is defined inductively by the factorization 
$$
\rho(T)=\omega\cdot \rho(T'),
$$
where the subword $\omega$ and the subtree $T'$ are determined as follows.
\begin{enumerate}[(a)]
\item 
If the root of $T$ has only one child $x$ then let $\omega = x$ (consisting of a single
letter $x$), let $T'=\tau(x)$ (i.e, obtained from $T$ by deleting the root of $T$), and
relabel the vertex $x$ by $1$.
\item 
If the root of $T$ has two children $u$, $v$ with $u>v$ 
then traverse the right subtree $\tau(u)$ reversely in inorder, 
put down the word $\omega$ of the vertices of $\tau(u)$ and let $T'=T-\tau(u)$.
\end{enumerate}
\end{itemize}
As deleted only $1$ in $\rho(T)$, 
every permutation $\rho(T)=a_1 a_2\cdots a_n$ is a Simsun permutation on $\set{2, 3, \dots, n+1}$. 
Thus we get a Simsun permutation $\phi(T)=b_1 b_2\cdots b_n$ on $[n]$ with $b_i = a_i - 1$ for all $1\le i \le n$.

\begin{rmk}
Originally, in \cite{CEFP12}, the increasing 1-2 trees on $n$ vertices are labeled with $0,1, \ldots, n-1$ instead of $1,2, \ldots, n$ and was drawn in a \emph{canonical form} such that if a vertex $x$ has two children $u$, $v$ with $u > v$ then $u$ is the left child, $v$ is the right child. 
\end{rmk}

%
\begin{thm}
The bijection $\phi : \Tnk \to \RSnkone$ is the composition of $\varphi$ and $\omega$; i.e.,  $\phi = \varphi \circ \omega$.
\end{thm}




\begin{proof}
Suppose that we let $\omega$ be the root of $T$ without relabeling the vertex $x$ in (A2a),
it is obvious that Algorithm~A follow the bijeciton $\omega$, i.e., 
reading the vertex of $T$ reversely in inorder.
So it is enough to show that the change in two rules of (A2a) follows $\varphi$.

The root of $T'$ becomes to the left-child of the root of $T$ after (A2a) and 
the root of $T'$ is same with the root of $T$ in (A2b).
So Algorihm A executes the step (A2a) 
only when a vertex on the minimal path of an original tree becomes the root of its subtree.

To record the left-child $x$ instead of the root $1$ with relabeling the vertex $x$ by $1$ in each of (A2a)'s means 
to exchange $x$ and $1$ sequentially according to the minimal path from the root $1$ of the original tree $T$.

It is clear that all vertices in the minimal path in a tree become the right-to-left minimums in a permutation under $\omega$. So each right-to-left minimums is recorded in the position of the previous right-to-left minimums in a permutation obtained from a tree by $\omega$.
Since all elements are decreased by 1 in the last step, 
it satisfies \eqref{eq:move}, 
then Algorithm~A follows $\varphi \circ \omega$.
\end{proof}

\subsection{A new description of Gelineau et al.'s 
bijection $\psi: \Ank\to \Tnk$}
The  bijection $\psi$ between $\Ank$ and $\Tnk$  was 
constructed  as a composition of two bijections via the set $\mathcal{ES}_{n,k}$ of \emph{encoding sequences}  in \cite{GSZ11}.
In this section, we just give directly another description of this bijection $\psi$ from $\Ank$ to $\Tnk$, which does not use encoding sequences.

Given an increasing 1-2 tree $T \in T_n$, by convention,
if a vertex $x$ of $T$ has two children $u$, $v$ with $u<v$ then $u$ is the \emph{left child} and $v$ is the \emph{right child}.
By convention, if $x$ has only one child then it is the left child of $x$.

\subsection*{Algorithm~B}
Gelineau et al. described the bijection $\psi$ between alternating permutations and increasing 1-2 trees using the following algorithm.
Due to, for $n=1$ or $2$,
$$\abs{\A_n}=\abs{\T_n}=1,$$
we can define trivially $\psi$.
For $n\ge 3$, given $\pi \in \Ank$ ($k=\pi_1$), we define the mapping $\psi:\Ank \to \Tnk$ recursively as follows:
\begin{enumerate}[(B1)]
\item \label{case:a} If $\pi_2=k-1$, then define 
$\pi' = \pi'_1\pi'_2 \dots \pi'_{n-2} \in \A_{n-2,i-2}$ 
by deleting $k-1$ and $k$ from $\pi$ 
and relabeling by $[n-2]$ where $i> k$, that is, for all $1\le j \le n-2$
$$\pi'_{j} = 
\begin{cases}
\pi_{j+2},   &\text{if $\pi_{j+2}<k-1$,}\\
\pi_{j+2}-2, &\text{if $\pi_{j+2}>k$.}
\end{cases}
$$
We get $T'=\psi(\pi') \in \T_{n-2,i-2}$.
Relabel $T'$ by $\set{1,\dots,k-2,k+1,\dots,n}$ keeping in the order of labels, 
denoted by $T''$. Let $m$ be the smallest vertex greater than $k$ in the minimal path of $T''$ and $\ell$ the parent of $m$ in $T''$. Then insert a vertex $k-1$ in the middle of the edge $(\ell,m)$ and graft $k$ as the left-child of $k-1$.
$$
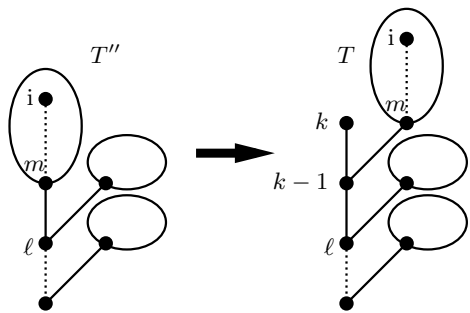
\centering
\begin{pgfpicture}{38.00mm}{29.20mm}{122.00mm}{73.20mm}
\pgfsetxvec{\pgfpoint{0.80mm}{0mm}}
\pgfsetyvec{\pgfpoint{0mm}{0.80mm}}
\color[rgb]{0,0,0}\pgfsetlinewidth{0.30mm}\pgfsetdash{}{0mm}
\pgfsetlinewidth{1.20mm}\pgfmoveto{\pgfxy(95.00,65.00)}\pgflineto{\pgfxy(105.00,65.00)}\pgfstroke
\pgfmoveto{\pgfxy(105.00,65.00)}\pgflineto{\pgfxy(102.20,65.70)}\pgflineto{\pgfxy(102.20,64.30)}\pgflineto{\pgfxy(105.00,65.00)}\pgfclosepath\pgffill
\pgfmoveto{\pgfxy(105.00,65.00)}\pgflineto{\pgfxy(102.20,65.70)}\pgflineto{\pgfxy(102.20,64.30)}\pgflineto{\pgfxy(105.00,65.00)}\pgfclosepath\pgfstroke
\pgfputat{\pgfxy(120.00,80.00)}{\pgfbox[bottom,left]{\fontsize{9.10}{10.93}\selectfont \makebox[0pt]{$T$}}}
\pgfsetdash{{0.30mm}{0.50mm}}{0mm}\pgfsetlinewidth{0.30mm}\pgfmoveto{\pgfxy(120.00,50.00)}\pgflineto{\pgfxy(120.00,40.00)}\pgfstroke
\pgfcircle[fill]{\pgfxy(120.00,50.00)}{0.80mm}
\pgfsetdash{}{0mm}\pgfcircle[stroke]{\pgfxy(120.00,50.00)}{0.80mm}
\pgfcircle[fill]{\pgfxy(120.00,40.00)}{0.80mm}
\pgfcircle[stroke]{\pgfxy(120.00,40.00)}{0.80mm}
\pgfmoveto{\pgfxy(120.00,50.00)}\pgflineto{\pgfxy(130.00,60.00)}\pgfstroke
\pgfcircle[fill]{\pgfxy(120.00,70.00)}{0.80mm}
\pgfcircle[stroke]{\pgfxy(120.00,70.00)}{0.80mm}
\pgfmoveto{\pgfxy(120.00,40.00)}\pgflineto{\pgfxy(130.00,50.00)}\pgfstroke
\pgfcircle[fill]{\pgfxy(130.00,50.00)}{0.80mm}
\pgfcircle[stroke]{\pgfxy(130.00,50.00)}{0.80mm}
\pgfputat{\pgfxy(117.00,69.00)}{\pgfbox[bottom,left]{\fontsize{9.10}{10.93}\selectfont \makebox[0pt][r]{$k$}}}
\pgfputat{\pgfxy(117.00,59.00)}{\pgfbox[bottom,left]{\fontsize{9.10}{10.93}\selectfont \makebox[0pt][r]{$k-1$}}}
\pgfcircle[fill]{\pgfxy(120.00,60.00)}{0.80mm}
\pgfcircle[stroke]{\pgfxy(120.00,60.00)}{0.80mm}
\pgfmoveto{\pgfxy(120.00,50.00)}\pgflineto{\pgfxy(120.00,60.00)}\pgfstroke
\pgfmoveto{\pgfxy(130.00,70.00)}\pgflineto{\pgfxy(120.00,60.00)}\pgfstroke
\pgfmoveto{\pgfxy(120.00,70.00)}\pgflineto{\pgfxy(120.00,60.00)}\pgfstroke
\pgfellipse[stroke]{\pgfxy(130.00,79.49)}{\pgfxy(6.00,0.00)}{\pgfxy(0.00,9.51)}
\pgfcircle[fill]{\pgfxy(130.00,70.00)}{0.80mm}
\pgfcircle[stroke]{\pgfxy(130.00,70.00)}{0.80mm}
\pgfputat{\pgfxy(130.00,72.00)}{\pgfbox[bottom,left]{\fontsize{9.10}{10.93}\selectfont \makebox[0pt][r]{$m$}}}
\pgfputat{\pgfxy(118.00,48.00)}{\pgfbox[bottom,left]{\fontsize{9.10}{10.93}\selectfont \makebox[0pt][r]{$\ell$}}}
\pgfcircle[fill]{\pgfxy(130.00,60.00)}{0.80mm}
\pgfcircle[stroke]{\pgfxy(130.00,60.00)}{0.80mm}
\pgfellipse[stroke]{\pgfxy(133.50,63.49)}{\pgfxy(6.50,0.00)}{\pgfxy(0.00,4.51)}
\pgfellipse[stroke]{\pgfxy(133.50,53.49)}{\pgfxy(6.50,0.00)}{\pgfxy(0.00,4.51)}
\pgfcircle[fill]{\pgfxy(130.00,84.00)}{0.80mm}
\pgfcircle[stroke]{\pgfxy(130.00,84.00)}{0.80mm}
\pgfsetdash{{0.30mm}{0.50mm}}{0mm}\pgfmoveto{\pgfxy(130.00,84.00)}\pgflineto{\pgfxy(130.00,70.00)}\pgfstroke
\pgfputat{\pgfxy(128.00,83.00)}{\pgfbox[bottom,left]{\fontsize{9.10}{10.93}\selectfont \makebox[0pt][r]{i}}}
\pgfputat{\pgfxy(80.00,80.00)}{\pgfbox[bottom,left]{\fontsize{9.10}{10.93}\selectfont \makebox[0pt]{$T''$}}}
\pgfmoveto{\pgfxy(70.00,50.00)}\pgflineto{\pgfxy(70.00,40.00)}\pgfstroke
\pgfcircle[fill]{\pgfxy(70.00,50.00)}{0.80mm}
\pgfsetdash{}{0mm}\pgfcircle[stroke]{\pgfxy(70.00,50.00)}{0.80mm}
\pgfcircle[fill]{\pgfxy(70.00,40.00)}{0.80mm}
\pgfcircle[stroke]{\pgfxy(70.00,40.00)}{0.80mm}
\pgfmoveto{\pgfxy(70.00,50.00)}\pgflineto{\pgfxy(80.00,60.00)}\pgfstroke
\pgfmoveto{\pgfxy(70.00,40.00)}\pgflineto{\pgfxy(80.00,50.00)}\pgfstroke
\pgfcircle[fill]{\pgfxy(80.00,50.00)}{0.80mm}
\pgfcircle[stroke]{\pgfxy(80.00,50.00)}{0.80mm}
\pgfmoveto{\pgfxy(70.00,50.00)}\pgflineto{\pgfxy(70.00,60.00)}\pgfstroke
\pgfellipse[stroke]{\pgfxy(70.00,69.49)}{\pgfxy(6.00,0.00)}{\pgfxy(0.00,9.51)}
\pgfcircle[fill]{\pgfxy(70.00,60.00)}{0.80mm}
\pgfcircle[stroke]{\pgfxy(70.00,60.00)}{0.80mm}
\pgfputat{\pgfxy(70.00,62.00)}{\pgfbox[bottom,left]{\fontsize{9.10}{10.93}\selectfont \makebox[0pt][r]{$m$}}}
\pgfputat{\pgfxy(68.00,48.00)}{\pgfbox[bottom,left]{\fontsize{9.10}{10.93}\selectfont \makebox[0pt][r]{$\ell$}}}
\pgfcircle[fill]{\pgfxy(80.00,60.00)}{0.80mm}
\pgfcircle[stroke]{\pgfxy(80.00,60.00)}{0.80mm}
\pgfellipse[stroke]{\pgfxy(83.50,63.49)}{\pgfxy(6.50,0.00)}{\pgfxy(0.00,4.51)}
\pgfellipse[stroke]{\pgfxy(83.50,53.49)}{\pgfxy(6.50,0.00)}{\pgfxy(0.00,4.51)}
\pgfcircle[fill]{\pgfxy(70.00,74.00)}{0.80mm}
\pgfcircle[stroke]{\pgfxy(70.00,74.00)}{0.80mm}
\pgfsetdash{{0.30mm}{0.50mm}}{0mm}\pgfmoveto{\pgfxy(70.00,74.00)}\pgflineto{\pgfxy(70.00,60.00)}\pgfstroke
\pgfputat{\pgfxy(68.00,73.00)}{\pgfbox[bottom,left]{\fontsize{9.10}{10.93}\selectfont \makebox[0pt][r]{i}}}
\end{pgfpicture}%
$$
We get the tree $T = \psi(\pi) \in \Tnk$.

\item \label{case:b} 
If $\pi_2<k-1$, then define $\pi'=(k-1~k)\pi \in \A_{n,k-1}$
by exchanging $k-1$ and $k$ in $\pi$. 
We get $T'=\psi(\pi') \in \T_{n,k-1}$.

\begin{enumerate}[(a)]
\item If $k$ is a not sibling of $k-1$ in $T'$, then we get the tree $T = \psi(\pi) \in \Tnk$ exchanging the labels $k-1$ and $k$ in $T'$. 
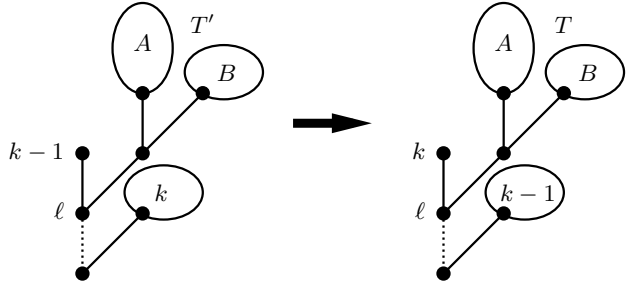
\label{case:b2}
$$
\centering
\begin{pgfpicture}{38.00mm}{29.20mm}{122.00mm}{70.00mm}
\pgfsetxvec{\pgfpoint{0.80mm}{0mm}}
\pgfsetyvec{\pgfpoint{0mm}{0.80mm}}
\color[rgb]{0,0,0}\pgfsetlinewidth{0.30mm}\pgfsetdash{}{0mm}
\pgfputat{\pgfxy(140.00,80.00)}{\pgfbox[bottom,left]{\fontsize{9.10}{10.93}\selectfont \makebox[0pt]{$T$}}}
\pgfsetlinewidth{1.20mm}\pgfmoveto{\pgfxy(95.00,65.00)}\pgflineto{\pgfxy(105.00,65.00)}\pgfstroke
\pgfmoveto{\pgfxy(105.00,65.00)}\pgflineto{\pgfxy(102.20,65.70)}\pgflineto{\pgfxy(102.20,64.30)}\pgflineto{\pgfxy(105.00,65.00)}\pgfclosepath\pgffill
\pgfmoveto{\pgfxy(105.00,65.00)}\pgflineto{\pgfxy(102.20,65.70)}\pgflineto{\pgfxy(102.20,64.30)}\pgflineto{\pgfxy(105.00,65.00)}\pgfclosepath\pgfstroke
\pgfputat{\pgfxy(80.00,80.00)}{\pgfbox[bottom,left]{\fontsize{9.10}{10.93}\selectfont \makebox[0pt]{$T'$}}}
\pgfputat{\pgfxy(57.00,59.00)}{\pgfbox[bottom,left]{\fontsize{9.10}{10.93}\selectfont \makebox[0pt][r]{$k-1$}}}
\pgfsetdash{{0.30mm}{0.50mm}}{0mm}\pgfsetlinewidth{0.30mm}\pgfmoveto{\pgfxy(60.00,50.00)}\pgflineto{\pgfxy(60.00,40.00)}\pgfstroke
\pgfcircle[fill]{\pgfxy(70.00,60.00)}{0.80mm}
\pgfsetdash{}{0mm}\pgfcircle[stroke]{\pgfxy(70.00,60.00)}{0.80mm}
\pgfcircle[fill]{\pgfxy(60.00,50.00)}{0.80mm}
\pgfcircle[stroke]{\pgfxy(60.00,50.00)}{0.80mm}
\pgfcircle[fill]{\pgfxy(60.00,40.00)}{0.80mm}
\pgfcircle[stroke]{\pgfxy(60.00,40.00)}{0.80mm}
\pgfmoveto{\pgfxy(70.00,60.00)}\pgflineto{\pgfxy(80.00,70.00)}\pgfstroke
\pgfmoveto{\pgfxy(60.00,40.00)}\pgflineto{\pgfxy(70.00,50.00)}\pgfstroke
\pgfcircle[fill]{\pgfxy(70.00,50.00)}{0.80mm}
\pgfcircle[stroke]{\pgfxy(70.00,50.00)}{0.80mm}
\pgfputat{\pgfxy(72.00,52.00)}{\pgfbox[bottom,left]{\fontsize{9.10}{10.93}\selectfont $k$}}
\pgfcircle[fill]{\pgfxy(60.00,60.00)}{0.80mm}
\pgfcircle[stroke]{\pgfxy(60.00,60.00)}{0.80mm}
\pgfmoveto{\pgfxy(60.00,50.00)}\pgflineto{\pgfxy(60.00,60.00)}\pgfstroke
\pgfputat{\pgfxy(57.00,49.00)}{\pgfbox[bottom,left]{\fontsize{9.10}{10.93}\selectfont \makebox[0pt][r]{$\ell$}}}
\pgfmoveto{\pgfxy(70.00,60.00)}\pgflineto{\pgfxy(60.00,50.00)}\pgfstroke
\pgfmoveto{\pgfxy(70.00,70.00)}\pgflineto{\pgfxy(70.00,60.00)}\pgfstroke
\pgfellipse[stroke]{\pgfxy(70.00,77.49)}{\pgfxy(5.00,0.00)}{\pgfxy(0.00,7.51)}
\pgfputat{\pgfxy(70.00,77.00)}{\pgfbox[bottom,left]{\fontsize{9.10}{10.93}\selectfont \makebox[0pt]{$A$}}}
\pgfcircle[fill]{\pgfxy(70.00,70.00)}{0.80mm}
\pgfcircle[stroke]{\pgfxy(70.00,70.00)}{0.80mm}
\pgfcircle[fill]{\pgfxy(80.00,70.00)}{0.80mm}
\pgfcircle[stroke]{\pgfxy(80.00,70.00)}{0.80mm}
\pgfputat{\pgfxy(84.00,72.00)}{\pgfbox[bottom,left]{\fontsize{9.10}{10.93}\selectfont \makebox[0pt]{$B$}}}
\pgfellipse[stroke]{\pgfxy(83.50,73.49)}{\pgfxy(6.50,0.00)}{\pgfxy(0.00,4.51)}
\pgfellipse[stroke]{\pgfxy(73.50,53.49)}{\pgfxy(6.50,0.00)}{\pgfxy(0.00,4.51)}
\pgfputat{\pgfxy(117.00,59.00)}{\pgfbox[bottom,left]{\fontsize{9.10}{10.93}\selectfont \makebox[0pt][r]{$k$}}}
\pgfsetdash{{0.30mm}{0.50mm}}{0mm}\pgfmoveto{\pgfxy(120.00,50.00)}\pgflineto{\pgfxy(120.00,40.00)}\pgfstroke
\pgfcircle[fill]{\pgfxy(130.00,60.00)}{0.80mm}
\pgfsetdash{}{0mm}\pgfcircle[stroke]{\pgfxy(130.00,60.00)}{0.80mm}
\pgfcircle[fill]{\pgfxy(120.00,50.00)}{0.80mm}
\pgfcircle[stroke]{\pgfxy(120.00,50.00)}{0.80mm}
\pgfcircle[fill]{\pgfxy(120.00,40.00)}{0.80mm}
\pgfcircle[stroke]{\pgfxy(120.00,40.00)}{0.80mm}
\pgfmoveto{\pgfxy(130.00,60.00)}\pgflineto{\pgfxy(140.00,70.00)}\pgfstroke
\pgfmoveto{\pgfxy(120.00,40.00)}\pgflineto{\pgfxy(130.00,50.00)}\pgfstroke
\pgfcircle[fill]{\pgfxy(130.00,50.00)}{0.80mm}
\pgfcircle[stroke]{\pgfxy(130.00,50.00)}{0.80mm}
\pgfputat{\pgfxy(134.00,52.00)}{\pgfbox[bottom,left]{\fontsize{9.10}{10.93}\selectfont \makebox[0pt]{$k-1$}}}
\pgfcircle[fill]{\pgfxy(120.00,60.00)}{0.80mm}
\pgfcircle[stroke]{\pgfxy(120.00,60.00)}{0.80mm}
\pgfmoveto{\pgfxy(120.00,50.00)}\pgflineto{\pgfxy(120.00,60.00)}\pgfstroke
\pgfputat{\pgfxy(117.00,49.00)}{\pgfbox[bottom,left]{\fontsize{9.10}{10.93}\selectfont \makebox[0pt][r]{$\ell$}}}
\pgfmoveto{\pgfxy(130.00,60.00)}\pgflineto{\pgfxy(120.00,50.00)}\pgfstroke
\pgfmoveto{\pgfxy(130.00,70.00)}\pgflineto{\pgfxy(130.00,60.00)}\pgfstroke
\pgfellipse[stroke]{\pgfxy(130.00,77.49)}{\pgfxy(5.00,0.00)}{\pgfxy(0.00,7.51)}
\pgfcircle[fill]{\pgfxy(130.00,70.00)}{0.80mm}
\pgfcircle[stroke]{\pgfxy(130.00,70.00)}{0.80mm}
\pgfcircle[fill]{\pgfxy(140.00,70.00)}{0.80mm}
\pgfcircle[stroke]{\pgfxy(140.00,70.00)}{0.80mm}
\pgfellipse[stroke]{\pgfxy(143.50,73.49)}{\pgfxy(6.50,0.00)}{\pgfxy(0.00,4.51)}
\pgfellipse[stroke]{\pgfxy(133.50,53.49)}{\pgfxy(6.50,0.00)}{\pgfxy(0.00,4.51)}
\pgfputat{\pgfxy(130.00,77.00)}{\pgfbox[bottom,left]{\fontsize{9.10}{10.93}\selectfont \makebox[0pt]{$A$}}}
\pgfputat{\pgfxy(144.00,72.00)}{\pgfbox[bottom,left]{\fontsize{9.10}{10.93}\selectfont \makebox[0pt]{$B$}}}
\end{pgfpicture}%
$$

\item If $k$ is a sibling of $k-1$ in $T'$, then we get the tree $T = \psi(\pi) \in \Tnk$ modifying as follows: 
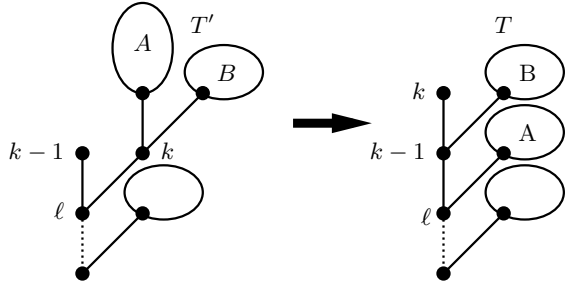
\label{case:b1}
$$
\centering
\begin{pgfpicture}{38.00mm}{29.20mm}{122.00mm}{70.00mm}
\pgfsetxvec{\pgfpoint{0.80mm}{0mm}}
\pgfsetyvec{\pgfpoint{0mm}{0.80mm}}
\color[rgb]{0,0,0}\pgfsetlinewidth{0.30mm}\pgfsetdash{}{0mm}
\pgfputat{\pgfxy(130.00,80.00)}{\pgfbox[bottom,left]{\fontsize{9.10}{10.93}\selectfont \makebox[0pt]{$T$}}}
\pgfsetdash{{0.30mm}{0.50mm}}{0mm}\pgfmoveto{\pgfxy(120.00,50.00)}\pgflineto{\pgfxy(120.00,40.00)}\pgfstroke
\pgfcircle[fill]{\pgfxy(120.00,50.00)}{0.80mm}
\pgfsetdash{}{0mm}\pgfcircle[stroke]{\pgfxy(120.00,50.00)}{0.80mm}
\pgfcircle[fill]{\pgfxy(120.00,40.00)}{0.80mm}
\pgfcircle[stroke]{\pgfxy(120.00,40.00)}{0.80mm}
\pgfmoveto{\pgfxy(120.00,50.00)}\pgflineto{\pgfxy(130.00,60.00)}\pgfstroke
\pgfcircle[fill]{\pgfxy(120.00,70.00)}{0.80mm}
\pgfcircle[stroke]{\pgfxy(120.00,70.00)}{0.80mm}
\pgfmoveto{\pgfxy(120.00,40.00)}\pgflineto{\pgfxy(130.00,50.00)}\pgfstroke
\pgfcircle[fill]{\pgfxy(130.00,50.00)}{0.80mm}
\pgfcircle[stroke]{\pgfxy(130.00,50.00)}{0.80mm}
\pgfputat{\pgfxy(117.00,69.00)}{\pgfbox[bottom,left]{\fontsize{9.10}{10.93}\selectfont \makebox[0pt][r]{$k$}}}
\pgfputat{\pgfxy(117.00,59.00)}{\pgfbox[bottom,left]{\fontsize{9.10}{10.93}\selectfont \makebox[0pt][r]{$k-1$}}}
\pgfcircle[fill]{\pgfxy(120.00,60.00)}{0.80mm}
\pgfcircle[stroke]{\pgfxy(120.00,60.00)}{0.80mm}
\pgfmoveto{\pgfxy(120.00,50.00)}\pgflineto{\pgfxy(120.00,60.00)}\pgfstroke
\pgfmoveto{\pgfxy(130.00,70.00)}\pgflineto{\pgfxy(120.00,60.00)}\pgfstroke
\pgfmoveto{\pgfxy(120.00,70.00)}\pgflineto{\pgfxy(120.00,60.00)}\pgfstroke
\pgfellipse[stroke]{\pgfxy(133.50,73.49)}{\pgfxy(6.50,0.00)}{\pgfxy(0.00,4.51)}
\pgfputat{\pgfxy(134.00,72.00)}{\pgfbox[bottom,left]{\fontsize{9.10}{10.93}\selectfont \makebox[0pt]{B}}}
\pgfcircle[fill]{\pgfxy(130.00,70.00)}{0.80mm}
\pgfcircle[stroke]{\pgfxy(130.00,70.00)}{0.80mm}
\pgfputat{\pgfxy(118.00,48.00)}{\pgfbox[bottom,left]{\fontsize{9.10}{10.93}\selectfont \makebox[0pt][r]{$\ell$}}}
\pgfcircle[fill]{\pgfxy(130.00,60.00)}{0.80mm}
\pgfcircle[stroke]{\pgfxy(130.00,60.00)}{0.80mm}
\pgfputat{\pgfxy(134.00,62.00)}{\pgfbox[bottom,left]{\fontsize{9.10}{10.93}\selectfont \makebox[0pt]{A}}}
\pgfellipse[stroke]{\pgfxy(133.50,63.49)}{\pgfxy(6.50,0.00)}{\pgfxy(0.00,4.51)}
\pgfellipse[stroke]{\pgfxy(133.50,53.49)}{\pgfxy(6.50,0.00)}{\pgfxy(0.00,4.51)}
\pgfsetlinewidth{1.20mm}\pgfmoveto{\pgfxy(95.00,65.00)}\pgflineto{\pgfxy(105.00,65.00)}\pgfstroke
\pgfmoveto{\pgfxy(105.00,65.00)}\pgflineto{\pgfxy(102.20,65.70)}\pgflineto{\pgfxy(102.20,64.30)}\pgflineto{\pgfxy(105.00,65.00)}\pgfclosepath\pgffill
\pgfmoveto{\pgfxy(105.00,65.00)}\pgflineto{\pgfxy(102.20,65.70)}\pgflineto{\pgfxy(102.20,64.30)}\pgflineto{\pgfxy(105.00,65.00)}\pgfclosepath\pgfstroke
\pgfputat{\pgfxy(80.00,80.00)}{\pgfbox[bottom,left]{\fontsize{9.10}{10.93}\selectfont \makebox[0pt]{$T'$}}}
\pgfputat{\pgfxy(57.00,59.00)}{\pgfbox[bottom,left]{\fontsize{9.10}{10.93}\selectfont \makebox[0pt][r]{$k-1$}}}
\pgfsetdash{{0.30mm}{0.50mm}}{0mm}\pgfsetlinewidth{0.30mm}\pgfmoveto{\pgfxy(60.00,50.00)}\pgflineto{\pgfxy(60.00,40.00)}\pgfstroke
\pgfcircle[fill]{\pgfxy(70.00,60.00)}{0.80mm}
\pgfsetdash{}{0mm}\pgfcircle[stroke]{\pgfxy(70.00,60.00)}{0.80mm}
\pgfcircle[fill]{\pgfxy(60.00,50.00)}{0.80mm}
\pgfcircle[stroke]{\pgfxy(60.00,50.00)}{0.80mm}
\pgfcircle[fill]{\pgfxy(60.00,40.00)}{0.80mm}
\pgfcircle[stroke]{\pgfxy(60.00,40.00)}{0.80mm}
\pgfmoveto{\pgfxy(70.00,60.00)}\pgflineto{\pgfxy(80.00,70.00)}\pgfstroke
\pgfmoveto{\pgfxy(60.00,40.00)}\pgflineto{\pgfxy(70.00,50.00)}\pgfstroke
\pgfcircle[fill]{\pgfxy(70.00,50.00)}{0.80mm}
\pgfcircle[stroke]{\pgfxy(70.00,50.00)}{0.80mm}
\pgfputat{\pgfxy(73.00,59.00)}{\pgfbox[bottom,left]{\fontsize{9.10}{10.93}\selectfont $k$}}
\pgfcircle[fill]{\pgfxy(60.00,60.00)}{0.80mm}
\pgfcircle[stroke]{\pgfxy(60.00,60.00)}{0.80mm}
\pgfmoveto{\pgfxy(60.00,50.00)}\pgflineto{\pgfxy(60.00,60.00)}\pgfstroke
\pgfputat{\pgfxy(57.00,49.00)}{\pgfbox[bottom,left]{\fontsize{9.10}{10.93}\selectfont \makebox[0pt][r]{$\ell$}}}
\pgfmoveto{\pgfxy(70.00,60.00)}\pgflineto{\pgfxy(60.00,50.00)}\pgfstroke
\pgfmoveto{\pgfxy(70.00,70.00)}\pgflineto{\pgfxy(70.00,60.00)}\pgfstroke
\pgfellipse[stroke]{\pgfxy(70.00,77.49)}{\pgfxy(5.00,0.00)}{\pgfxy(0.00,7.51)}
\pgfputat{\pgfxy(70.00,77.00)}{\pgfbox[bottom,left]{\fontsize{9.10}{10.93}\selectfont \makebox[0pt]{$A$}}}
\pgfcircle[fill]{\pgfxy(70.00,70.00)}{0.80mm}
\pgfcircle[stroke]{\pgfxy(70.00,70.00)}{0.80mm}
\pgfcircle[fill]{\pgfxy(80.00,70.00)}{0.80mm}
\pgfcircle[stroke]{\pgfxy(80.00,70.00)}{0.80mm}
\pgfputat{\pgfxy(84.00,72.00)}{\pgfbox[bottom,left]{\fontsize{9.10}{10.93}\selectfont \makebox[0pt]{$B$}}}
\pgfellipse[stroke]{\pgfxy(83.50,73.49)}{\pgfxy(6.50,0.00)}{\pgfxy(0.00,4.51)}
\pgfellipse[stroke]{\pgfxy(73.50,53.49)}{\pgfxy(6.50,0.00)}{\pgfxy(0.00,4.51)}
\end{pgfpicture}%
$$
\end{enumerate}
\end{enumerate}


\subsection*{Algorithm~C}
We define another bijection $\psi' : \Ank \to \Tnk$ by the following algorithm.
Given $\sigma = \sigma_1 \dots \sigma_n \in \Ank$, 
denote by $d_i(\sigma) = (\sigma_{2i-1}, \sigma_{2i})$ with $\sigma_{2i-1} > \sigma_{2i}$
for $1 \le i \le m$, 
where $m = \lfloor (n+1)/2\rfloor$, but $d_{m}(\sigma) = (\sigma_{n})$ if $n$ is odd.
We shall construct a series of  trees $T^{(m)}, \ldots, T^{(1)}$
by reading the pairs $d_m(\sigma)$, \dots, $d_1(\sigma)$ in this order.

If $n$ is odd, so $d_{m} (\sigma) = (\sigma_n)$, 
define $T^{(m)}$ to be the tree with only root $\sigma_n$.
If $n$ is even, so $d_{m}(\sigma) = (\sigma_{n-1}, \sigma_{n})$ with $\sigma_{n-1} > \sigma_{n}$, 
define $T^{(m)}$ to be  the tree with the root $\sigma_{n}$ and its left child $\sigma_{n-1}$.
We note that the vertex $\sigma_{2m-1}$ is the minimal leaf of $T^{(m)}$.

For $1 \le i \le m-1$, 
assuming that we have a tree $T^{(i+1)}$ of which the minimal leaf is $\sigma_{2i+1}$, 
read $d_i(\sigma) = (\sigma_{2i-1}, \sigma_{2i})$.
As $\sigma_{2i} < \sigma_{2i+1}$, 
the smallest vertex, say $a^{(i)}$, greater than $\sigma_{2i}$ in the minimal path of $T^{(i+1)}$ is well-defined.
By removing all left edges from the increasing 1-2 tree $T^{(i+1)}$, we have several paths, each connected component of which is called a \emph{maximal path} consisting only right edges.
\begin{enumerate}[(C1)]
\item 
If $a^{(i)} < \sigma_{2i-1}$, 
the largest vertex, say $b^{(i)}$, less than $\sigma_{2i-1}$ in the maximal path from $a^{(i)}$ of $T^{(i+1)}$ is well-defined.
For some $j\ge 1$, let $(v_1, v_1, \ldots, v_j)$ be the path
from $a^{(i)}$ to $b^{(i)}$ with $v_1=a^{(i)}$ and $v_j=b^{(i)}$.
The vertices $v_1, \dots, v_{j-1}$ should have left child $u_1, \dots, u_{j-1}$, but $v_j$ may not, with
\begin{align*}
v_1 < u_1 < v_2 < u_2 < \dots < v_{j-1} < u_{j-1} < v_j 
\end{align*}
Decomposing by the maximal path from $a^{(i)}$ to $b^{(i)}$,
we write $S_1, \ldots, S_{j}$ the left-subtrees of the vertices $v_1, \ldots, v_{j}$
and $S_{j+1}$  the right-subtree of $b^{(i)}$. 
Since each of $u_1, \dots, u_{j-1}$ lies on each of $S_1, \dots, S_{j-1}$, 
$S_1, \dots, S_{j-1}$ should not be empty, but two trees $S_{j}$ and $S_{j+1}$ may be empty.  
Then, the tree $T^{(i)}$ is obtained from $T^{(i+1)}$ by the following operations:
\begin{itemize}
\item Graft $\sigma_{2i}$ so that $a^{(i)}$ is its left-child;
\item Flip the tree at vertex $a^{(i)}$;
\item Transplant the trees $S_1, \ldots, S_{j+1}$ as right-subtrees of
the vertices $\sigma_{2i}$, $a^{(i)}$, $v_2, \ldots, v_{j-1}$, $b^{(i)}$;
\item Graft $\sigma_{2i-1}$ as the left-child of $b^{(i)}$.
\end{itemize}
We can illustrate the above transformation by 
$$
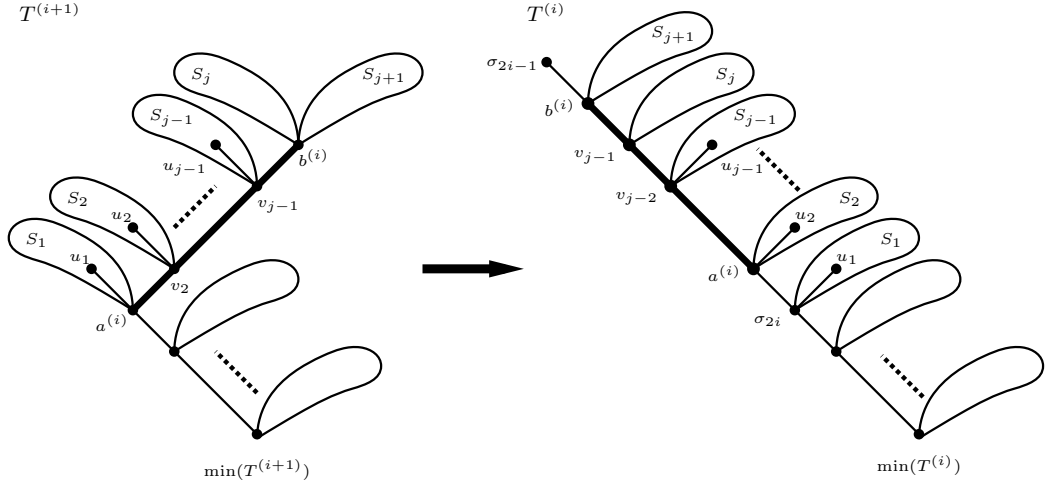
\centering
\begin{pgfpicture}{-7.50mm}{-13.48mm}{134.00mm}{53.66mm}
\pgfsetxvec{\pgfpoint{0.55mm}{0mm}}
\pgfsetyvec{\pgfpoint{0mm}{0.55mm}}
\color[rgb]{0,0,0}\pgfsetlinewidth{0.30mm}\pgfsetdash{}{0mm}
\pgfcircle[fill]{\pgfxy(20.00,20.00)}{0.55mm}
\pgfcircle[stroke]{\pgfxy(20.00,20.00)}{0.55mm}
\pgfcircle[fill]{\pgfxy(30.00,30.00)}{0.55mm}
\pgfcircle[stroke]{\pgfxy(25.00,39.00)}{0.00mm}
\pgfmoveto{\pgfxy(20.00,20.00)}\pgflineto{\pgfxy(30.00,10.00)}\pgfstroke
\pgfmoveto{\pgfxy(30.00,10.00)}\pgflineto{\pgfxy(50.00,-10.00)}\pgfstroke
\pgfcircle[fill]{\pgfxy(30.00,10.00)}{0.55mm}
\pgfcircle[stroke]{\pgfxy(30.00,10.00)}{0.55mm}
\pgfcircle[fill]{\pgfxy(50.00,-10.00)}{0.55mm}
\pgfcircle[stroke]{\pgfxy(50.00,-10.00)}{0.55mm}
\pgfputat{\pgfxy(19.00,15.00)}{\pgfbox[bottom,left]{\fontsize{6.26}{7.51}\selectfont \makebox[0pt][r]{$a^{(i)}$}}}
\pgfputat{\pgfxy(50.00,-20.00)}{\pgfbox[bottom,left]{\fontsize{6.26}{7.51}\selectfont \makebox[0pt]{$\min(T^{(i+1)})$}}}
\pgfsetlinewidth{0.90mm}\pgfmoveto{\pgfxy(20.00,20.00)}\pgflineto{\pgfxy(60.00,60.00)}\pgfstroke
\pgfsetlinewidth{0.30mm}\pgfmoveto{\pgfxy(20.00,20.00)}\pgfcurveto{\pgfxy(16.71,22.07)}{\pgfxy(13.37,24.07)}{\pgfxy(10.00,26.00)}\pgfcurveto{\pgfxy(5.85,28.38)}{\pgfxy(1.61,30.67)}{\pgfxy(-3.00,32.00)}\pgfcurveto{\pgfxy(-6.32,32.96)}{\pgfxy(-10.00,34.27)}{\pgfxy(-10.00,37.50)}\pgfcurveto{\pgfxy(-10.00,39.99)}{\pgfxy(-7.62,41.65)}{\pgfxy(-5.00,42.00)}\pgfcurveto{\pgfxy(-1.89,42.42)}{\pgfxy(1.15,41.30)}{\pgfxy(4.00,40.00)}\pgfcurveto{\pgfxy(9.08,37.69)}{\pgfxy(13.95,34.68)}{\pgfxy(17.00,30.00)}\pgfcurveto{\pgfxy(18.94,27.02)}{\pgfxy(19.98,23.55)}{\pgfxy(20.00,20.00)}\pgfstroke
\pgfputat{\pgfxy(-6.00,36.00)}{\pgfbox[bottom,left]{\fontsize{6.26}{7.51}\selectfont $S_1$}}
\pgfmoveto{\pgfxy(20.00,40.00)}\pgflineto{\pgfxy(30.00,30.00)}\pgfstroke
\pgfcircle[fill]{\pgfxy(60.00,60.00)}{0.55mm}
\pgfcircle[stroke]{\pgfxy(60.00,60.00)}{0.55mm}
\pgfputat{\pgfxy(86.00,76.00)}{\pgfbox[bottom,left]{\fontsize{6.26}{7.51}\selectfont \makebox[0pt][r]{$S_{j+1}$}}}
\pgfputat{\pgfxy(60.00,55.00)}{\pgfbox[bottom,left]{\fontsize{6.26}{7.51}\selectfont $b^{(i)}$}}
\pgfputat{\pgfxy(34.00,76.00)}{\pgfbox[bottom,left]{\fontsize{6.26}{7.51}\selectfont $S_{j}$}}
\pgfsetdash{{0.60mm}{0.50mm}}{0mm}\pgfsetlinewidth{0.60mm}\pgfmoveto{\pgfxy(30.00,40.00)}\pgflineto{\pgfxy(40.00,50.00)}\pgfstroke
\pgfmoveto{\pgfxy(50.00,0.00)}\pgflineto{\pgfxy(40.00,10.00)}\pgfstroke
\pgfputat{\pgfxy(0.00,90.00)}{\pgfbox[bottom,left]{\fontsize{7.82}{9.39}\selectfont \makebox[0pt]{$T^{(i+1)}$}}}
\pgfcircle[fill]{\pgfxy(180.00,20.00)}{0.55mm}
\pgfsetdash{}{0mm}\pgfsetlinewidth{0.30mm}\pgfcircle[stroke]{\pgfxy(180.00,20.00)}{0.55mm}
\pgfcircle[stroke]{\pgfxy(165.00,39.00)}{0.00mm}
\pgfmoveto{\pgfxy(180.00,20.00)}\pgflineto{\pgfxy(190.00,10.00)}\pgfstroke
\pgfputat{\pgfxy(167.00,27.00)}{\pgfbox[bottom,left]{\fontsize{6.26}{7.51}\selectfont \makebox[0pt][r]{$a^{(i)}$}}}
\pgfsetlinewidth{0.90mm}\pgfmoveto{\pgfxy(170.00,30.00)}\pgflineto{\pgfxy(130.00,70.00)}\pgfstroke
\pgfputat{\pgfxy(120.00,90.00)}{\pgfbox[bottom,left]{\fontsize{7.82}{9.39}\selectfont \makebox[0pt]{$T^{(i)}$}}}
\pgfputat{\pgfxy(206.00,36.00)}{\pgfbox[bottom,left]{\fontsize{6.26}{7.51}\selectfont \makebox[0pt][r]{$S_1$}}}
\pgfsetdash{{0.60mm}{0.50mm}}{0mm}\pgfsetlinewidth{0.60mm}\pgfmoveto{\pgfxy(181.00,49.00)}\pgflineto{\pgfxy(171.00,59.00)}\pgfstroke
\pgfputat{\pgfxy(177.00,17.00)}{\pgfbox[bottom,left]{\fontsize{6.26}{7.51}\selectfont \makebox[0pt][r]{$\sigma_{2i}$}}}
\pgfputat{\pgfxy(127.00,67.00)}{\pgfbox[bottom,left]{\fontsize{6.26}{7.51}\selectfont \makebox[0pt][r]{$b^{(i)}$}}}
\pgfputat{\pgfxy(156.00,86.00)}{\pgfbox[bottom,left]{\fontsize{6.26}{7.51}\selectfont \makebox[0pt][r]{$S_{j+1}$}}}
\pgfcircle[fill]{\pgfxy(130.00,70.00)}{0.55mm}
\pgfsetdash{}{0mm}\pgfcircle[stroke]{\pgfxy(130.00,70.00)}{0.55mm}
\pgfcircle[fill]{\pgfxy(140.00,60.00)}{0.55mm}
\pgfcircle[stroke]{\pgfxy(140.00,60.00)}{0.55mm}
\pgfputat{\pgfxy(166.00,76.00)}{\pgfbox[bottom,left]{\fontsize{6.26}{7.51}\selectfont \makebox[0pt][r]{$S_{j}$}}}
\pgfcircle[fill]{\pgfxy(120.00,80.00)}{0.55mm}
\pgfsetlinewidth{0.30mm}\pgfcircle[stroke]{\pgfxy(120.00,80.00)}{0.55mm}
\pgfputat{\pgfxy(118.00,79.00)}{\pgfbox[bottom,left]{\fontsize{6.26}{7.51}\selectfont \makebox[0pt][r]{$\sigma_{2i-1}$}}}
\pgfsetlinewidth{1.20mm}\pgfmoveto{\pgfxy(90.00,30.00)}\pgflineto{\pgfxy(110.00,30.00)}\pgfstroke
\pgfmoveto{\pgfxy(110.00,30.00)}\pgflineto{\pgfxy(107.20,30.70)}\pgflineto{\pgfxy(107.20,29.30)}\pgflineto{\pgfxy(110.00,30.00)}\pgfclosepath\pgffill
\pgfmoveto{\pgfxy(110.00,30.00)}\pgflineto{\pgfxy(107.20,30.70)}\pgflineto{\pgfxy(107.20,29.30)}\pgflineto{\pgfxy(110.00,30.00)}\pgfclosepath\pgfstroke
\pgfsetlinewidth{0.30mm}\pgfmoveto{\pgfxy(10.00,30.00)}\pgflineto{\pgfxy(20.00,20.00)}\pgfstroke
\pgfmoveto{\pgfxy(120.00,80.00)}\pgflineto{\pgfxy(180.00,20.00)}\pgfstroke
\pgfcircle[fill]{\pgfxy(50.00,50.00)}{0.55mm}
\pgfcircle[stroke]{\pgfxy(50.00,50.00)}{0.55mm}
\pgfmoveto{\pgfxy(40.00,60.00)}\pgflineto{\pgfxy(50.00,50.00)}\pgfstroke
\pgfcircle[fill]{\pgfxy(40.00,60.00)}{0.55mm}
\pgfcircle[stroke]{\pgfxy(40.00,60.00)}{0.55mm}
\pgfmoveto{\pgfxy(190.00,30.00)}\pgflineto{\pgfxy(180.00,20.00)}\pgfstroke
\pgfcircle[fill]{\pgfxy(190.00,30.00)}{0.55mm}
\pgfcircle[stroke]{\pgfxy(190.00,30.00)}{0.55mm}
\pgfputat{\pgfxy(196.00,46.00)}{\pgfbox[bottom,left]{\fontsize{6.26}{7.51}\selectfont \makebox[0pt][r]{$S_2$}}}
\pgfcircle[fill]{\pgfxy(170.00,30.00)}{0.55mm}
\pgfsetlinewidth{0.60mm}\pgfcircle[stroke]{\pgfxy(170.00,30.00)}{0.55mm}
\pgfsetlinewidth{0.30mm}\pgfmoveto{\pgfxy(180.00,40.00)}\pgflineto{\pgfxy(170.00,30.00)}\pgfstroke
\pgfcircle[fill]{\pgfxy(180.00,40.00)}{0.55mm}
\pgfcircle[stroke]{\pgfxy(180.00,40.00)}{0.55mm}
\pgfputat{\pgfxy(176.00,66.00)}{\pgfbox[bottom,left]{\fontsize{6.26}{7.51}\selectfont \makebox[0pt][r]{$S_{j-1}$}}}
\pgfcircle[fill]{\pgfxy(150.00,50.00)}{0.55mm}
\pgfsetlinewidth{0.60mm}\pgfcircle[stroke]{\pgfxy(150.00,50.00)}{0.55mm}
\pgfsetlinewidth{0.30mm}\pgfmoveto{\pgfxy(160.00,60.00)}\pgflineto{\pgfxy(150.00,50.00)}\pgfstroke
\pgfcircle[fill]{\pgfxy(160.00,60.00)}{0.55mm}
\pgfcircle[stroke]{\pgfxy(160.00,60.00)}{0.55mm}
\pgfputat{\pgfxy(24.00,66.00)}{\pgfbox[bottom,left]{\fontsize{6.26}{7.51}\selectfont $S_{j-1}$}}
\pgfputat{\pgfxy(4.00,46.00)}{\pgfbox[bottom,left]{\fontsize{6.26}{7.51}\selectfont $S_2$}}
\pgfcircle[fill]{\pgfxy(190.00,10.00)}{0.55mm}
\pgfcircle[stroke]{\pgfxy(190.00,10.00)}{0.55mm}
\pgfcircle[fill]{\pgfxy(210.00,-10.00)}{0.55mm}
\pgfcircle[stroke]{\pgfxy(210.00,-10.00)}{0.55mm}
\pgfputat{\pgfxy(210.00,-19.00)}{\pgfbox[bottom,left]{\fontsize{6.26}{7.51}\selectfont \makebox[0pt]{$\min(T^{(i)})$}}}
\pgfsetdash{{0.60mm}{0.50mm}}{0mm}\pgfsetlinewidth{0.60mm}\pgfmoveto{\pgfxy(211.00,-1.00)}\pgflineto{\pgfxy(201.00,9.00)}\pgfstroke
\pgfsetdash{}{0mm}\pgfsetlinewidth{0.30mm}\pgfmoveto{\pgfxy(190.00,10.00)}\pgflineto{\pgfxy(210.00,-10.00)}\pgfstroke
\pgfcircle[fill]{\pgfxy(30.00,30.00)}{0.55mm}
\pgfcircle[stroke]{\pgfxy(30.00,30.00)}{0.55mm}
\pgfcircle[fill]{\pgfxy(20.00,40.00)}{0.55mm}
\pgfcircle[stroke]{\pgfxy(20.00,40.00)}{0.55mm}
\pgfcircle[fill]{\pgfxy(10.00,30.00)}{0.55mm}
\pgfcircle[stroke]{\pgfxy(10.00,30.00)}{0.55mm}
\pgfputat{\pgfxy(10.00,32.00)}{\pgfbox[bottom,left]{\fontsize{6.26}{7.51}\selectfont \makebox[0pt][r]{$u_1$}}}
\pgfputat{\pgfxy(38.00,54.00)}{\pgfbox[bottom,left]{\fontsize{6.26}{7.51}\selectfont \makebox[0pt][r]{$u_{j-1}$}}}
\pgfputat{\pgfxy(20.00,42.00)}{\pgfbox[bottom,left]{\fontsize{6.26}{7.51}\selectfont \makebox[0pt][r]{$u_2$}}}
\pgfputat{\pgfxy(29.00,25.00)}{\pgfbox[bottom,left]{\fontsize{6.26}{7.51}\selectfont $v_2$}}
\pgfputat{\pgfxy(50.00,45.00)}{\pgfbox[bottom,left]{\fontsize{6.26}{7.51}\selectfont $v_{j-1}$}}
\pgfputat{\pgfxy(190.00,32.00)}{\pgfbox[bottom,left]{\fontsize{6.26}{7.51}\selectfont $u_1$}}
\pgfputat{\pgfxy(180.00,42.00)}{\pgfbox[bottom,left]{\fontsize{6.26}{7.51}\selectfont $u_2$}}
\pgfputat{\pgfxy(162.00,54.00)}{\pgfbox[bottom,left]{\fontsize{6.26}{7.51}\selectfont $u_{j-1}$}}
\pgfputat{\pgfxy(137.00,57.00)}{\pgfbox[bottom,left]{\fontsize{6.26}{7.51}\selectfont \makebox[0pt][r]{$v_{j-1}$}}}
\pgfputat{\pgfxy(147.00,47.00)}{\pgfbox[bottom,left]{\fontsize{6.26}{7.51}\selectfont \makebox[0pt][r]{$v_{j-2}$}}}
\pgfmoveto{\pgfxy(50.00,50.00)}\pgfcurveto{\pgfxy(46.71,52.07)}{\pgfxy(43.37,54.07)}{\pgfxy(40.00,56.00)}\pgfcurveto{\pgfxy(35.85,58.38)}{\pgfxy(31.61,60.67)}{\pgfxy(27.00,62.00)}\pgfcurveto{\pgfxy(23.68,62.96)}{\pgfxy(20.00,64.27)}{\pgfxy(20.00,67.50)}\pgfcurveto{\pgfxy(20.00,69.99)}{\pgfxy(22.38,71.65)}{\pgfxy(25.00,72.00)}\pgfcurveto{\pgfxy(28.11,72.42)}{\pgfxy(31.15,71.30)}{\pgfxy(34.00,70.00)}\pgfcurveto{\pgfxy(39.08,67.69)}{\pgfxy(43.95,64.68)}{\pgfxy(47.00,60.00)}\pgfcurveto{\pgfxy(48.94,57.02)}{\pgfxy(49.98,53.55)}{\pgfxy(50.00,50.00)}\pgfstroke
\pgfmoveto{\pgfxy(30.00,30.00)}\pgfcurveto{\pgfxy(26.71,32.07)}{\pgfxy(23.37,34.07)}{\pgfxy(20.00,36.00)}\pgfcurveto{\pgfxy(15.85,38.38)}{\pgfxy(11.61,40.67)}{\pgfxy(7.00,42.00)}\pgfcurveto{\pgfxy(3.68,42.96)}{\pgfxy(0.00,44.27)}{\pgfxy(0.00,47.50)}\pgfcurveto{\pgfxy(0.00,49.99)}{\pgfxy(2.38,51.65)}{\pgfxy(5.00,52.00)}\pgfcurveto{\pgfxy(8.11,52.42)}{\pgfxy(11.15,51.30)}{\pgfxy(14.00,50.00)}\pgfcurveto{\pgfxy(19.08,47.69)}{\pgfxy(23.95,44.68)}{\pgfxy(27.00,40.00)}\pgfcurveto{\pgfxy(28.94,37.02)}{\pgfxy(29.98,33.55)}{\pgfxy(30.00,30.00)}\pgfstroke
\pgfmoveto{\pgfxy(60.00,60.00)}\pgfcurveto{\pgfxy(56.71,62.07)}{\pgfxy(53.37,64.07)}{\pgfxy(50.00,66.00)}\pgfcurveto{\pgfxy(45.85,68.38)}{\pgfxy(41.61,70.67)}{\pgfxy(37.00,72.00)}\pgfcurveto{\pgfxy(33.68,72.96)}{\pgfxy(30.00,74.27)}{\pgfxy(30.00,77.50)}\pgfcurveto{\pgfxy(30.00,79.99)}{\pgfxy(32.38,81.65)}{\pgfxy(35.00,82.00)}\pgfcurveto{\pgfxy(38.11,82.42)}{\pgfxy(41.15,81.30)}{\pgfxy(44.00,80.00)}\pgfcurveto{\pgfxy(49.08,77.69)}{\pgfxy(53.95,74.68)}{\pgfxy(57.00,70.00)}\pgfcurveto{\pgfxy(58.94,67.02)}{\pgfxy(59.98,63.55)}{\pgfxy(60.00,60.00)}\pgfstroke
\pgfmoveto{\pgfxy(210.00,-11.00)}\pgfcurveto{\pgfxy(213.29,-8.93)}{\pgfxy(216.63,-6.93)}{\pgfxy(220.00,-5.00)}\pgfcurveto{\pgfxy(224.15,-2.62)}{\pgfxy(228.39,-0.33)}{\pgfxy(233.00,1.00)}\pgfcurveto{\pgfxy(236.32,1.96)}{\pgfxy(240.00,3.27)}{\pgfxy(240.00,6.50)}\pgfcurveto{\pgfxy(240.00,8.99)}{\pgfxy(237.62,10.65)}{\pgfxy(235.00,11.00)}\pgfcurveto{\pgfxy(231.89,11.42)}{\pgfxy(228.85,10.30)}{\pgfxy(226.00,9.00)}\pgfcurveto{\pgfxy(220.92,6.69)}{\pgfxy(216.05,3.68)}{\pgfxy(213.00,-1.00)}\pgfcurveto{\pgfxy(211.06,-3.98)}{\pgfxy(210.02,-7.45)}{\pgfxy(210.00,-11.00)}\pgfstroke
\pgfmoveto{\pgfxy(180.00,20.00)}\pgfcurveto{\pgfxy(183.29,22.07)}{\pgfxy(186.63,24.07)}{\pgfxy(190.00,26.00)}\pgfcurveto{\pgfxy(194.15,28.38)}{\pgfxy(198.39,30.67)}{\pgfxy(203.00,32.00)}\pgfcurveto{\pgfxy(206.32,32.96)}{\pgfxy(210.00,34.27)}{\pgfxy(210.00,37.50)}\pgfcurveto{\pgfxy(210.00,39.99)}{\pgfxy(207.62,41.65)}{\pgfxy(205.00,42.00)}\pgfcurveto{\pgfxy(201.89,42.42)}{\pgfxy(198.85,41.30)}{\pgfxy(196.00,40.00)}\pgfcurveto{\pgfxy(190.92,37.69)}{\pgfxy(186.05,34.68)}{\pgfxy(183.00,30.00)}\pgfcurveto{\pgfxy(181.06,27.02)}{\pgfxy(180.02,23.55)}{\pgfxy(180.00,20.00)}\pgfstroke
\pgfmoveto{\pgfxy(190.00,10.00)}\pgfcurveto{\pgfxy(193.29,12.07)}{\pgfxy(196.63,14.07)}{\pgfxy(200.00,16.00)}\pgfcurveto{\pgfxy(204.15,18.38)}{\pgfxy(208.39,20.67)}{\pgfxy(213.00,22.00)}\pgfcurveto{\pgfxy(216.32,22.96)}{\pgfxy(220.00,24.27)}{\pgfxy(220.00,27.50)}\pgfcurveto{\pgfxy(220.00,29.99)}{\pgfxy(217.62,31.65)}{\pgfxy(215.00,32.00)}\pgfcurveto{\pgfxy(211.89,32.42)}{\pgfxy(208.85,31.30)}{\pgfxy(206.00,30.00)}\pgfcurveto{\pgfxy(200.92,27.69)}{\pgfxy(196.05,24.68)}{\pgfxy(193.00,20.00)}\pgfcurveto{\pgfxy(191.06,17.02)}{\pgfxy(190.02,13.55)}{\pgfxy(190.00,10.00)}\pgfstroke
\pgfmoveto{\pgfxy(170.00,30.00)}\pgfcurveto{\pgfxy(173.29,32.07)}{\pgfxy(176.63,34.07)}{\pgfxy(180.00,36.00)}\pgfcurveto{\pgfxy(184.15,38.38)}{\pgfxy(188.39,40.67)}{\pgfxy(193.00,42.00)}\pgfcurveto{\pgfxy(196.32,42.96)}{\pgfxy(200.00,44.27)}{\pgfxy(200.00,47.50)}\pgfcurveto{\pgfxy(200.00,49.99)}{\pgfxy(197.62,51.65)}{\pgfxy(195.00,52.00)}\pgfcurveto{\pgfxy(191.89,52.42)}{\pgfxy(188.85,51.30)}{\pgfxy(186.00,50.00)}\pgfcurveto{\pgfxy(180.92,47.69)}{\pgfxy(176.05,44.68)}{\pgfxy(173.00,40.00)}\pgfcurveto{\pgfxy(171.06,37.02)}{\pgfxy(170.02,33.55)}{\pgfxy(170.00,30.00)}\pgfstroke
\pgfmoveto{\pgfxy(150.00,50.00)}\pgfcurveto{\pgfxy(153.29,52.07)}{\pgfxy(156.63,54.07)}{\pgfxy(160.00,56.00)}\pgfcurveto{\pgfxy(164.15,58.38)}{\pgfxy(168.39,60.67)}{\pgfxy(173.00,62.00)}\pgfcurveto{\pgfxy(176.32,62.96)}{\pgfxy(180.00,64.27)}{\pgfxy(180.00,67.50)}\pgfcurveto{\pgfxy(180.00,69.99)}{\pgfxy(177.62,71.65)}{\pgfxy(175.00,72.00)}\pgfcurveto{\pgfxy(171.89,72.42)}{\pgfxy(168.85,71.30)}{\pgfxy(166.00,70.00)}\pgfcurveto{\pgfxy(160.92,67.69)}{\pgfxy(156.05,64.68)}{\pgfxy(153.00,60.00)}\pgfcurveto{\pgfxy(151.06,57.02)}{\pgfxy(150.02,53.55)}{\pgfxy(150.00,50.00)}\pgfstroke
\pgfmoveto{\pgfxy(140.00,60.00)}\pgfcurveto{\pgfxy(143.29,62.07)}{\pgfxy(146.63,64.07)}{\pgfxy(150.00,66.00)}\pgfcurveto{\pgfxy(154.15,68.38)}{\pgfxy(158.39,70.67)}{\pgfxy(163.00,72.00)}\pgfcurveto{\pgfxy(166.32,72.96)}{\pgfxy(170.00,74.27)}{\pgfxy(170.00,77.50)}\pgfcurveto{\pgfxy(170.00,79.99)}{\pgfxy(167.62,81.65)}{\pgfxy(165.00,82.00)}\pgfcurveto{\pgfxy(161.89,82.42)}{\pgfxy(158.85,81.30)}{\pgfxy(156.00,80.00)}\pgfcurveto{\pgfxy(150.92,77.69)}{\pgfxy(146.05,74.68)}{\pgfxy(143.00,70.00)}\pgfcurveto{\pgfxy(141.06,67.02)}{\pgfxy(140.02,63.55)}{\pgfxy(140.00,60.00)}\pgfstroke
\pgfmoveto{\pgfxy(130.00,70.00)}\pgfcurveto{\pgfxy(133.29,72.07)}{\pgfxy(136.63,74.07)}{\pgfxy(140.00,76.00)}\pgfcurveto{\pgfxy(144.15,78.38)}{\pgfxy(148.39,80.67)}{\pgfxy(153.00,82.00)}\pgfcurveto{\pgfxy(156.32,82.96)}{\pgfxy(160.00,84.27)}{\pgfxy(160.00,87.50)}\pgfcurveto{\pgfxy(160.00,89.99)}{\pgfxy(157.62,91.65)}{\pgfxy(155.00,92.00)}\pgfcurveto{\pgfxy(151.89,92.42)}{\pgfxy(148.85,91.30)}{\pgfxy(146.00,90.00)}\pgfcurveto{\pgfxy(140.92,87.69)}{\pgfxy(136.05,84.68)}{\pgfxy(133.00,80.00)}\pgfcurveto{\pgfxy(131.06,77.02)}{\pgfxy(130.02,73.55)}{\pgfxy(130.00,70.00)}\pgfstroke
\pgfmoveto{\pgfxy(50.00,-11.00)}\pgfcurveto{\pgfxy(53.29,-8.93)}{\pgfxy(56.63,-6.93)}{\pgfxy(60.00,-5.00)}\pgfcurveto{\pgfxy(64.15,-2.62)}{\pgfxy(68.39,-0.33)}{\pgfxy(73.00,1.00)}\pgfcurveto{\pgfxy(76.32,1.96)}{\pgfxy(80.00,3.27)}{\pgfxy(80.00,6.50)}\pgfcurveto{\pgfxy(80.00,8.99)}{\pgfxy(77.62,10.65)}{\pgfxy(75.00,11.00)}\pgfcurveto{\pgfxy(71.89,11.42)}{\pgfxy(68.85,10.30)}{\pgfxy(66.00,9.00)}\pgfcurveto{\pgfxy(60.92,6.69)}{\pgfxy(56.05,3.68)}{\pgfxy(53.00,-1.00)}\pgfcurveto{\pgfxy(51.06,-3.98)}{\pgfxy(50.02,-7.45)}{\pgfxy(50.00,-11.00)}\pgfstroke
\pgfmoveto{\pgfxy(30.00,10.00)}\pgfcurveto{\pgfxy(33.29,12.07)}{\pgfxy(36.63,14.07)}{\pgfxy(40.00,16.00)}\pgfcurveto{\pgfxy(44.15,18.38)}{\pgfxy(48.39,20.67)}{\pgfxy(53.00,22.00)}\pgfcurveto{\pgfxy(56.32,22.96)}{\pgfxy(60.00,24.27)}{\pgfxy(60.00,27.50)}\pgfcurveto{\pgfxy(60.00,29.99)}{\pgfxy(57.62,31.65)}{\pgfxy(55.00,32.00)}\pgfcurveto{\pgfxy(51.89,32.42)}{\pgfxy(48.85,31.30)}{\pgfxy(46.00,30.00)}\pgfcurveto{\pgfxy(40.92,27.69)}{\pgfxy(36.05,24.68)}{\pgfxy(33.00,20.00)}\pgfcurveto{\pgfxy(31.06,17.02)}{\pgfxy(30.02,13.55)}{\pgfxy(30.00,10.00)}\pgfstroke
\pgfmoveto{\pgfxy(60.00,60.00)}\pgfcurveto{\pgfxy(63.29,62.07)}{\pgfxy(66.63,64.07)}{\pgfxy(70.00,66.00)}\pgfcurveto{\pgfxy(74.15,68.38)}{\pgfxy(78.39,70.67)}{\pgfxy(83.00,72.00)}\pgfcurveto{\pgfxy(86.32,72.96)}{\pgfxy(90.00,74.27)}{\pgfxy(90.00,77.50)}\pgfcurveto{\pgfxy(90.00,79.99)}{\pgfxy(87.62,81.65)}{\pgfxy(85.00,82.00)}\pgfcurveto{\pgfxy(81.89,82.42)}{\pgfxy(78.85,81.30)}{\pgfxy(76.00,80.00)}\pgfcurveto{\pgfxy(70.92,77.69)}{\pgfxy(66.05,74.68)}{\pgfxy(63.00,70.00)}\pgfcurveto{\pgfxy(61.06,67.02)}{\pgfxy(60.02,63.55)}{\pgfxy(60.00,60.00)}\pgfstroke
\end{pgfpicture}%
$$
\item If $\sigma_{2i-1}<a^{(i)}$, then $b^{(i)}$ does not exist.
Let $S$ be  the subtree with root $a^{(i)}$ of $T^{(i+1)}$, then
the tree $T^{(i)}$ is defined as follows.
\begin{itemize}
\item Graft $\sigma_{2i}$ so that  $a^{(i)}$ is its right-child;
\item Transplant the trees $S$ as the right-subtree of the vertex $\sigma_{2i}$;
\item Graft $\sigma_{2i-1}$ as the left-child of $\sigma_{2i}$.
\end{itemize}
We can illustrate this transformation by the following 
$$
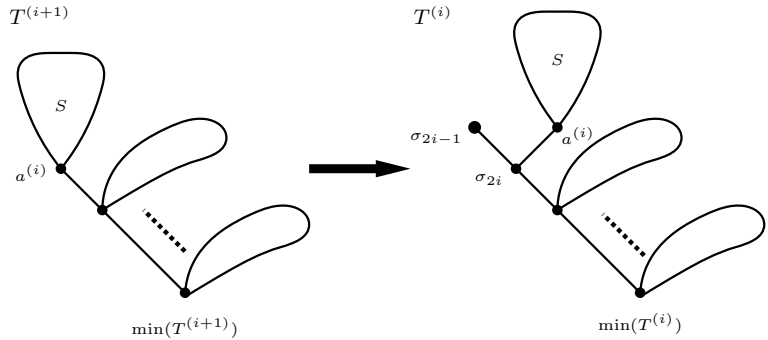
\centering
\begin{pgfpicture}{8.75mm}{-13.48mm}{112.00mm}{34.41mm}
\pgfsetxvec{\pgfpoint{0.55mm}{0mm}}
\pgfsetyvec{\pgfpoint{0mm}{0.55mm}}
\color[rgb]{0,0,0}\pgfsetlinewidth{0.30mm}\pgfsetdash{}{0mm}
\pgfputat{\pgfxy(25.00,55.00)}{\pgfbox[bottom,left]{\fontsize{7.82}{9.39}\selectfont \makebox[0pt]{$T^{(i+1)}$}}}
\pgfcircle[fill]{\pgfxy(140.00,20.00)}{0.55mm}
\pgfcircle[stroke]{\pgfxy(140.00,20.00)}{0.55mm}
\pgfmoveto{\pgfxy(140.00,20.00)}\pgflineto{\pgfxy(150.00,10.00)}\pgfstroke
\pgfputat{\pgfxy(120.00,55.00)}{\pgfbox[bottom,left]{\fontsize{7.82}{9.39}\selectfont \makebox[0pt]{$T^{(i)}$}}}
\pgfmoveto{\pgfxy(140.00,20.00)}\pgflineto{\pgfxy(130.00,30.00)}\pgfstroke
\pgfcircle[fill]{\pgfxy(130.00,30.00)}{0.55mm}
\pgfsetlinewidth{0.60mm}\pgfcircle[stroke]{\pgfxy(130.00,30.00)}{0.55mm}
\pgfsetlinewidth{0.30mm}\pgfmoveto{\pgfxy(140.00,20.00)}\pgflineto{\pgfxy(150.00,30.00)}\pgfstroke
\pgfputat{\pgfxy(137.00,17.00)}{\pgfbox[bottom,left]{\fontsize{6.26}{7.51}\selectfont \makebox[0pt][r]{$\sigma_{2i}$}}}
\pgfputat{\pgfxy(127.00,27.00)}{\pgfbox[bottom,left]{\fontsize{6.26}{7.51}\selectfont \makebox[0pt][r]{$\sigma_{2i-1}$}}}
\pgfcircle[stroke]{\pgfxy(150.00,39.00)}{0.00mm}
\pgfputat{\pgfxy(151.00,26.00)}{\pgfbox[bottom,left]{\fontsize{6.26}{7.51}\selectfont $a^{(i)}$}}
\pgfcircle[fill]{\pgfxy(150.00,30.00)}{0.55mm}
\pgfcircle[stroke]{\pgfxy(150.00,30.00)}{0.55mm}
\pgfputat{\pgfxy(150.00,45.00)}{\pgfbox[bottom,left]{\fontsize{6.26}{7.51}\selectfont \makebox[0pt]{$S$}}}
\pgfmoveto{\pgfxy(150.00,30.00)}\pgfcurveto{\pgfxy(145.27,35.87)}{\pgfxy(141.86,42.69)}{\pgfxy(140.00,50.00)}\pgfcurveto{\pgfxy(139.57,51.67)}{\pgfxy(139.25,53.45)}{\pgfxy(140.00,55.00)}\pgfcurveto{\pgfxy(141.64,58.39)}{\pgfxy(146.11,58.00)}{\pgfxy(150.00,58.00)}\pgfcurveto{\pgfxy(153.89,58.00)}{\pgfxy(158.36,58.39)}{\pgfxy(160.00,55.00)}\pgfcurveto{\pgfxy(160.75,53.45)}{\pgfxy(160.43,51.67)}{\pgfxy(160.00,50.00)}\pgfcurveto{\pgfxy(158.14,42.69)}{\pgfxy(154.73,35.87)}{\pgfxy(150.00,30.00)}\pgfstroke
\pgfcircle[fill]{\pgfxy(30.00,20.00)}{0.55mm}
\pgfcircle[stroke]{\pgfxy(30.00,20.00)}{0.55mm}
\pgfmoveto{\pgfxy(30.00,20.00)}\pgflineto{\pgfxy(40.00,10.00)}\pgfstroke
\pgfputat{\pgfxy(27.00,17.00)}{\pgfbox[bottom,left]{\fontsize{6.26}{7.51}\selectfont \makebox[0pt][r]{$a^{(i)}$}}}
\pgfputat{\pgfxy(30.00,34.00)}{\pgfbox[bottom,left]{\fontsize{6.26}{7.51}\selectfont \makebox[0pt]{$S$}}}
\pgfmoveto{\pgfxy(30.00,20.00)}\pgfcurveto{\pgfxy(25.27,25.87)}{\pgfxy(21.86,32.69)}{\pgfxy(20.00,40.00)}\pgfcurveto{\pgfxy(19.57,41.67)}{\pgfxy(19.25,43.45)}{\pgfxy(20.00,45.00)}\pgfcurveto{\pgfxy(21.64,48.39)}{\pgfxy(26.11,48.00)}{\pgfxy(30.00,48.00)}\pgfcurveto{\pgfxy(33.89,48.00)}{\pgfxy(38.36,48.39)}{\pgfxy(40.00,45.00)}\pgfcurveto{\pgfxy(40.75,43.45)}{\pgfxy(40.43,41.67)}{\pgfxy(40.00,40.00)}\pgfcurveto{\pgfxy(38.14,32.69)}{\pgfxy(34.73,25.87)}{\pgfxy(30.00,20.00)}\pgfstroke
\pgfsetlinewidth{1.20mm}\pgfmoveto{\pgfxy(90.00,20.00)}\pgflineto{\pgfxy(110.00,20.00)}\pgfstroke
\pgfmoveto{\pgfxy(110.00,20.00)}\pgflineto{\pgfxy(107.20,20.70)}\pgflineto{\pgfxy(107.20,19.30)}\pgflineto{\pgfxy(110.00,20.00)}\pgfclosepath\pgffill
\pgfmoveto{\pgfxy(110.00,20.00)}\pgflineto{\pgfxy(107.20,20.70)}\pgflineto{\pgfxy(107.20,19.30)}\pgflineto{\pgfxy(110.00,20.00)}\pgfclosepath\pgfstroke
\pgfsetlinewidth{0.30mm}\pgfmoveto{\pgfxy(40.00,10.00)}\pgflineto{\pgfxy(60.00,-10.00)}\pgfstroke
\pgfcircle[fill]{\pgfxy(40.00,10.00)}{0.55mm}
\pgfcircle[stroke]{\pgfxy(40.00,10.00)}{0.55mm}
\pgfcircle[fill]{\pgfxy(60.00,-10.00)}{0.55mm}
\pgfcircle[stroke]{\pgfxy(60.00,-10.00)}{0.55mm}
\pgfputat{\pgfxy(60.00,-20.00)}{\pgfbox[bottom,left]{\fontsize{6.26}{7.51}\selectfont \makebox[0pt]{$\min(T^{(i+1)})$}}}
\pgfsetdash{{0.60mm}{0.50mm}}{0mm}\pgfsetlinewidth{0.60mm}\pgfmoveto{\pgfxy(60.00,0.00)}\pgflineto{\pgfxy(50.00,10.00)}\pgfstroke
\pgfsetdash{}{0mm}\pgfsetlinewidth{0.30mm}\pgfmoveto{\pgfxy(60.00,-11.00)}\pgfcurveto{\pgfxy(63.29,-8.93)}{\pgfxy(66.63,-6.93)}{\pgfxy(70.00,-5.00)}\pgfcurveto{\pgfxy(74.15,-2.62)}{\pgfxy(78.39,-0.33)}{\pgfxy(83.00,1.00)}\pgfcurveto{\pgfxy(86.32,1.96)}{\pgfxy(90.00,3.27)}{\pgfxy(90.00,6.50)}\pgfcurveto{\pgfxy(90.00,8.99)}{\pgfxy(87.62,10.65)}{\pgfxy(85.00,11.00)}\pgfcurveto{\pgfxy(81.89,11.42)}{\pgfxy(78.85,10.30)}{\pgfxy(76.00,9.00)}\pgfcurveto{\pgfxy(70.92,6.69)}{\pgfxy(66.05,3.68)}{\pgfxy(63.00,-1.00)}\pgfcurveto{\pgfxy(61.06,-3.98)}{\pgfxy(60.02,-7.45)}{\pgfxy(60.00,-11.00)}\pgfstroke
\pgfmoveto{\pgfxy(40.00,10.00)}\pgfcurveto{\pgfxy(43.29,12.07)}{\pgfxy(46.63,14.07)}{\pgfxy(50.00,16.00)}\pgfcurveto{\pgfxy(54.15,18.38)}{\pgfxy(58.39,20.67)}{\pgfxy(63.00,22.00)}\pgfcurveto{\pgfxy(66.32,22.96)}{\pgfxy(70.00,24.27)}{\pgfxy(70.00,27.50)}\pgfcurveto{\pgfxy(70.00,29.99)}{\pgfxy(67.62,31.65)}{\pgfxy(65.00,32.00)}\pgfcurveto{\pgfxy(61.89,32.42)}{\pgfxy(58.85,31.30)}{\pgfxy(56.00,30.00)}\pgfcurveto{\pgfxy(50.92,27.69)}{\pgfxy(46.05,24.68)}{\pgfxy(43.00,20.00)}\pgfcurveto{\pgfxy(41.06,17.02)}{\pgfxy(40.02,13.55)}{\pgfxy(40.00,10.00)}\pgfstroke
\pgfcircle[fill]{\pgfxy(150.00,10.00)}{0.55mm}
\pgfcircle[stroke]{\pgfxy(150.00,10.00)}{0.55mm}
\pgfcircle[fill]{\pgfxy(170.00,-10.00)}{0.55mm}
\pgfcircle[stroke]{\pgfxy(170.00,-10.00)}{0.55mm}
\pgfputat{\pgfxy(170.00,-19.00)}{\pgfbox[bottom,left]{\fontsize{6.26}{7.51}\selectfont \makebox[0pt]{$\min(T^{(i)})$}}}
\pgfsetdash{{0.60mm}{0.50mm}}{0mm}\pgfsetlinewidth{0.60mm}\pgfmoveto{\pgfxy(171.00,-1.00)}\pgflineto{\pgfxy(161.00,9.00)}\pgfstroke
\pgfsetdash{}{0mm}\pgfsetlinewidth{0.30mm}\pgfmoveto{\pgfxy(150.00,10.00)}\pgflineto{\pgfxy(170.00,-10.00)}\pgfstroke
\pgfmoveto{\pgfxy(170.00,-11.00)}\pgfcurveto{\pgfxy(173.29,-8.93)}{\pgfxy(176.63,-6.93)}{\pgfxy(180.00,-5.00)}\pgfcurveto{\pgfxy(184.15,-2.62)}{\pgfxy(188.39,-0.33)}{\pgfxy(193.00,1.00)}\pgfcurveto{\pgfxy(196.32,1.96)}{\pgfxy(200.00,3.27)}{\pgfxy(200.00,6.50)}\pgfcurveto{\pgfxy(200.00,8.99)}{\pgfxy(197.62,10.65)}{\pgfxy(195.00,11.00)}\pgfcurveto{\pgfxy(191.89,11.42)}{\pgfxy(188.85,10.30)}{\pgfxy(186.00,9.00)}\pgfcurveto{\pgfxy(180.92,6.69)}{\pgfxy(176.05,3.68)}{\pgfxy(173.00,-1.00)}\pgfcurveto{\pgfxy(171.06,-3.98)}{\pgfxy(170.02,-7.45)}{\pgfxy(170.00,-11.00)}\pgfstroke
\pgfmoveto{\pgfxy(150.00,10.00)}\pgfcurveto{\pgfxy(153.29,12.07)}{\pgfxy(156.63,14.07)}{\pgfxy(160.00,16.00)}\pgfcurveto{\pgfxy(164.15,18.38)}{\pgfxy(168.39,20.67)}{\pgfxy(173.00,22.00)}\pgfcurveto{\pgfxy(176.32,22.96)}{\pgfxy(180.00,24.27)}{\pgfxy(180.00,27.50)}\pgfcurveto{\pgfxy(180.00,29.99)}{\pgfxy(177.62,31.65)}{\pgfxy(175.00,32.00)}\pgfcurveto{\pgfxy(171.89,32.42)}{\pgfxy(168.85,31.30)}{\pgfxy(166.00,30.00)}\pgfcurveto{\pgfxy(160.92,27.69)}{\pgfxy(156.05,24.68)}{\pgfxy(153.00,20.00)}\pgfcurveto{\pgfxy(151.06,17.02)}{\pgfxy(150.02,13.55)}{\pgfxy(150.00,10.00)}\pgfstroke
\end{pgfpicture}%
$$
\end{enumerate}
We note that the vertex $\sigma_{2i-1}$ is the minimal leaf of $T^{(i)}$ for $1 \le i \le m-1$.
And then, we define $\psi'(\sigma) = T^{(1)}$.
\begin{ex}
We run the new algorithm to the examples $\sigma = 748591623$ in Example 3.2 and Fig. 2 in \cite{GSZ11}.
As $n=9$, we have five pairs 
\begin{align*}
d_5(\sigma) &= (3),&
d_4(\sigma) &= (6,2),&
d_3(\sigma) &= (9,1),&
d_2(\sigma) &= (8,5),&
d_1(\sigma) &= (7,4).
\end{align*}
By Algorithm~C, we get five trees sequentially
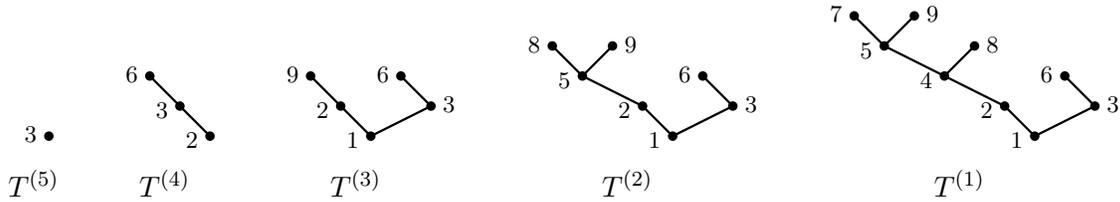
\begin{align*}
\centering
\begin{pgfpicture}{-6.22mm}{-10.86mm}{2.48mm}{3.71mm}
\pgfsetxvec{\pgfpoint{0.80mm}{0mm}}
\pgfsetyvec{\pgfpoint{0mm}{0.80mm}}
\color[rgb]{0,0,0}\pgfsetlinewidth{0.30mm}\pgfsetdash{}{0mm}
\pgfcircle[fill]{\pgfxy(0.00,0.00)}{0.48mm}
\pgfcircle[stroke]{\pgfxy(0.00,0.00)}{0.48mm}
\pgfputat{\pgfxy(-2.00,-1.00)}{\pgfbox[bottom,left]{\fontsize{9.10}{10.93}\selectfont \makebox[0pt][r]{3}}}
\pgfputat{\pgfxy(-2.50,-10.00)}{\pgfbox[bottom,left]{\fontsize{11.38}{13.66}\selectfont \makebox[0pt]{$T^{(5)}$}}}
\end{pgfpicture}%
&&
\centering
\begin{pgfpicture}{-13.20mm}{-10.86mm}{2.48mm}{11.71mm}
\pgfsetxvec{\pgfpoint{0.80mm}{0mm}}
\pgfsetyvec{\pgfpoint{0mm}{0.80mm}}
\color[rgb]{0,0,0}\pgfsetlinewidth{0.30mm}\pgfsetdash{}{0mm}
\pgfcircle[fill]{\pgfxy(0.00,0.00)}{0.48mm}
\pgfcircle[stroke]{\pgfxy(0.00,0.00)}{0.48mm}
\pgfcircle[fill]{\pgfxy(-5.00,5.00)}{0.48mm}
\pgfcircle[stroke]{\pgfxy(-5.00,5.00)}{0.48mm}
\pgfcircle[fill]{\pgfxy(-10.00,10.00)}{0.48mm}
\pgfcircle[stroke]{\pgfxy(-10.00,10.00)}{0.48mm}
\pgfmoveto{\pgfxy(-10.00,10.00)}\pgflineto{\pgfxy(-5.00,5.00)}\pgfstroke
\pgfmoveto{\pgfxy(-5.00,5.00)}\pgflineto{\pgfxy(0.00,0.00)}\pgfstroke
\pgfputat{\pgfxy(-2.00,-2.00)}{\pgfbox[bottom,left]{\fontsize{9.10}{10.93}\selectfont \makebox[0pt][r]{2}}}
\pgfputat{\pgfxy(-7.00,3.00)}{\pgfbox[bottom,left]{\fontsize{9.10}{10.93}\selectfont \makebox[0pt][r]{3}}}
\pgfputat{\pgfxy(-12.00,9.00)}{\pgfbox[bottom,left]{\fontsize{9.10}{10.93}\selectfont \makebox[0pt][r]{6}}}
\pgfputat{\pgfxy(-7.50,-10.00)}{\pgfbox[bottom,left]{\fontsize{11.38}{13.66}\selectfont \makebox[0pt]{$T^{(4)}$}}}
\end{pgfpicture}%
&&
\centering
\begin{pgfpicture}{-13.20mm}{-10.86mm}{13.20mm}{11.71mm}
\pgfsetxvec{\pgfpoint{0.80mm}{0mm}}
\pgfsetyvec{\pgfpoint{0mm}{0.80mm}}
\color[rgb]{0,0,0}\pgfsetlinewidth{0.30mm}\pgfsetdash{}{0mm}
\pgfcircle[fill]{\pgfxy(0.00,0.00)}{0.48mm}
\pgfcircle[stroke]{\pgfxy(0.00,0.00)}{0.48mm}
\pgfputat{\pgfxy(-2.00,-2.00)}{\pgfbox[bottom,left]{\fontsize{9.10}{10.93}\selectfont \makebox[0pt][r]{1}}}
\pgfcircle[fill]{\pgfxy(-5.00,5.00)}{0.48mm}
\pgfcircle[stroke]{\pgfxy(-5.00,5.00)}{0.48mm}
\pgfcircle[fill]{\pgfxy(-10.00,10.00)}{0.48mm}
\pgfcircle[stroke]{\pgfxy(-10.00,10.00)}{0.48mm}
\pgfcircle[fill]{\pgfxy(5.00,10.00)}{0.48mm}
\pgfcircle[stroke]{\pgfxy(5.00,10.00)}{0.48mm}
\pgfcircle[fill]{\pgfxy(10.00,5.00)}{0.48mm}
\pgfcircle[stroke]{\pgfxy(10.00,5.00)}{0.48mm}
\pgfmoveto{\pgfxy(-10.00,10.00)}\pgflineto{\pgfxy(-5.00,5.00)}\pgfstroke
\pgfmoveto{\pgfxy(-5.00,5.00)}\pgflineto{\pgfxy(0.00,0.00)}\pgfstroke
\pgfmoveto{\pgfxy(0.00,0.00)}\pgflineto{\pgfxy(10.00,5.00)}\pgfstroke
\pgfmoveto{\pgfxy(10.00,5.00)}\pgflineto{\pgfxy(5.00,10.00)}\pgfstroke
\pgfputat{\pgfxy(-7.00,3.00)}{\pgfbox[bottom,left]{\fontsize{9.10}{10.93}\selectfont \makebox[0pt][r]{2}}}
\pgfputat{\pgfxy(-12.00,9.00)}{\pgfbox[bottom,left]{\fontsize{9.10}{10.93}\selectfont \makebox[0pt][r]{9}}}
\pgfputat{\pgfxy(12.00,4.00)}{\pgfbox[bottom,left]{\fontsize{9.10}{10.93}\selectfont 3}}
\pgfputat{\pgfxy(3.00,9.00)}{\pgfbox[bottom,left]{\fontsize{9.10}{10.93}\selectfont \makebox[0pt][r]{6}}}
\pgfputat{\pgfxy(-2.50,-10.00)}{\pgfbox[bottom,left]{\fontsize{11.38}{13.66}\selectfont \makebox[0pt]{$T^{(3)}$}}}
\end{pgfpicture}%
&&
\centering
\begin{pgfpicture}{-21.20mm}{-10.86mm}{13.20mm}{15.71mm}
\pgfsetxvec{\pgfpoint{0.80mm}{0mm}}
\pgfsetyvec{\pgfpoint{0mm}{0.80mm}}
\color[rgb]{0,0,0}\pgfsetlinewidth{0.30mm}\pgfsetdash{}{0mm}
\pgfcircle[fill]{\pgfxy(0.00,0.00)}{0.48mm}
\pgfcircle[stroke]{\pgfxy(0.00,0.00)}{0.48mm}
\pgfputat{\pgfxy(-2.00,-2.00)}{\pgfbox[bottom,left]{\fontsize{9.10}{10.93}\selectfont \makebox[0pt][r]{1}}}
\pgfcircle[fill]{\pgfxy(-5.00,5.00)}{0.48mm}
\pgfcircle[stroke]{\pgfxy(-5.00,5.00)}{0.48mm}
\pgfcircle[fill]{\pgfxy(-15.00,10.00)}{0.48mm}
\pgfcircle[stroke]{\pgfxy(-15.00,10.00)}{0.48mm}
\pgfcircle[fill]{\pgfxy(-10.00,15.00)}{0.48mm}
\pgfcircle[stroke]{\pgfxy(-10.00,15.00)}{0.48mm}
\pgfcircle[fill]{\pgfxy(-20.00,15.00)}{0.48mm}
\pgfcircle[stroke]{\pgfxy(-20.00,15.00)}{0.48mm}
\pgfcircle[fill]{\pgfxy(5.00,10.00)}{0.48mm}
\pgfcircle[stroke]{\pgfxy(5.00,10.00)}{0.48mm}
\pgfcircle[fill]{\pgfxy(10.00,5.00)}{0.48mm}
\pgfcircle[stroke]{\pgfxy(10.00,5.00)}{0.48mm}
\pgfmoveto{\pgfxy(-20.00,15.00)}\pgflineto{\pgfxy(-15.00,10.00)}\pgfstroke
\pgfmoveto{\pgfxy(-10.00,15.00)}\pgflineto{\pgfxy(-15.00,10.00)}\pgfstroke
\pgfmoveto{\pgfxy(-15.00,10.00)}\pgflineto{\pgfxy(-5.00,5.00)}\pgfstroke
\pgfmoveto{\pgfxy(-5.00,5.00)}\pgflineto{\pgfxy(0.00,0.00)}\pgfstroke
\pgfmoveto{\pgfxy(0.00,0.00)}\pgflineto{\pgfxy(10.00,5.00)}\pgfstroke
\pgfmoveto{\pgfxy(10.00,5.00)}\pgflineto{\pgfxy(5.00,10.00)}\pgfstroke
\pgfputat{\pgfxy(-7.00,3.00)}{\pgfbox[bottom,left]{\fontsize{9.10}{10.93}\selectfont \makebox[0pt][r]{2}}}
\pgfputat{\pgfxy(-17.00,8.00)}{\pgfbox[bottom,left]{\fontsize{9.10}{10.93}\selectfont \makebox[0pt][r]{5}}}
\pgfputat{\pgfxy(-8.00,14.00)}{\pgfbox[bottom,left]{\fontsize{9.10}{10.93}\selectfont 9}}
\pgfputat{\pgfxy(-22.00,14.00)}{\pgfbox[bottom,left]{\fontsize{9.10}{10.93}\selectfont \makebox[0pt][r]{8}}}
\pgfputat{\pgfxy(12.00,4.00)}{\pgfbox[bottom,left]{\fontsize{9.10}{10.93}\selectfont 3}}
\pgfputat{\pgfxy(3.00,9.00)}{\pgfbox[bottom,left]{\fontsize{9.10}{10.93}\selectfont \makebox[0pt][r]{6}}}
\pgfputat{\pgfxy(-7.50,-10.00)}{\pgfbox[bottom,left]{\fontsize{11.38}{13.66}\selectfont \makebox[0pt]{$T^{(2)}$}}}
\end{pgfpicture}%
&&
\centering
\begin{pgfpicture}{-29.20mm}{-10.86mm}{13.20mm}{19.71mm}
\pgfsetxvec{\pgfpoint{0.80mm}{0mm}}
\pgfsetyvec{\pgfpoint{0mm}{0.80mm}}
\color[rgb]{0,0,0}\pgfsetlinewidth{0.30mm}\pgfsetdash{}{0mm}
\pgfcircle[fill]{\pgfxy(0.00,0.00)}{0.48mm}
\pgfcircle[stroke]{\pgfxy(0.00,0.00)}{0.48mm}
\pgfputat{\pgfxy(-2.00,-2.00)}{\pgfbox[bottom,left]{\fontsize{9.10}{10.93}\selectfont \makebox[0pt][r]{1}}}
\pgfcircle[fill]{\pgfxy(-5.00,5.00)}{0.48mm}
\pgfcircle[stroke]{\pgfxy(-5.00,5.00)}{0.48mm}
\pgfcircle[fill]{\pgfxy(-15.00,10.00)}{0.48mm}
\pgfcircle[stroke]{\pgfxy(-15.00,10.00)}{0.48mm}
\pgfcircle[fill]{\pgfxy(-10.00,15.00)}{0.48mm}
\pgfcircle[stroke]{\pgfxy(-10.00,15.00)}{0.48mm}
\pgfcircle[fill]{\pgfxy(-25.00,15.00)}{0.48mm}
\pgfcircle[stroke]{\pgfxy(-25.00,15.00)}{0.48mm}
\pgfcircle[fill]{\pgfxy(-20.00,20.00)}{0.48mm}
\pgfcircle[stroke]{\pgfxy(-20.00,20.00)}{0.48mm}
\pgfcircle[fill]{\pgfxy(-30.00,20.00)}{0.48mm}
\pgfcircle[stroke]{\pgfxy(-30.00,20.00)}{0.48mm}
\pgfcircle[fill]{\pgfxy(5.00,10.00)}{0.48mm}
\pgfcircle[stroke]{\pgfxy(5.00,10.00)}{0.48mm}
\pgfcircle[fill]{\pgfxy(10.00,5.00)}{0.48mm}
\pgfcircle[stroke]{\pgfxy(10.00,5.00)}{0.48mm}
\pgfmoveto{\pgfxy(-30.00,20.00)}\pgflineto{\pgfxy(-25.00,15.00)}\pgfstroke
\pgfmoveto{\pgfxy(-20.00,20.00)}\pgflineto{\pgfxy(-25.00,15.00)}\pgfstroke
\pgfmoveto{\pgfxy(-25.00,15.00)}\pgflineto{\pgfxy(-15.00,10.00)}\pgfstroke
\pgfmoveto{\pgfxy(-10.00,15.00)}\pgflineto{\pgfxy(-15.00,10.00)}\pgfstroke
\pgfmoveto{\pgfxy(-15.00,10.00)}\pgflineto{\pgfxy(-5.00,5.00)}\pgfstroke
\pgfmoveto{\pgfxy(-5.00,5.00)}\pgflineto{\pgfxy(0.00,0.00)}\pgfstroke
\pgfmoveto{\pgfxy(0.00,0.00)}\pgflineto{\pgfxy(10.00,5.00)}\pgfstroke
\pgfmoveto{\pgfxy(10.00,5.00)}\pgflineto{\pgfxy(5.00,10.00)}\pgfstroke
\pgfputat{\pgfxy(-7.00,3.00)}{\pgfbox[bottom,left]{\fontsize{9.10}{10.93}\selectfont \makebox[0pt][r]{2}}}
\pgfputat{\pgfxy(-17.00,8.00)}{\pgfbox[bottom,left]{\fontsize{9.10}{10.93}\selectfont \makebox[0pt][r]{4}}}
\pgfputat{\pgfxy(-27.00,13.00)}{\pgfbox[bottom,left]{\fontsize{9.10}{10.93}\selectfont \makebox[0pt][r]{5}}}
\pgfputat{\pgfxy(-32.00,19.00)}{\pgfbox[bottom,left]{\fontsize{9.10}{10.93}\selectfont \makebox[0pt][r]{7}}}
\pgfputat{\pgfxy(-18.00,19.00)}{\pgfbox[bottom,left]{\fontsize{9.10}{10.93}\selectfont 9}}
\pgfputat{\pgfxy(-8.00,14.00)}{\pgfbox[bottom,left]{\fontsize{9.10}{10.93}\selectfont 8}}
\pgfputat{\pgfxy(12.00,4.00)}{\pgfbox[bottom,left]{\fontsize{9.10}{10.93}\selectfont 3}}
\pgfputat{\pgfxy(3.00,9.00)}{\pgfbox[bottom,left]{\fontsize{9.10}{10.93}\selectfont \makebox[0pt][r]{6}}}
\pgfputat{\pgfxy(-12.50,-10.00)}{\pgfbox[bottom,left]{\fontsize{11.38}{13.66}\selectfont \makebox[0pt]{$T^{(1)}$}}}
\end{pgfpicture}%
\end{align*}
with 
\begin{align*}
a^{(4)}&=3, & a^{(3)}&=2, & a^{(2)}&=9, & a^{(1)}&=5,\\
b^{(4)}&=3, & b^{(3)}&=2, & b^{(2)}&\text{ does not exist}, & b^{(1)}&=5.
\end{align*}
Thus, the increasing 1-2 tree $T^{(1)}$ obtained from $\sigma$ under $\psi'$ is the same with the tree obtained from $\sigma$ under $\psi$ in Example 3.2 and Fig. 2 in \cite{GSZ11}.
\end{ex}

\begin{thm}
The two bijections $\psi$ and $\psi'$ from $\Ank$ to $\Tnk$ are equal.
\end{thm}
\begin{proof}
It is clear that (C2) is equivalent to (B1).
Since the rule (B2a) just exchange two labels, but does not change the tree-structure, 
it is enough to show that (C1) is produced recursively from (B1) and (B2b).

Assume that $\sigma_{2i-1} > a^{(i)} > \sigma_{2i}$ for some $1 \le i \le m-1$. 
We obtain the right tree in the following by applying (B1) to the tree $T^{(i+1)}$.
$$
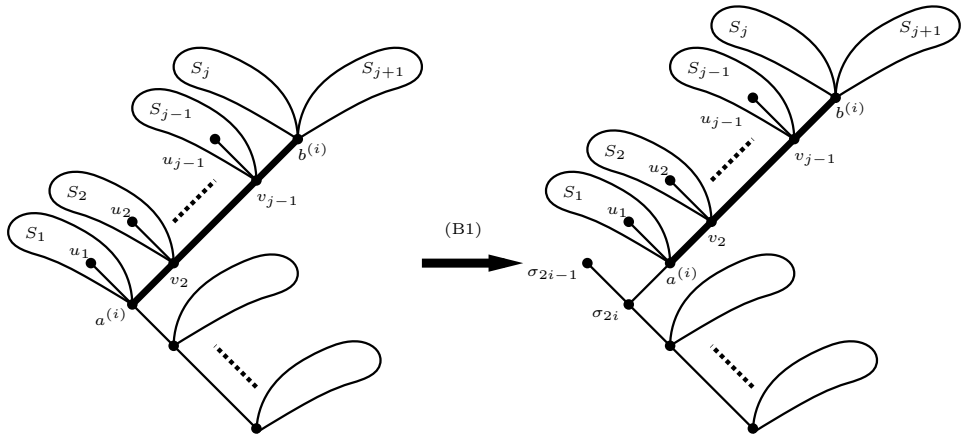
\centering
\begin{pgfpicture}{-7.50mm}{-8.05mm}{123.00mm}{52.65mm}
\pgfsetxvec{\pgfpoint{0.55mm}{0mm}}
\pgfsetyvec{\pgfpoint{0mm}{0.55mm}}
\color[rgb]{0,0,0}\pgfsetlinewidth{0.30mm}\pgfsetdash{}{0mm}
\pgfsetlinewidth{1.20mm}\pgfmoveto{\pgfxy(90.00,30.00)}\pgflineto{\pgfxy(110.00,30.00)}\pgfstroke
\pgfmoveto{\pgfxy(110.00,30.00)}\pgflineto{\pgfxy(107.20,30.70)}\pgflineto{\pgfxy(107.20,29.30)}\pgflineto{\pgfxy(110.00,30.00)}\pgfclosepath\pgffill
\pgfmoveto{\pgfxy(110.00,30.00)}\pgflineto{\pgfxy(107.20,30.70)}\pgflineto{\pgfxy(107.20,29.30)}\pgflineto{\pgfxy(110.00,30.00)}\pgfclosepath\pgfstroke
\pgfcircle[fill]{\pgfxy(20.00,20.00)}{0.55mm}
\pgfsetlinewidth{0.30mm}\pgfcircle[stroke]{\pgfxy(20.00,20.00)}{0.55mm}
\pgfcircle[fill]{\pgfxy(30.00,30.00)}{0.55mm}
\pgfcircle[stroke]{\pgfxy(25.00,39.00)}{0.00mm}
\pgfmoveto{\pgfxy(20.00,20.00)}\pgflineto{\pgfxy(30.00,10.00)}\pgfstroke
\pgfmoveto{\pgfxy(30.00,10.00)}\pgflineto{\pgfxy(50.00,-10.00)}\pgfstroke
\pgfcircle[fill]{\pgfxy(30.00,10.00)}{0.55mm}
\pgfcircle[stroke]{\pgfxy(30.00,10.00)}{0.55mm}
\pgfcircle[fill]{\pgfxy(50.00,-10.00)}{0.55mm}
\pgfcircle[stroke]{\pgfxy(50.00,-10.00)}{0.55mm}
\pgfputat{\pgfxy(19.00,15.00)}{\pgfbox[bottom,left]{\fontsize{6.26}{7.51}\selectfont \makebox[0pt][r]{$a^{(i)}$}}}
\pgfsetlinewidth{0.90mm}\pgfmoveto{\pgfxy(20.00,20.00)}\pgflineto{\pgfxy(60.00,60.00)}\pgfstroke
\pgfsetlinewidth{0.30mm}\pgfmoveto{\pgfxy(20.00,20.00)}\pgfcurveto{\pgfxy(16.71,22.07)}{\pgfxy(13.37,24.07)}{\pgfxy(10.00,26.00)}\pgfcurveto{\pgfxy(5.85,28.38)}{\pgfxy(1.61,30.67)}{\pgfxy(-3.00,32.00)}\pgfcurveto{\pgfxy(-6.32,32.96)}{\pgfxy(-10.00,34.27)}{\pgfxy(-10.00,37.50)}\pgfcurveto{\pgfxy(-10.00,39.99)}{\pgfxy(-7.62,41.65)}{\pgfxy(-5.00,42.00)}\pgfcurveto{\pgfxy(-1.89,42.42)}{\pgfxy(1.15,41.30)}{\pgfxy(4.00,40.00)}\pgfcurveto{\pgfxy(9.08,37.69)}{\pgfxy(13.95,34.68)}{\pgfxy(17.00,30.00)}\pgfcurveto{\pgfxy(18.94,27.02)}{\pgfxy(19.98,23.55)}{\pgfxy(20.00,20.00)}\pgfstroke
\pgfputat{\pgfxy(-6.00,36.00)}{\pgfbox[bottom,left]{\fontsize{6.26}{7.51}\selectfont $S_1$}}
\pgfmoveto{\pgfxy(20.00,40.00)}\pgflineto{\pgfxy(30.00,30.00)}\pgfstroke
\pgfcircle[fill]{\pgfxy(60.00,60.00)}{0.55mm}
\pgfcircle[stroke]{\pgfxy(60.00,60.00)}{0.55mm}
\pgfputat{\pgfxy(86.00,76.00)}{\pgfbox[bottom,left]{\fontsize{6.26}{7.51}\selectfont \makebox[0pt][r]{$S_{j+1}$}}}
\pgfputat{\pgfxy(60.00,55.00)}{\pgfbox[bottom,left]{\fontsize{6.26}{7.51}\selectfont $b^{(i)}$}}
\pgfputat{\pgfxy(34.00,76.00)}{\pgfbox[bottom,left]{\fontsize{6.26}{7.51}\selectfont $S_{j}$}}
\pgfsetdash{{0.60mm}{0.50mm}}{0mm}\pgfsetlinewidth{0.60mm}\pgfmoveto{\pgfxy(30.00,40.00)}\pgflineto{\pgfxy(40.00,50.00)}\pgfstroke
\pgfmoveto{\pgfxy(50.00,0.00)}\pgflineto{\pgfxy(40.00,10.00)}\pgfstroke
\pgfsetdash{}{0mm}\pgfsetlinewidth{0.30mm}\pgfmoveto{\pgfxy(10.00,30.00)}\pgflineto{\pgfxy(20.00,20.00)}\pgfstroke
\pgfcircle[fill]{\pgfxy(50.00,50.00)}{0.55mm}
\pgfcircle[stroke]{\pgfxy(50.00,50.00)}{0.55mm}
\pgfmoveto{\pgfxy(40.00,60.00)}\pgflineto{\pgfxy(50.00,50.00)}\pgfstroke
\pgfcircle[fill]{\pgfxy(40.00,60.00)}{0.55mm}
\pgfcircle[stroke]{\pgfxy(40.00,60.00)}{0.55mm}
\pgfputat{\pgfxy(24.00,66.00)}{\pgfbox[bottom,left]{\fontsize{6.26}{7.51}\selectfont $S_{j-1}$}}
\pgfputat{\pgfxy(4.00,46.00)}{\pgfbox[bottom,left]{\fontsize{6.26}{7.51}\selectfont $S_2$}}
\pgfcircle[fill]{\pgfxy(30.00,30.00)}{0.55mm}
\pgfcircle[stroke]{\pgfxy(30.00,30.00)}{0.55mm}
\pgfcircle[fill]{\pgfxy(20.00,40.00)}{0.55mm}
\pgfcircle[stroke]{\pgfxy(20.00,40.00)}{0.55mm}
\pgfcircle[fill]{\pgfxy(10.00,30.00)}{0.55mm}
\pgfcircle[stroke]{\pgfxy(10.00,30.00)}{0.55mm}
\pgfputat{\pgfxy(10.00,32.00)}{\pgfbox[bottom,left]{\fontsize{6.26}{7.51}\selectfont \makebox[0pt][r]{$u_1$}}}
\pgfputat{\pgfxy(38.00,54.00)}{\pgfbox[bottom,left]{\fontsize{6.26}{7.51}\selectfont \makebox[0pt][r]{$u_{j-1}$}}}
\pgfputat{\pgfxy(20.00,42.00)}{\pgfbox[bottom,left]{\fontsize{6.26}{7.51}\selectfont \makebox[0pt][r]{$u_2$}}}
\pgfputat{\pgfxy(29.00,25.00)}{\pgfbox[bottom,left]{\fontsize{6.26}{7.51}\selectfont $v_2$}}
\pgfputat{\pgfxy(50.00,45.00)}{\pgfbox[bottom,left]{\fontsize{6.26}{7.51}\selectfont $v_{j-1}$}}
\pgfmoveto{\pgfxy(50.00,50.00)}\pgfcurveto{\pgfxy(46.71,52.07)}{\pgfxy(43.37,54.07)}{\pgfxy(40.00,56.00)}\pgfcurveto{\pgfxy(35.85,58.38)}{\pgfxy(31.61,60.67)}{\pgfxy(27.00,62.00)}\pgfcurveto{\pgfxy(23.68,62.96)}{\pgfxy(20.00,64.27)}{\pgfxy(20.00,67.50)}\pgfcurveto{\pgfxy(20.00,69.99)}{\pgfxy(22.38,71.65)}{\pgfxy(25.00,72.00)}\pgfcurveto{\pgfxy(28.11,72.42)}{\pgfxy(31.15,71.30)}{\pgfxy(34.00,70.00)}\pgfcurveto{\pgfxy(39.08,67.69)}{\pgfxy(43.95,64.68)}{\pgfxy(47.00,60.00)}\pgfcurveto{\pgfxy(48.94,57.02)}{\pgfxy(49.98,53.55)}{\pgfxy(50.00,50.00)}\pgfstroke
\pgfmoveto{\pgfxy(30.00,30.00)}\pgfcurveto{\pgfxy(26.71,32.07)}{\pgfxy(23.37,34.07)}{\pgfxy(20.00,36.00)}\pgfcurveto{\pgfxy(15.85,38.38)}{\pgfxy(11.61,40.67)}{\pgfxy(7.00,42.00)}\pgfcurveto{\pgfxy(3.68,42.96)}{\pgfxy(0.00,44.27)}{\pgfxy(0.00,47.50)}\pgfcurveto{\pgfxy(0.00,49.99)}{\pgfxy(2.38,51.65)}{\pgfxy(5.00,52.00)}\pgfcurveto{\pgfxy(8.11,52.42)}{\pgfxy(11.15,51.30)}{\pgfxy(14.00,50.00)}\pgfcurveto{\pgfxy(19.08,47.69)}{\pgfxy(23.95,44.68)}{\pgfxy(27.00,40.00)}\pgfcurveto{\pgfxy(28.94,37.02)}{\pgfxy(29.98,33.55)}{\pgfxy(30.00,30.00)}\pgfstroke
\pgfmoveto{\pgfxy(60.00,60.00)}\pgfcurveto{\pgfxy(56.71,62.07)}{\pgfxy(53.37,64.07)}{\pgfxy(50.00,66.00)}\pgfcurveto{\pgfxy(45.85,68.38)}{\pgfxy(41.61,70.67)}{\pgfxy(37.00,72.00)}\pgfcurveto{\pgfxy(33.68,72.96)}{\pgfxy(30.00,74.27)}{\pgfxy(30.00,77.50)}\pgfcurveto{\pgfxy(30.00,79.99)}{\pgfxy(32.38,81.65)}{\pgfxy(35.00,82.00)}\pgfcurveto{\pgfxy(38.11,82.42)}{\pgfxy(41.15,81.30)}{\pgfxy(44.00,80.00)}\pgfcurveto{\pgfxy(49.08,77.69)}{\pgfxy(53.95,74.68)}{\pgfxy(57.00,70.00)}\pgfcurveto{\pgfxy(58.94,67.02)}{\pgfxy(59.98,63.55)}{\pgfxy(60.00,60.00)}\pgfstroke
\pgfmoveto{\pgfxy(50.00,-11.00)}\pgfcurveto{\pgfxy(53.29,-8.93)}{\pgfxy(56.63,-6.93)}{\pgfxy(60.00,-5.00)}\pgfcurveto{\pgfxy(64.15,-2.62)}{\pgfxy(68.39,-0.33)}{\pgfxy(73.00,1.00)}\pgfcurveto{\pgfxy(76.32,1.96)}{\pgfxy(80.00,3.27)}{\pgfxy(80.00,6.50)}\pgfcurveto{\pgfxy(80.00,8.99)}{\pgfxy(77.62,10.65)}{\pgfxy(75.00,11.00)}\pgfcurveto{\pgfxy(71.89,11.42)}{\pgfxy(68.85,10.30)}{\pgfxy(66.00,9.00)}\pgfcurveto{\pgfxy(60.92,6.69)}{\pgfxy(56.05,3.68)}{\pgfxy(53.00,-1.00)}\pgfcurveto{\pgfxy(51.06,-3.98)}{\pgfxy(50.02,-7.45)}{\pgfxy(50.00,-11.00)}\pgfstroke
\pgfmoveto{\pgfxy(30.00,10.00)}\pgfcurveto{\pgfxy(33.29,12.07)}{\pgfxy(36.63,14.07)}{\pgfxy(40.00,16.00)}\pgfcurveto{\pgfxy(44.15,18.38)}{\pgfxy(48.39,20.67)}{\pgfxy(53.00,22.00)}\pgfcurveto{\pgfxy(56.32,22.96)}{\pgfxy(60.00,24.27)}{\pgfxy(60.00,27.50)}\pgfcurveto{\pgfxy(60.00,29.99)}{\pgfxy(57.62,31.65)}{\pgfxy(55.00,32.00)}\pgfcurveto{\pgfxy(51.89,32.42)}{\pgfxy(48.85,31.30)}{\pgfxy(46.00,30.00)}\pgfcurveto{\pgfxy(40.92,27.69)}{\pgfxy(36.05,24.68)}{\pgfxy(33.00,20.00)}\pgfcurveto{\pgfxy(31.06,17.02)}{\pgfxy(30.02,13.55)}{\pgfxy(30.00,10.00)}\pgfstroke
\pgfmoveto{\pgfxy(60.00,60.00)}\pgfcurveto{\pgfxy(63.29,62.07)}{\pgfxy(66.63,64.07)}{\pgfxy(70.00,66.00)}\pgfcurveto{\pgfxy(74.15,68.38)}{\pgfxy(78.39,70.67)}{\pgfxy(83.00,72.00)}\pgfcurveto{\pgfxy(86.32,72.96)}{\pgfxy(90.00,74.27)}{\pgfxy(90.00,77.50)}\pgfcurveto{\pgfxy(90.00,79.99)}{\pgfxy(87.62,81.65)}{\pgfxy(85.00,82.00)}\pgfcurveto{\pgfxy(81.89,82.42)}{\pgfxy(78.85,81.30)}{\pgfxy(76.00,80.00)}\pgfcurveto{\pgfxy(70.92,77.69)}{\pgfxy(66.05,74.68)}{\pgfxy(63.00,70.00)}\pgfcurveto{\pgfxy(61.06,67.02)}{\pgfxy(60.02,63.55)}{\pgfxy(60.00,60.00)}\pgfstroke
\pgfputat{\pgfxy(128.00,27.00)}{\pgfbox[bottom,left]{\fontsize{6.26}{7.51}\selectfont \makebox[0pt][r]{$\sigma_{2i-1}$}}}
\pgfputat{\pgfxy(138.00,17.00)}{\pgfbox[bottom,left]{\fontsize{6.26}{7.51}\selectfont \makebox[0pt][r]{$\sigma_{2i}$}}}
\pgfmoveto{\pgfxy(150.00,10.00)}\pgflineto{\pgfxy(170.00,-10.00)}\pgfstroke
\pgfcircle[fill]{\pgfxy(150.00,10.00)}{0.55mm}
\pgfcircle[stroke]{\pgfxy(150.00,10.00)}{0.55mm}
\pgfcircle[fill]{\pgfxy(170.00,-10.00)}{0.55mm}
\pgfcircle[stroke]{\pgfxy(170.00,-10.00)}{0.55mm}
\pgfsetdash{{0.60mm}{0.50mm}}{0mm}\pgfsetlinewidth{0.60mm}\pgfmoveto{\pgfxy(170.00,0.00)}\pgflineto{\pgfxy(160.00,10.00)}\pgfstroke
\pgfsetdash{}{0mm}\pgfsetlinewidth{0.30mm}\pgfmoveto{\pgfxy(170.00,-11.00)}\pgfcurveto{\pgfxy(173.29,-8.93)}{\pgfxy(176.63,-6.93)}{\pgfxy(180.00,-5.00)}\pgfcurveto{\pgfxy(184.15,-2.62)}{\pgfxy(188.39,-0.33)}{\pgfxy(193.00,1.00)}\pgfcurveto{\pgfxy(196.32,1.96)}{\pgfxy(200.00,3.27)}{\pgfxy(200.00,6.50)}\pgfcurveto{\pgfxy(200.00,8.99)}{\pgfxy(197.62,10.65)}{\pgfxy(195.00,11.00)}\pgfcurveto{\pgfxy(191.89,11.42)}{\pgfxy(188.85,10.30)}{\pgfxy(186.00,9.00)}\pgfcurveto{\pgfxy(180.92,6.69)}{\pgfxy(176.05,3.68)}{\pgfxy(173.00,-1.00)}\pgfcurveto{\pgfxy(171.06,-3.98)}{\pgfxy(170.02,-7.45)}{\pgfxy(170.00,-11.00)}\pgfstroke
\pgfmoveto{\pgfxy(150.00,10.00)}\pgfcurveto{\pgfxy(153.29,12.07)}{\pgfxy(156.63,14.07)}{\pgfxy(160.00,16.00)}\pgfcurveto{\pgfxy(164.15,18.38)}{\pgfxy(168.39,20.67)}{\pgfxy(173.00,22.00)}\pgfcurveto{\pgfxy(176.32,22.96)}{\pgfxy(180.00,24.27)}{\pgfxy(180.00,27.50)}\pgfcurveto{\pgfxy(180.00,29.99)}{\pgfxy(177.62,31.65)}{\pgfxy(175.00,32.00)}\pgfcurveto{\pgfxy(171.89,32.42)}{\pgfxy(168.85,31.30)}{\pgfxy(166.00,30.00)}\pgfcurveto{\pgfxy(160.92,27.69)}{\pgfxy(156.05,24.68)}{\pgfxy(153.00,20.00)}\pgfcurveto{\pgfxy(151.06,17.02)}{\pgfxy(150.02,13.55)}{\pgfxy(150.00,10.00)}\pgfstroke
\pgfmoveto{\pgfxy(150.00,30.00)}\pgflineto{\pgfxy(140.00,20.00)}\pgfstroke
\pgfcircle[fill]{\pgfxy(150.00,30.00)}{0.55mm}
\pgfcircle[stroke]{\pgfxy(150.00,30.00)}{0.55mm}
\pgfcircle[fill]{\pgfxy(160.00,40.00)}{0.55mm}
\pgfcircle[stroke]{\pgfxy(155.00,49.00)}{0.00mm}
\pgfputat{\pgfxy(149.00,25.00)}{\pgfbox[bottom,left]{\fontsize{6.26}{7.51}\selectfont $a^{(i)}$}}
\pgfsetlinewidth{0.90mm}\pgfmoveto{\pgfxy(150.00,30.00)}\pgflineto{\pgfxy(190.00,70.00)}\pgfstroke
\pgfsetlinewidth{0.30mm}\pgfmoveto{\pgfxy(150.00,30.00)}\pgfcurveto{\pgfxy(146.71,32.07)}{\pgfxy(143.37,34.07)}{\pgfxy(140.00,36.00)}\pgfcurveto{\pgfxy(135.85,38.38)}{\pgfxy(131.61,40.67)}{\pgfxy(127.00,42.00)}\pgfcurveto{\pgfxy(123.68,42.96)}{\pgfxy(120.00,44.27)}{\pgfxy(120.00,47.50)}\pgfcurveto{\pgfxy(120.00,49.99)}{\pgfxy(122.38,51.65)}{\pgfxy(125.00,52.00)}\pgfcurveto{\pgfxy(128.11,52.42)}{\pgfxy(131.15,51.30)}{\pgfxy(134.00,50.00)}\pgfcurveto{\pgfxy(139.08,47.69)}{\pgfxy(143.95,44.68)}{\pgfxy(147.00,40.00)}\pgfcurveto{\pgfxy(148.94,37.02)}{\pgfxy(149.98,33.55)}{\pgfxy(150.00,30.00)}\pgfstroke
\pgfputat{\pgfxy(124.00,46.00)}{\pgfbox[bottom,left]{\fontsize{6.26}{7.51}\selectfont $S_1$}}
\pgfmoveto{\pgfxy(150.00,50.00)}\pgflineto{\pgfxy(160.00,40.00)}\pgfstroke
\pgfcircle[fill]{\pgfxy(190.00,70.00)}{0.55mm}
\pgfcircle[stroke]{\pgfxy(190.00,70.00)}{0.55mm}
\pgfputat{\pgfxy(216.00,86.00)}{\pgfbox[bottom,left]{\fontsize{6.26}{7.51}\selectfont \makebox[0pt][r]{$S_{j+1}$}}}
\pgfputat{\pgfxy(190.00,65.00)}{\pgfbox[bottom,left]{\fontsize{6.26}{7.51}\selectfont $b^{(i)}$}}
\pgfputat{\pgfxy(164.00,86.00)}{\pgfbox[bottom,left]{\fontsize{6.26}{7.51}\selectfont $S_{j}$}}
\pgfsetdash{{0.60mm}{0.50mm}}{0mm}\pgfsetlinewidth{0.60mm}\pgfmoveto{\pgfxy(160.00,50.00)}\pgflineto{\pgfxy(170.00,60.00)}\pgfstroke
\pgfsetdash{}{0mm}\pgfsetlinewidth{0.30mm}\pgfmoveto{\pgfxy(140.00,40.00)}\pgflineto{\pgfxy(150.00,30.00)}\pgfstroke
\pgfcircle[fill]{\pgfxy(180.00,60.00)}{0.55mm}
\pgfcircle[stroke]{\pgfxy(180.00,60.00)}{0.55mm}
\pgfmoveto{\pgfxy(170.00,70.00)}\pgflineto{\pgfxy(180.00,60.00)}\pgfstroke
\pgfcircle[fill]{\pgfxy(170.00,70.00)}{0.55mm}
\pgfcircle[stroke]{\pgfxy(170.00,70.00)}{0.55mm}
\pgfputat{\pgfxy(154.00,76.00)}{\pgfbox[bottom,left]{\fontsize{6.26}{7.51}\selectfont $S_{j-1}$}}
\pgfputat{\pgfxy(134.00,56.00)}{\pgfbox[bottom,left]{\fontsize{6.26}{7.51}\selectfont $S_2$}}
\pgfcircle[fill]{\pgfxy(160.00,40.00)}{0.55mm}
\pgfcircle[stroke]{\pgfxy(160.00,40.00)}{0.55mm}
\pgfcircle[fill]{\pgfxy(150.00,50.00)}{0.55mm}
\pgfcircle[stroke]{\pgfxy(150.00,50.00)}{0.55mm}
\pgfcircle[fill]{\pgfxy(140.00,40.00)}{0.55mm}
\pgfcircle[stroke]{\pgfxy(140.00,40.00)}{0.55mm}
\pgfputat{\pgfxy(140.00,42.00)}{\pgfbox[bottom,left]{\fontsize{6.26}{7.51}\selectfont \makebox[0pt][r]{$u_1$}}}
\pgfputat{\pgfxy(168.00,64.00)}{\pgfbox[bottom,left]{\fontsize{6.26}{7.51}\selectfont \makebox[0pt][r]{$u_{j-1}$}}}
\pgfputat{\pgfxy(150.00,52.00)}{\pgfbox[bottom,left]{\fontsize{6.26}{7.51}\selectfont \makebox[0pt][r]{$u_2$}}}
\pgfputat{\pgfxy(159.00,35.00)}{\pgfbox[bottom,left]{\fontsize{6.26}{7.51}\selectfont $v_2$}}
\pgfputat{\pgfxy(180.00,55.00)}{\pgfbox[bottom,left]{\fontsize{6.26}{7.51}\selectfont $v_{j-1}$}}
\pgfmoveto{\pgfxy(180.00,60.00)}\pgfcurveto{\pgfxy(176.71,62.07)}{\pgfxy(173.37,64.07)}{\pgfxy(170.00,66.00)}\pgfcurveto{\pgfxy(165.85,68.38)}{\pgfxy(161.61,70.67)}{\pgfxy(157.00,72.00)}\pgfcurveto{\pgfxy(153.68,72.96)}{\pgfxy(150.00,74.27)}{\pgfxy(150.00,77.50)}\pgfcurveto{\pgfxy(150.00,79.99)}{\pgfxy(152.38,81.65)}{\pgfxy(155.00,82.00)}\pgfcurveto{\pgfxy(158.11,82.42)}{\pgfxy(161.15,81.30)}{\pgfxy(164.00,80.00)}\pgfcurveto{\pgfxy(169.08,77.69)}{\pgfxy(173.95,74.68)}{\pgfxy(177.00,70.00)}\pgfcurveto{\pgfxy(178.94,67.02)}{\pgfxy(179.98,63.55)}{\pgfxy(180.00,60.00)}\pgfstroke
\pgfmoveto{\pgfxy(160.00,40.00)}\pgfcurveto{\pgfxy(156.71,42.07)}{\pgfxy(153.37,44.07)}{\pgfxy(150.00,46.00)}\pgfcurveto{\pgfxy(145.85,48.38)}{\pgfxy(141.61,50.67)}{\pgfxy(137.00,52.00)}\pgfcurveto{\pgfxy(133.68,52.96)}{\pgfxy(130.00,54.27)}{\pgfxy(130.00,57.50)}\pgfcurveto{\pgfxy(130.00,59.99)}{\pgfxy(132.38,61.65)}{\pgfxy(135.00,62.00)}\pgfcurveto{\pgfxy(138.11,62.42)}{\pgfxy(141.15,61.30)}{\pgfxy(144.00,60.00)}\pgfcurveto{\pgfxy(149.08,57.69)}{\pgfxy(153.95,54.68)}{\pgfxy(157.00,50.00)}\pgfcurveto{\pgfxy(158.94,47.02)}{\pgfxy(159.98,43.55)}{\pgfxy(160.00,40.00)}\pgfstroke
\pgfmoveto{\pgfxy(190.00,70.00)}\pgfcurveto{\pgfxy(186.71,72.07)}{\pgfxy(183.37,74.07)}{\pgfxy(180.00,76.00)}\pgfcurveto{\pgfxy(175.85,78.38)}{\pgfxy(171.61,80.67)}{\pgfxy(167.00,82.00)}\pgfcurveto{\pgfxy(163.68,82.96)}{\pgfxy(160.00,84.27)}{\pgfxy(160.00,87.50)}\pgfcurveto{\pgfxy(160.00,89.99)}{\pgfxy(162.38,91.65)}{\pgfxy(165.00,92.00)}\pgfcurveto{\pgfxy(168.11,92.42)}{\pgfxy(171.15,91.30)}{\pgfxy(174.00,90.00)}\pgfcurveto{\pgfxy(179.08,87.69)}{\pgfxy(183.95,84.68)}{\pgfxy(187.00,80.00)}\pgfcurveto{\pgfxy(188.94,77.02)}{\pgfxy(189.98,73.55)}{\pgfxy(190.00,70.00)}\pgfstroke
\pgfmoveto{\pgfxy(190.00,70.00)}\pgfcurveto{\pgfxy(193.29,72.07)}{\pgfxy(196.63,74.07)}{\pgfxy(200.00,76.00)}\pgfcurveto{\pgfxy(204.15,78.38)}{\pgfxy(208.39,80.67)}{\pgfxy(213.00,82.00)}\pgfcurveto{\pgfxy(216.32,82.96)}{\pgfxy(220.00,84.27)}{\pgfxy(220.00,87.50)}\pgfcurveto{\pgfxy(220.00,89.99)}{\pgfxy(217.62,91.65)}{\pgfxy(215.00,92.00)}\pgfcurveto{\pgfxy(211.89,92.42)}{\pgfxy(208.85,91.30)}{\pgfxy(206.00,90.00)}\pgfcurveto{\pgfxy(200.92,87.69)}{\pgfxy(196.05,84.68)}{\pgfxy(193.00,80.00)}\pgfcurveto{\pgfxy(191.06,77.02)}{\pgfxy(190.02,73.55)}{\pgfxy(190.00,70.00)}\pgfstroke
\pgfmoveto{\pgfxy(150.00,10.00)}\pgflineto{\pgfxy(130.00,30.00)}\pgfstroke
\pgfcircle[fill]{\pgfxy(130.00,30.00)}{0.55mm}
\pgfcircle[stroke]{\pgfxy(130.00,30.00)}{0.55mm}
\pgfcircle[fill]{\pgfxy(140.00,20.00)}{0.55mm}
\pgfcircle[stroke]{\pgfxy(140.00,20.00)}{0.55mm}
\pgfputat{\pgfxy(100.00,37.00)}{\pgfbox[bottom,left]{\fontsize{6.26}{7.51}\selectfont \makebox[0pt]{(B1)}}}
\end{pgfpicture}%
$$
Due to $\sigma_{2i-1} > v_1$ $(= a^{(i)})$, 
it is not an increasing 1-2 tree 
and we apply (B2b) to the above tree as follows.
$$
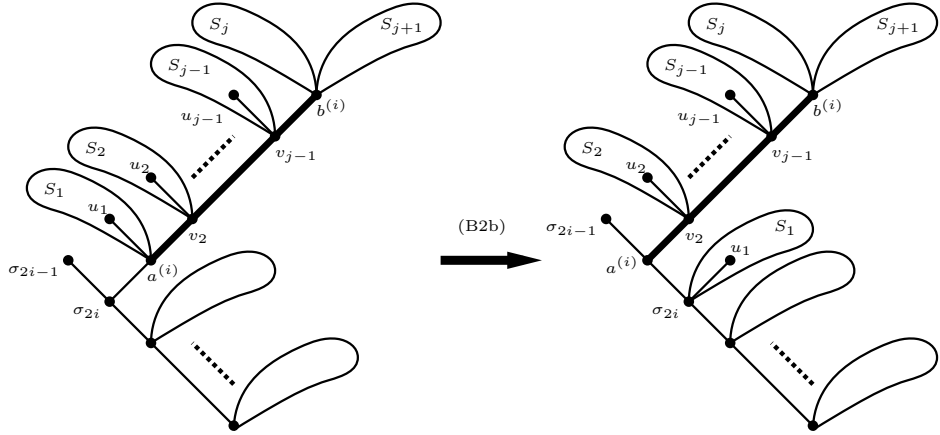
\centering
\begin{pgfpicture}{-79.00mm}{-8.05mm}{46.00mm}{52.65mm}
\pgfsetxvec{\pgfpoint{0.55mm}{0mm}}
\pgfsetyvec{\pgfpoint{0mm}{0.55mm}}
\color[rgb]{0,0,0}\pgfsetlinewidth{0.30mm}\pgfsetdash{}{0mm}
\pgfputat{\pgfxy(-2.00,37.00)}{\pgfbox[bottom,left]{\fontsize{6.26}{7.51}\selectfont \makebox[0pt][r]{$\sigma_{2i-1}$}}}
\pgfputat{\pgfxy(18.00,17.00)}{\pgfbox[bottom,left]{\fontsize{6.26}{7.51}\selectfont \makebox[0pt][r]{$\sigma_{2i}$}}}
\pgfmoveto{\pgfxy(30.00,10.00)}\pgflineto{\pgfxy(50.00,-10.00)}\pgfstroke
\pgfcircle[fill]{\pgfxy(30.00,10.00)}{0.55mm}
\pgfcircle[stroke]{\pgfxy(30.00,10.00)}{0.55mm}
\pgfcircle[fill]{\pgfxy(50.00,-10.00)}{0.55mm}
\pgfcircle[stroke]{\pgfxy(50.00,-10.00)}{0.55mm}
\pgfsetdash{{0.60mm}{0.50mm}}{0mm}\pgfsetlinewidth{0.60mm}\pgfmoveto{\pgfxy(50.00,0.00)}\pgflineto{\pgfxy(40.00,10.00)}\pgfstroke
\pgfsetdash{}{0mm}\pgfsetlinewidth{0.30mm}\pgfmoveto{\pgfxy(50.00,-11.00)}\pgfcurveto{\pgfxy(53.29,-8.93)}{\pgfxy(56.63,-6.93)}{\pgfxy(60.00,-5.00)}\pgfcurveto{\pgfxy(64.15,-2.62)}{\pgfxy(68.39,-0.33)}{\pgfxy(73.00,1.00)}\pgfcurveto{\pgfxy(76.32,1.96)}{\pgfxy(80.00,3.27)}{\pgfxy(80.00,6.50)}\pgfcurveto{\pgfxy(80.00,8.99)}{\pgfxy(77.62,10.65)}{\pgfxy(75.00,11.00)}\pgfcurveto{\pgfxy(71.89,11.42)}{\pgfxy(68.85,10.30)}{\pgfxy(66.00,9.00)}\pgfcurveto{\pgfxy(60.92,6.69)}{\pgfxy(56.05,3.68)}{\pgfxy(53.00,-1.00)}\pgfcurveto{\pgfxy(51.06,-3.98)}{\pgfxy(50.02,-7.45)}{\pgfxy(50.00,-11.00)}\pgfstroke
\pgfmoveto{\pgfxy(30.00,10.00)}\pgfcurveto{\pgfxy(33.29,12.07)}{\pgfxy(36.63,14.07)}{\pgfxy(40.00,16.00)}\pgfcurveto{\pgfxy(44.15,18.38)}{\pgfxy(48.39,20.67)}{\pgfxy(53.00,22.00)}\pgfcurveto{\pgfxy(56.32,22.96)}{\pgfxy(60.00,24.27)}{\pgfxy(60.00,27.50)}\pgfcurveto{\pgfxy(60.00,29.99)}{\pgfxy(57.62,31.65)}{\pgfxy(55.00,32.00)}\pgfcurveto{\pgfxy(51.89,32.42)}{\pgfxy(48.85,31.30)}{\pgfxy(46.00,30.00)}\pgfcurveto{\pgfxy(40.92,27.69)}{\pgfxy(36.05,24.68)}{\pgfxy(33.00,20.00)}\pgfcurveto{\pgfxy(31.06,17.02)}{\pgfxy(30.02,13.55)}{\pgfxy(30.00,10.00)}\pgfstroke
\pgfmoveto{\pgfxy(0.00,40.00)}\pgflineto{\pgfxy(10.00,30.00)}\pgfstroke
\pgfputat{\pgfxy(8.00,27.00)}{\pgfbox[bottom,left]{\fontsize{6.26}{7.51}\selectfont \makebox[0pt][r]{$a^{(i)}$}}}
\pgfmoveto{\pgfxy(20.00,20.00)}\pgfcurveto{\pgfxy(23.29,22.07)}{\pgfxy(26.63,24.07)}{\pgfxy(30.00,26.00)}\pgfcurveto{\pgfxy(34.15,28.38)}{\pgfxy(38.39,30.67)}{\pgfxy(43.00,32.00)}\pgfcurveto{\pgfxy(46.32,32.96)}{\pgfxy(50.00,34.27)}{\pgfxy(50.00,37.50)}\pgfcurveto{\pgfxy(50.00,39.99)}{\pgfxy(47.62,41.65)}{\pgfxy(45.00,42.00)}\pgfcurveto{\pgfxy(41.89,42.42)}{\pgfxy(38.85,41.30)}{\pgfxy(36.00,40.00)}\pgfcurveto{\pgfxy(30.92,37.69)}{\pgfxy(26.05,34.68)}{\pgfxy(23.00,30.00)}\pgfcurveto{\pgfxy(21.06,27.02)}{\pgfxy(20.02,23.55)}{\pgfxy(20.00,20.00)}\pgfstroke
\pgfputat{\pgfxy(46.00,37.00)}{\pgfbox[bottom,left]{\fontsize{6.26}{7.51}\selectfont \makebox[0pt][r]{$S_1$}}}
\pgfmoveto{\pgfxy(20.00,20.00)}\pgflineto{\pgfxy(30.00,30.00)}\pgfstroke
\pgfcircle[fill]{\pgfxy(30.00,30.00)}{0.55mm}
\pgfcircle[stroke]{\pgfxy(30.00,30.00)}{0.55mm}
\pgfputat{\pgfxy(30.00,32.00)}{\pgfbox[bottom,left]{\fontsize{6.26}{7.51}\selectfont $u_1$}}
\pgfcircle[fill]{\pgfxy(10.00,30.00)}{0.55mm}
\pgfcircle[stroke]{\pgfxy(10.00,30.00)}{0.55mm}
\pgfcircle[fill]{\pgfxy(20.00,40.00)}{0.55mm}
\pgfcircle[stroke]{\pgfxy(15.00,49.00)}{0.00mm}
\pgfmoveto{\pgfxy(10.00,50.00)}\pgflineto{\pgfxy(20.00,40.00)}\pgfstroke
\pgfcircle[fill]{\pgfxy(50.00,70.00)}{0.55mm}
\pgfcircle[stroke]{\pgfxy(50.00,70.00)}{0.55mm}
\pgfputat{\pgfxy(76.00,86.00)}{\pgfbox[bottom,left]{\fontsize{6.26}{7.51}\selectfont \makebox[0pt][r]{$S_{j+1}$}}}
\pgfputat{\pgfxy(50.00,65.00)}{\pgfbox[bottom,left]{\fontsize{6.26}{7.51}\selectfont $b^{(i)}$}}
\pgfputat{\pgfxy(24.00,86.00)}{\pgfbox[bottom,left]{\fontsize{6.26}{7.51}\selectfont $S_{j}$}}
\pgfsetdash{{0.60mm}{0.50mm}}{0mm}\pgfsetlinewidth{0.60mm}\pgfmoveto{\pgfxy(20.00,50.00)}\pgflineto{\pgfxy(30.00,60.00)}\pgfstroke
\pgfcircle[fill]{\pgfxy(40.00,60.00)}{0.55mm}
\pgfsetdash{}{0mm}\pgfsetlinewidth{0.30mm}\pgfcircle[stroke]{\pgfxy(40.00,60.00)}{0.55mm}
\pgfmoveto{\pgfxy(30.00,70.00)}\pgflineto{\pgfxy(40.00,60.00)}\pgfstroke
\pgfcircle[fill]{\pgfxy(30.00,70.00)}{0.55mm}
\pgfcircle[stroke]{\pgfxy(30.00,70.00)}{0.55mm}
\pgfputat{\pgfxy(14.00,76.00)}{\pgfbox[bottom,left]{\fontsize{6.26}{7.51}\selectfont $S_{j-1}$}}
\pgfputat{\pgfxy(-6.00,56.00)}{\pgfbox[bottom,left]{\fontsize{6.26}{7.51}\selectfont $S_2$}}
\pgfcircle[fill]{\pgfxy(20.00,40.00)}{0.55mm}
\pgfcircle[stroke]{\pgfxy(20.00,40.00)}{0.55mm}
\pgfcircle[fill]{\pgfxy(10.00,50.00)}{0.55mm}
\pgfcircle[stroke]{\pgfxy(10.00,50.00)}{0.55mm}
\pgfputat{\pgfxy(28.00,64.00)}{\pgfbox[bottom,left]{\fontsize{6.26}{7.51}\selectfont \makebox[0pt][r]{$u_{j-1}$}}}
\pgfputat{\pgfxy(10.00,52.00)}{\pgfbox[bottom,left]{\fontsize{6.26}{7.51}\selectfont \makebox[0pt][r]{$u_2$}}}
\pgfputat{\pgfxy(40.00,55.00)}{\pgfbox[bottom,left]{\fontsize{6.26}{7.51}\selectfont $v_{j-1}$}}
\pgfmoveto{\pgfxy(40.00,60.00)}\pgfcurveto{\pgfxy(36.71,62.07)}{\pgfxy(33.37,64.07)}{\pgfxy(30.00,66.00)}\pgfcurveto{\pgfxy(25.85,68.38)}{\pgfxy(21.61,70.67)}{\pgfxy(17.00,72.00)}\pgfcurveto{\pgfxy(13.68,72.96)}{\pgfxy(10.00,74.27)}{\pgfxy(10.00,77.50)}\pgfcurveto{\pgfxy(10.00,79.99)}{\pgfxy(12.38,81.65)}{\pgfxy(15.00,82.00)}\pgfcurveto{\pgfxy(18.11,82.42)}{\pgfxy(21.15,81.30)}{\pgfxy(24.00,80.00)}\pgfcurveto{\pgfxy(29.08,77.69)}{\pgfxy(33.95,74.68)}{\pgfxy(37.00,70.00)}\pgfcurveto{\pgfxy(38.94,67.02)}{\pgfxy(39.98,63.55)}{\pgfxy(40.00,60.00)}\pgfstroke
\pgfmoveto{\pgfxy(20.00,40.00)}\pgfcurveto{\pgfxy(16.71,42.07)}{\pgfxy(13.37,44.07)}{\pgfxy(10.00,46.00)}\pgfcurveto{\pgfxy(5.85,48.38)}{\pgfxy(1.61,50.67)}{\pgfxy(-3.00,52.00)}\pgfcurveto{\pgfxy(-6.32,52.96)}{\pgfxy(-10.00,54.27)}{\pgfxy(-10.00,57.50)}\pgfcurveto{\pgfxy(-10.00,59.99)}{\pgfxy(-7.62,61.65)}{\pgfxy(-5.00,62.00)}\pgfcurveto{\pgfxy(-1.89,62.42)}{\pgfxy(1.15,61.30)}{\pgfxy(4.00,60.00)}\pgfcurveto{\pgfxy(9.08,57.69)}{\pgfxy(13.95,54.68)}{\pgfxy(17.00,50.00)}\pgfcurveto{\pgfxy(18.94,47.02)}{\pgfxy(19.98,43.55)}{\pgfxy(20.00,40.00)}\pgfstroke
\pgfmoveto{\pgfxy(50.00,70.00)}\pgfcurveto{\pgfxy(46.71,72.07)}{\pgfxy(43.37,74.07)}{\pgfxy(40.00,76.00)}\pgfcurveto{\pgfxy(35.85,78.38)}{\pgfxy(31.61,80.67)}{\pgfxy(27.00,82.00)}\pgfcurveto{\pgfxy(23.68,82.96)}{\pgfxy(20.00,84.27)}{\pgfxy(20.00,87.50)}\pgfcurveto{\pgfxy(20.00,89.99)}{\pgfxy(22.38,91.65)}{\pgfxy(25.00,92.00)}\pgfcurveto{\pgfxy(28.11,92.42)}{\pgfxy(31.15,91.30)}{\pgfxy(34.00,90.00)}\pgfcurveto{\pgfxy(39.08,87.69)}{\pgfxy(43.95,84.68)}{\pgfxy(47.00,80.00)}\pgfcurveto{\pgfxy(48.94,77.02)}{\pgfxy(49.98,73.55)}{\pgfxy(50.00,70.00)}\pgfstroke
\pgfmoveto{\pgfxy(50.00,70.00)}\pgfcurveto{\pgfxy(53.29,72.07)}{\pgfxy(56.63,74.07)}{\pgfxy(60.00,76.00)}\pgfcurveto{\pgfxy(64.15,78.38)}{\pgfxy(68.39,80.67)}{\pgfxy(73.00,82.00)}\pgfcurveto{\pgfxy(76.32,82.96)}{\pgfxy(80.00,84.27)}{\pgfxy(80.00,87.50)}\pgfcurveto{\pgfxy(80.00,89.99)}{\pgfxy(77.62,91.65)}{\pgfxy(75.00,92.00)}\pgfcurveto{\pgfxy(71.89,92.42)}{\pgfxy(68.85,91.30)}{\pgfxy(66.00,90.00)}\pgfcurveto{\pgfxy(60.92,87.69)}{\pgfxy(56.05,84.68)}{\pgfxy(53.00,80.00)}\pgfcurveto{\pgfxy(51.06,77.02)}{\pgfxy(50.02,73.55)}{\pgfxy(50.00,70.00)}\pgfstroke
\pgfsetlinewidth{0.90mm}\pgfmoveto{\pgfxy(10.00,30.00)}\pgflineto{\pgfxy(50.00,70.00)}\pgfstroke
\pgfputat{\pgfxy(19.00,35.00)}{\pgfbox[bottom,left]{\fontsize{6.26}{7.51}\selectfont $v_2$}}
\pgfsetlinewidth{0.30mm}\pgfmoveto{\pgfxy(30.00,10.00)}\pgflineto{\pgfxy(10.00,30.00)}\pgfstroke
\pgfcircle[fill]{\pgfxy(0.00,40.00)}{0.55mm}
\pgfcircle[stroke]{\pgfxy(0.00,40.00)}{0.55mm}
\pgfcircle[fill]{\pgfxy(20.00,20.00)}{0.55mm}
\pgfcircle[stroke]{\pgfxy(20.00,20.00)}{0.55mm}
\pgfsetlinewidth{1.20mm}\pgfmoveto{\pgfxy(-40.00,30.00)}\pgflineto{\pgfxy(-20.00,30.00)}\pgfstroke
\pgfmoveto{\pgfxy(-20.00,30.00)}\pgflineto{\pgfxy(-22.80,30.70)}\pgflineto{\pgfxy(-22.80,29.30)}\pgflineto{\pgfxy(-20.00,30.00)}\pgfclosepath\pgffill
\pgfmoveto{\pgfxy(-20.00,30.00)}\pgflineto{\pgfxy(-22.80,30.70)}\pgflineto{\pgfxy(-22.80,29.30)}\pgflineto{\pgfxy(-20.00,30.00)}\pgfclosepath\pgfstroke
\pgfputat{\pgfxy(-132.00,27.00)}{\pgfbox[bottom,left]{\fontsize{6.26}{7.51}\selectfont \makebox[0pt][r]{$\sigma_{2i-1}$}}}
\pgfputat{\pgfxy(-122.00,17.00)}{\pgfbox[bottom,left]{\fontsize{6.26}{7.51}\selectfont \makebox[0pt][r]{$\sigma_{2i}$}}}
\pgfsetlinewidth{0.30mm}\pgfmoveto{\pgfxy(-110.00,10.00)}\pgflineto{\pgfxy(-90.00,-10.00)}\pgfstroke
\pgfcircle[fill]{\pgfxy(-110.00,10.00)}{0.55mm}
\pgfcircle[stroke]{\pgfxy(-110.00,10.00)}{0.55mm}
\pgfcircle[fill]{\pgfxy(-90.00,-10.00)}{0.55mm}
\pgfcircle[stroke]{\pgfxy(-90.00,-10.00)}{0.55mm}
\pgfsetdash{{0.60mm}{0.50mm}}{0mm}\pgfsetlinewidth{0.60mm}\pgfmoveto{\pgfxy(-90.00,0.00)}\pgflineto{\pgfxy(-100.00,10.00)}\pgfstroke
\pgfsetdash{}{0mm}\pgfsetlinewidth{0.30mm}\pgfmoveto{\pgfxy(-90.00,-11.00)}\pgfcurveto{\pgfxy(-86.71,-8.93)}{\pgfxy(-83.37,-6.93)}{\pgfxy(-80.00,-5.00)}\pgfcurveto{\pgfxy(-75.85,-2.62)}{\pgfxy(-71.61,-0.33)}{\pgfxy(-67.00,1.00)}\pgfcurveto{\pgfxy(-63.68,1.96)}{\pgfxy(-60.00,3.27)}{\pgfxy(-60.00,6.50)}\pgfcurveto{\pgfxy(-60.00,8.99)}{\pgfxy(-62.38,10.65)}{\pgfxy(-65.00,11.00)}\pgfcurveto{\pgfxy(-68.11,11.42)}{\pgfxy(-71.15,10.30)}{\pgfxy(-74.00,9.00)}\pgfcurveto{\pgfxy(-79.08,6.69)}{\pgfxy(-83.95,3.68)}{\pgfxy(-87.00,-1.00)}\pgfcurveto{\pgfxy(-88.94,-3.98)}{\pgfxy(-89.98,-7.45)}{\pgfxy(-90.00,-11.00)}\pgfstroke
\pgfmoveto{\pgfxy(-110.00,10.00)}\pgfcurveto{\pgfxy(-106.71,12.07)}{\pgfxy(-103.37,14.07)}{\pgfxy(-100.00,16.00)}\pgfcurveto{\pgfxy(-95.85,18.38)}{\pgfxy(-91.61,20.67)}{\pgfxy(-87.00,22.00)}\pgfcurveto{\pgfxy(-83.68,22.96)}{\pgfxy(-80.00,24.27)}{\pgfxy(-80.00,27.50)}\pgfcurveto{\pgfxy(-80.00,29.99)}{\pgfxy(-82.38,31.65)}{\pgfxy(-85.00,32.00)}\pgfcurveto{\pgfxy(-88.11,32.42)}{\pgfxy(-91.15,31.30)}{\pgfxy(-94.00,30.00)}\pgfcurveto{\pgfxy(-99.08,27.69)}{\pgfxy(-103.95,24.68)}{\pgfxy(-107.00,20.00)}\pgfcurveto{\pgfxy(-108.94,17.02)}{\pgfxy(-109.98,13.55)}{\pgfxy(-110.00,10.00)}\pgfstroke
\pgfmoveto{\pgfxy(-110.00,30.00)}\pgflineto{\pgfxy(-120.00,20.00)}\pgfstroke
\pgfcircle[fill]{\pgfxy(-110.00,30.00)}{0.55mm}
\pgfcircle[stroke]{\pgfxy(-110.00,30.00)}{0.55mm}
\pgfcircle[fill]{\pgfxy(-100.00,40.00)}{0.55mm}
\pgfcircle[stroke]{\pgfxy(-105.00,49.00)}{0.00mm}
\pgfputat{\pgfxy(-111.00,25.00)}{\pgfbox[bottom,left]{\fontsize{6.26}{7.51}\selectfont $a^{(i)}$}}
\pgfsetlinewidth{0.90mm}\pgfmoveto{\pgfxy(-110.00,30.00)}\pgflineto{\pgfxy(-70.00,70.00)}\pgfstroke
\pgfsetlinewidth{0.30mm}\pgfmoveto{\pgfxy(-110.00,30.00)}\pgfcurveto{\pgfxy(-113.29,32.07)}{\pgfxy(-116.63,34.07)}{\pgfxy(-120.00,36.00)}\pgfcurveto{\pgfxy(-124.15,38.38)}{\pgfxy(-128.39,40.67)}{\pgfxy(-133.00,42.00)}\pgfcurveto{\pgfxy(-136.32,42.96)}{\pgfxy(-140.00,44.27)}{\pgfxy(-140.00,47.50)}\pgfcurveto{\pgfxy(-140.00,49.99)}{\pgfxy(-137.62,51.65)}{\pgfxy(-135.00,52.00)}\pgfcurveto{\pgfxy(-131.89,52.42)}{\pgfxy(-128.85,51.30)}{\pgfxy(-126.00,50.00)}\pgfcurveto{\pgfxy(-120.92,47.69)}{\pgfxy(-116.05,44.68)}{\pgfxy(-113.00,40.00)}\pgfcurveto{\pgfxy(-111.06,37.02)}{\pgfxy(-110.02,33.55)}{\pgfxy(-110.00,30.00)}\pgfstroke
\pgfputat{\pgfxy(-136.00,46.00)}{\pgfbox[bottom,left]{\fontsize{6.26}{7.51}\selectfont $S_1$}}
\pgfmoveto{\pgfxy(-110.00,50.00)}\pgflineto{\pgfxy(-100.00,40.00)}\pgfstroke
\pgfcircle[fill]{\pgfxy(-70.00,70.00)}{0.55mm}
\pgfcircle[stroke]{\pgfxy(-70.00,70.00)}{0.55mm}
\pgfputat{\pgfxy(-44.00,86.00)}{\pgfbox[bottom,left]{\fontsize{6.26}{7.51}\selectfont \makebox[0pt][r]{$S_{j+1}$}}}
\pgfputat{\pgfxy(-70.00,65.00)}{\pgfbox[bottom,left]{\fontsize{6.26}{7.51}\selectfont $b^{(i)}$}}
\pgfputat{\pgfxy(-96.00,86.00)}{\pgfbox[bottom,left]{\fontsize{6.26}{7.51}\selectfont $S_{j}$}}
\pgfsetdash{{0.60mm}{0.50mm}}{0mm}\pgfsetlinewidth{0.60mm}\pgfmoveto{\pgfxy(-100.00,50.00)}\pgflineto{\pgfxy(-90.00,60.00)}\pgfstroke
\pgfsetdash{}{0mm}\pgfsetlinewidth{0.30mm}\pgfmoveto{\pgfxy(-120.00,40.00)}\pgflineto{\pgfxy(-110.00,30.00)}\pgfstroke
\pgfcircle[fill]{\pgfxy(-80.00,60.00)}{0.55mm}
\pgfcircle[stroke]{\pgfxy(-80.00,60.00)}{0.55mm}
\pgfmoveto{\pgfxy(-90.00,70.00)}\pgflineto{\pgfxy(-80.00,60.00)}\pgfstroke
\pgfcircle[fill]{\pgfxy(-90.00,70.00)}{0.55mm}
\pgfcircle[stroke]{\pgfxy(-90.00,70.00)}{0.55mm}
\pgfputat{\pgfxy(-106.00,76.00)}{\pgfbox[bottom,left]{\fontsize{6.26}{7.51}\selectfont $S_{j-1}$}}
\pgfputat{\pgfxy(-126.00,56.00)}{\pgfbox[bottom,left]{\fontsize{6.26}{7.51}\selectfont $S_2$}}
\pgfcircle[fill]{\pgfxy(-100.00,40.00)}{0.55mm}
\pgfcircle[stroke]{\pgfxy(-100.00,40.00)}{0.55mm}
\pgfcircle[fill]{\pgfxy(-110.00,50.00)}{0.55mm}
\pgfcircle[stroke]{\pgfxy(-110.00,50.00)}{0.55mm}
\pgfcircle[fill]{\pgfxy(-120.00,40.00)}{0.55mm}
\pgfcircle[stroke]{\pgfxy(-120.00,40.00)}{0.55mm}
\pgfputat{\pgfxy(-120.00,42.00)}{\pgfbox[bottom,left]{\fontsize{6.26}{7.51}\selectfont \makebox[0pt][r]{$u_1$}}}
\pgfputat{\pgfxy(-92.00,64.00)}{\pgfbox[bottom,left]{\fontsize{6.26}{7.51}\selectfont \makebox[0pt][r]{$u_{j-1}$}}}
\pgfputat{\pgfxy(-110.00,52.00)}{\pgfbox[bottom,left]{\fontsize{6.26}{7.51}\selectfont \makebox[0pt][r]{$u_2$}}}
\pgfputat{\pgfxy(-101.00,35.00)}{\pgfbox[bottom,left]{\fontsize{6.26}{7.51}\selectfont $v_2$}}
\pgfputat{\pgfxy(-80.00,55.00)}{\pgfbox[bottom,left]{\fontsize{6.26}{7.51}\selectfont $v_{j-1}$}}
\pgfmoveto{\pgfxy(-80.00,60.00)}\pgfcurveto{\pgfxy(-83.29,62.07)}{\pgfxy(-86.63,64.07)}{\pgfxy(-90.00,66.00)}\pgfcurveto{\pgfxy(-94.15,68.38)}{\pgfxy(-98.39,70.67)}{\pgfxy(-103.00,72.00)}\pgfcurveto{\pgfxy(-106.32,72.96)}{\pgfxy(-110.00,74.27)}{\pgfxy(-110.00,77.50)}\pgfcurveto{\pgfxy(-110.00,79.99)}{\pgfxy(-107.62,81.65)}{\pgfxy(-105.00,82.00)}\pgfcurveto{\pgfxy(-101.89,82.42)}{\pgfxy(-98.85,81.30)}{\pgfxy(-96.00,80.00)}\pgfcurveto{\pgfxy(-90.92,77.69)}{\pgfxy(-86.05,74.68)}{\pgfxy(-83.00,70.00)}\pgfcurveto{\pgfxy(-81.06,67.02)}{\pgfxy(-80.02,63.55)}{\pgfxy(-80.00,60.00)}\pgfstroke
\pgfmoveto{\pgfxy(-100.00,40.00)}\pgfcurveto{\pgfxy(-103.29,42.07)}{\pgfxy(-106.63,44.07)}{\pgfxy(-110.00,46.00)}\pgfcurveto{\pgfxy(-114.15,48.38)}{\pgfxy(-118.39,50.67)}{\pgfxy(-123.00,52.00)}\pgfcurveto{\pgfxy(-126.32,52.96)}{\pgfxy(-130.00,54.27)}{\pgfxy(-130.00,57.50)}\pgfcurveto{\pgfxy(-130.00,59.99)}{\pgfxy(-127.62,61.65)}{\pgfxy(-125.00,62.00)}\pgfcurveto{\pgfxy(-121.89,62.42)}{\pgfxy(-118.85,61.30)}{\pgfxy(-116.00,60.00)}\pgfcurveto{\pgfxy(-110.92,57.69)}{\pgfxy(-106.05,54.68)}{\pgfxy(-103.00,50.00)}\pgfcurveto{\pgfxy(-101.06,47.02)}{\pgfxy(-100.02,43.55)}{\pgfxy(-100.00,40.00)}\pgfstroke
\pgfmoveto{\pgfxy(-70.00,70.00)}\pgfcurveto{\pgfxy(-73.29,72.07)}{\pgfxy(-76.63,74.07)}{\pgfxy(-80.00,76.00)}\pgfcurveto{\pgfxy(-84.15,78.38)}{\pgfxy(-88.39,80.67)}{\pgfxy(-93.00,82.00)}\pgfcurveto{\pgfxy(-96.32,82.96)}{\pgfxy(-100.00,84.27)}{\pgfxy(-100.00,87.50)}\pgfcurveto{\pgfxy(-100.00,89.99)}{\pgfxy(-97.62,91.65)}{\pgfxy(-95.00,92.00)}\pgfcurveto{\pgfxy(-91.89,92.42)}{\pgfxy(-88.85,91.30)}{\pgfxy(-86.00,90.00)}\pgfcurveto{\pgfxy(-80.92,87.69)}{\pgfxy(-76.05,84.68)}{\pgfxy(-73.00,80.00)}\pgfcurveto{\pgfxy(-71.06,77.02)}{\pgfxy(-70.02,73.55)}{\pgfxy(-70.00,70.00)}\pgfstroke
\pgfmoveto{\pgfxy(-70.00,70.00)}\pgfcurveto{\pgfxy(-66.71,72.07)}{\pgfxy(-63.37,74.07)}{\pgfxy(-60.00,76.00)}\pgfcurveto{\pgfxy(-55.85,78.38)}{\pgfxy(-51.61,80.67)}{\pgfxy(-47.00,82.00)}\pgfcurveto{\pgfxy(-43.68,82.96)}{\pgfxy(-40.00,84.27)}{\pgfxy(-40.00,87.50)}\pgfcurveto{\pgfxy(-40.00,89.99)}{\pgfxy(-42.38,91.65)}{\pgfxy(-45.00,92.00)}\pgfcurveto{\pgfxy(-48.11,92.42)}{\pgfxy(-51.15,91.30)}{\pgfxy(-54.00,90.00)}\pgfcurveto{\pgfxy(-59.08,87.69)}{\pgfxy(-63.95,84.68)}{\pgfxy(-67.00,80.00)}\pgfcurveto{\pgfxy(-68.94,77.02)}{\pgfxy(-69.98,73.55)}{\pgfxy(-70.00,70.00)}\pgfstroke
\pgfmoveto{\pgfxy(-110.00,10.00)}\pgflineto{\pgfxy(-130.00,30.00)}\pgfstroke
\pgfcircle[fill]{\pgfxy(-130.00,30.00)}{0.55mm}
\pgfcircle[stroke]{\pgfxy(-130.00,30.00)}{0.55mm}
\pgfcircle[fill]{\pgfxy(-120.00,20.00)}{0.55mm}
\pgfcircle[stroke]{\pgfxy(-120.00,20.00)}{0.55mm}
\pgfputat{\pgfxy(-30.00,37.00)}{\pgfbox[bottom,left]{\fontsize{6.26}{7.51}\selectfont \makebox[0pt]{(B2b)}}}
\end{pgfpicture}%
$$
Since $\sigma_{2i-1} > v_j > \dots > v_2$, until we have an increasing 1-2 tree, repeat to apply (B2b) as follows.
$$
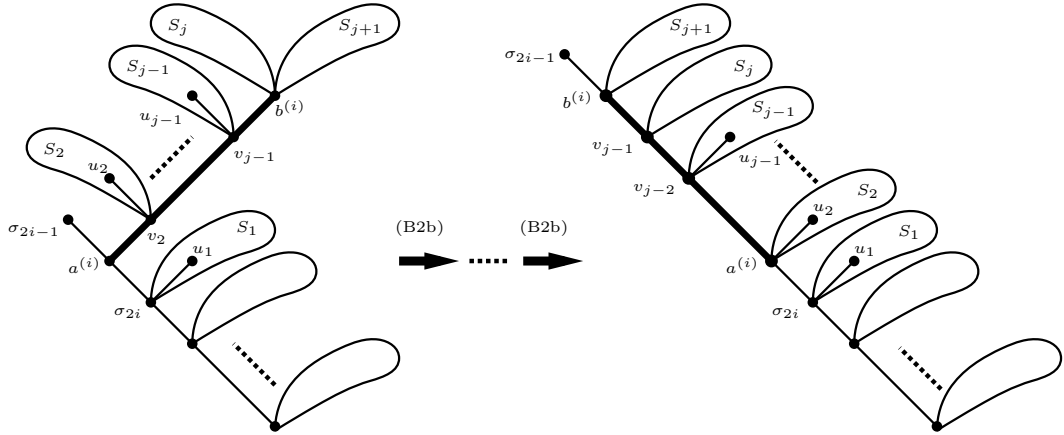
\centering
\begin{pgfpicture}{-73.50mm}{-8.05mm}{68.00mm}{52.65mm}
\pgfsetxvec{\pgfpoint{0.55mm}{0mm}}
\pgfsetyvec{\pgfpoint{0mm}{0.55mm}}
\color[rgb]{0,0,0}\pgfsetlinewidth{0.30mm}\pgfsetdash{}{0mm}
\pgfsetlinewidth{1.20mm}\pgfmoveto{\pgfxy(-10.00,30.00)}\pgflineto{\pgfxy(0.00,30.00)}\pgfstroke
\pgfmoveto{\pgfxy(0.00,30.00)}\pgflineto{\pgfxy(-2.80,30.70)}\pgflineto{\pgfxy(-2.80,29.30)}\pgflineto{\pgfxy(0.00,30.00)}\pgfclosepath\pgffill
\pgfmoveto{\pgfxy(0.00,30.00)}\pgflineto{\pgfxy(-2.80,30.70)}\pgflineto{\pgfxy(-2.80,29.30)}\pgflineto{\pgfxy(0.00,30.00)}\pgfclosepath\pgfstroke
\pgfcircle[fill]{\pgfxy(60.00,20.00)}{0.55mm}
\pgfsetlinewidth{0.30mm}\pgfcircle[stroke]{\pgfxy(60.00,20.00)}{0.55mm}
\pgfcircle[stroke]{\pgfxy(45.00,39.00)}{0.00mm}
\pgfmoveto{\pgfxy(60.00,20.00)}\pgflineto{\pgfxy(70.00,10.00)}\pgfstroke
\pgfputat{\pgfxy(47.00,27.00)}{\pgfbox[bottom,left]{\fontsize{6.26}{7.51}\selectfont \makebox[0pt][r]{$a^{(i)}$}}}
\pgfsetlinewidth{0.90mm}\pgfmoveto{\pgfxy(50.00,30.00)}\pgflineto{\pgfxy(10.00,70.00)}\pgfstroke
\pgfputat{\pgfxy(86.00,36.00)}{\pgfbox[bottom,left]{\fontsize{6.26}{7.51}\selectfont \makebox[0pt][r]{$S_1$}}}
\pgfsetdash{{0.60mm}{0.50mm}}{0mm}\pgfsetlinewidth{0.60mm}\pgfmoveto{\pgfxy(61.00,49.00)}\pgflineto{\pgfxy(51.00,59.00)}\pgfstroke
\pgfputat{\pgfxy(57.00,17.00)}{\pgfbox[bottom,left]{\fontsize{6.26}{7.51}\selectfont \makebox[0pt][r]{$\sigma_{2i}$}}}
\pgfputat{\pgfxy(7.00,67.00)}{\pgfbox[bottom,left]{\fontsize{6.26}{7.51}\selectfont \makebox[0pt][r]{$b^{(i)}$}}}
\pgfputat{\pgfxy(36.00,86.00)}{\pgfbox[bottom,left]{\fontsize{6.26}{7.51}\selectfont \makebox[0pt][r]{$S_{j+1}$}}}
\pgfcircle[fill]{\pgfxy(10.00,70.00)}{0.55mm}
\pgfsetdash{}{0mm}\pgfcircle[stroke]{\pgfxy(10.00,70.00)}{0.55mm}
\pgfcircle[fill]{\pgfxy(20.00,60.00)}{0.55mm}
\pgfcircle[stroke]{\pgfxy(20.00,60.00)}{0.55mm}
\pgfputat{\pgfxy(46.00,76.00)}{\pgfbox[bottom,left]{\fontsize{6.26}{7.51}\selectfont \makebox[0pt][r]{$S_{j}$}}}
\pgfcircle[fill]{\pgfxy(0.00,80.00)}{0.55mm}
\pgfsetlinewidth{0.30mm}\pgfcircle[stroke]{\pgfxy(0.00,80.00)}{0.55mm}
\pgfputat{\pgfxy(-2.00,79.00)}{\pgfbox[bottom,left]{\fontsize{6.26}{7.51}\selectfont \makebox[0pt][r]{$\sigma_{2i-1}$}}}
\pgfmoveto{\pgfxy(0.00,80.00)}\pgflineto{\pgfxy(60.00,20.00)}\pgfstroke
\pgfmoveto{\pgfxy(70.00,30.00)}\pgflineto{\pgfxy(60.00,20.00)}\pgfstroke
\pgfcircle[fill]{\pgfxy(70.00,30.00)}{0.55mm}
\pgfcircle[stroke]{\pgfxy(70.00,30.00)}{0.55mm}
\pgfputat{\pgfxy(76.00,46.00)}{\pgfbox[bottom,left]{\fontsize{6.26}{7.51}\selectfont \makebox[0pt][r]{$S_2$}}}
\pgfcircle[fill]{\pgfxy(50.00,30.00)}{0.55mm}
\pgfsetlinewidth{0.60mm}\pgfcircle[stroke]{\pgfxy(50.00,30.00)}{0.55mm}
\pgfsetlinewidth{0.30mm}\pgfmoveto{\pgfxy(60.00,40.00)}\pgflineto{\pgfxy(50.00,30.00)}\pgfstroke
\pgfcircle[fill]{\pgfxy(60.00,40.00)}{0.55mm}
\pgfcircle[stroke]{\pgfxy(60.00,40.00)}{0.55mm}
\pgfputat{\pgfxy(56.00,66.00)}{\pgfbox[bottom,left]{\fontsize{6.26}{7.51}\selectfont \makebox[0pt][r]{$S_{j-1}$}}}
\pgfcircle[fill]{\pgfxy(30.00,50.00)}{0.55mm}
\pgfsetlinewidth{0.60mm}\pgfcircle[stroke]{\pgfxy(30.00,50.00)}{0.55mm}
\pgfsetlinewidth{0.30mm}\pgfmoveto{\pgfxy(40.00,60.00)}\pgflineto{\pgfxy(30.00,50.00)}\pgfstroke
\pgfcircle[fill]{\pgfxy(40.00,60.00)}{0.55mm}
\pgfcircle[stroke]{\pgfxy(40.00,60.00)}{0.55mm}
\pgfcircle[fill]{\pgfxy(70.00,10.00)}{0.55mm}
\pgfcircle[stroke]{\pgfxy(70.00,10.00)}{0.55mm}
\pgfcircle[fill]{\pgfxy(90.00,-10.00)}{0.55mm}
\pgfcircle[stroke]{\pgfxy(90.00,-10.00)}{0.55mm}
\pgfsetdash{{0.60mm}{0.50mm}}{0mm}\pgfsetlinewidth{0.60mm}\pgfmoveto{\pgfxy(91.00,-1.00)}\pgflineto{\pgfxy(81.00,9.00)}\pgfstroke
\pgfsetdash{}{0mm}\pgfsetlinewidth{0.30mm}\pgfmoveto{\pgfxy(70.00,10.00)}\pgflineto{\pgfxy(90.00,-10.00)}\pgfstroke
\pgfputat{\pgfxy(70.00,32.00)}{\pgfbox[bottom,left]{\fontsize{6.26}{7.51}\selectfont $u_1$}}
\pgfputat{\pgfxy(60.00,42.00)}{\pgfbox[bottom,left]{\fontsize{6.26}{7.51}\selectfont $u_2$}}
\pgfputat{\pgfxy(42.00,54.00)}{\pgfbox[bottom,left]{\fontsize{6.26}{7.51}\selectfont $u_{j-1}$}}
\pgfputat{\pgfxy(17.00,57.00)}{\pgfbox[bottom,left]{\fontsize{6.26}{7.51}\selectfont \makebox[0pt][r]{$v_{j-1}$}}}
\pgfputat{\pgfxy(27.00,47.00)}{\pgfbox[bottom,left]{\fontsize{6.26}{7.51}\selectfont \makebox[0pt][r]{$v_{j-2}$}}}
\pgfmoveto{\pgfxy(90.00,-11.00)}\pgfcurveto{\pgfxy(93.29,-8.93)}{\pgfxy(96.63,-6.93)}{\pgfxy(100.00,-5.00)}\pgfcurveto{\pgfxy(104.15,-2.62)}{\pgfxy(108.39,-0.33)}{\pgfxy(113.00,1.00)}\pgfcurveto{\pgfxy(116.32,1.96)}{\pgfxy(120.00,3.27)}{\pgfxy(120.00,6.50)}\pgfcurveto{\pgfxy(120.00,8.99)}{\pgfxy(117.62,10.65)}{\pgfxy(115.00,11.00)}\pgfcurveto{\pgfxy(111.89,11.42)}{\pgfxy(108.85,10.30)}{\pgfxy(106.00,9.00)}\pgfcurveto{\pgfxy(100.92,6.69)}{\pgfxy(96.05,3.68)}{\pgfxy(93.00,-1.00)}\pgfcurveto{\pgfxy(91.06,-3.98)}{\pgfxy(90.02,-7.45)}{\pgfxy(90.00,-11.00)}\pgfstroke
\pgfmoveto{\pgfxy(60.00,20.00)}\pgfcurveto{\pgfxy(63.29,22.07)}{\pgfxy(66.63,24.07)}{\pgfxy(70.00,26.00)}\pgfcurveto{\pgfxy(74.15,28.38)}{\pgfxy(78.39,30.67)}{\pgfxy(83.00,32.00)}\pgfcurveto{\pgfxy(86.32,32.96)}{\pgfxy(90.00,34.27)}{\pgfxy(90.00,37.50)}\pgfcurveto{\pgfxy(90.00,39.99)}{\pgfxy(87.62,41.65)}{\pgfxy(85.00,42.00)}\pgfcurveto{\pgfxy(81.89,42.42)}{\pgfxy(78.85,41.30)}{\pgfxy(76.00,40.00)}\pgfcurveto{\pgfxy(70.92,37.69)}{\pgfxy(66.05,34.68)}{\pgfxy(63.00,30.00)}\pgfcurveto{\pgfxy(61.06,27.02)}{\pgfxy(60.02,23.55)}{\pgfxy(60.00,20.00)}\pgfstroke
\pgfmoveto{\pgfxy(70.00,10.00)}\pgfcurveto{\pgfxy(73.29,12.07)}{\pgfxy(76.63,14.07)}{\pgfxy(80.00,16.00)}\pgfcurveto{\pgfxy(84.15,18.38)}{\pgfxy(88.39,20.67)}{\pgfxy(93.00,22.00)}\pgfcurveto{\pgfxy(96.32,22.96)}{\pgfxy(100.00,24.27)}{\pgfxy(100.00,27.50)}\pgfcurveto{\pgfxy(100.00,29.99)}{\pgfxy(97.62,31.65)}{\pgfxy(95.00,32.00)}\pgfcurveto{\pgfxy(91.89,32.42)}{\pgfxy(88.85,31.30)}{\pgfxy(86.00,30.00)}\pgfcurveto{\pgfxy(80.92,27.69)}{\pgfxy(76.05,24.68)}{\pgfxy(73.00,20.00)}\pgfcurveto{\pgfxy(71.06,17.02)}{\pgfxy(70.02,13.55)}{\pgfxy(70.00,10.00)}\pgfstroke
\pgfmoveto{\pgfxy(50.00,30.00)}\pgfcurveto{\pgfxy(53.29,32.07)}{\pgfxy(56.63,34.07)}{\pgfxy(60.00,36.00)}\pgfcurveto{\pgfxy(64.15,38.38)}{\pgfxy(68.39,40.67)}{\pgfxy(73.00,42.00)}\pgfcurveto{\pgfxy(76.32,42.96)}{\pgfxy(80.00,44.27)}{\pgfxy(80.00,47.50)}\pgfcurveto{\pgfxy(80.00,49.99)}{\pgfxy(77.62,51.65)}{\pgfxy(75.00,52.00)}\pgfcurveto{\pgfxy(71.89,52.42)}{\pgfxy(68.85,51.30)}{\pgfxy(66.00,50.00)}\pgfcurveto{\pgfxy(60.92,47.69)}{\pgfxy(56.05,44.68)}{\pgfxy(53.00,40.00)}\pgfcurveto{\pgfxy(51.06,37.02)}{\pgfxy(50.02,33.55)}{\pgfxy(50.00,30.00)}\pgfstroke
\pgfmoveto{\pgfxy(30.00,50.00)}\pgfcurveto{\pgfxy(33.29,52.07)}{\pgfxy(36.63,54.07)}{\pgfxy(40.00,56.00)}\pgfcurveto{\pgfxy(44.15,58.38)}{\pgfxy(48.39,60.67)}{\pgfxy(53.00,62.00)}\pgfcurveto{\pgfxy(56.32,62.96)}{\pgfxy(60.00,64.27)}{\pgfxy(60.00,67.50)}\pgfcurveto{\pgfxy(60.00,69.99)}{\pgfxy(57.62,71.65)}{\pgfxy(55.00,72.00)}\pgfcurveto{\pgfxy(51.89,72.42)}{\pgfxy(48.85,71.30)}{\pgfxy(46.00,70.00)}\pgfcurveto{\pgfxy(40.92,67.69)}{\pgfxy(36.05,64.68)}{\pgfxy(33.00,60.00)}\pgfcurveto{\pgfxy(31.06,57.02)}{\pgfxy(30.02,53.55)}{\pgfxy(30.00,50.00)}\pgfstroke
\pgfmoveto{\pgfxy(20.00,60.00)}\pgfcurveto{\pgfxy(23.29,62.07)}{\pgfxy(26.63,64.07)}{\pgfxy(30.00,66.00)}\pgfcurveto{\pgfxy(34.15,68.38)}{\pgfxy(38.39,70.67)}{\pgfxy(43.00,72.00)}\pgfcurveto{\pgfxy(46.32,72.96)}{\pgfxy(50.00,74.27)}{\pgfxy(50.00,77.50)}\pgfcurveto{\pgfxy(50.00,79.99)}{\pgfxy(47.62,81.65)}{\pgfxy(45.00,82.00)}\pgfcurveto{\pgfxy(41.89,82.42)}{\pgfxy(38.85,81.30)}{\pgfxy(36.00,80.00)}\pgfcurveto{\pgfxy(30.92,77.69)}{\pgfxy(26.05,74.68)}{\pgfxy(23.00,70.00)}\pgfcurveto{\pgfxy(21.06,67.02)}{\pgfxy(20.02,63.55)}{\pgfxy(20.00,60.00)}\pgfstroke
\pgfmoveto{\pgfxy(10.00,70.00)}\pgfcurveto{\pgfxy(13.29,72.07)}{\pgfxy(16.63,74.07)}{\pgfxy(20.00,76.00)}\pgfcurveto{\pgfxy(24.15,78.38)}{\pgfxy(28.39,80.67)}{\pgfxy(33.00,82.00)}\pgfcurveto{\pgfxy(36.32,82.96)}{\pgfxy(40.00,84.27)}{\pgfxy(40.00,87.50)}\pgfcurveto{\pgfxy(40.00,89.99)}{\pgfxy(37.62,91.65)}{\pgfxy(35.00,92.00)}\pgfcurveto{\pgfxy(31.89,92.42)}{\pgfxy(28.85,91.30)}{\pgfxy(26.00,90.00)}\pgfcurveto{\pgfxy(20.92,87.69)}{\pgfxy(16.05,84.68)}{\pgfxy(13.00,80.00)}\pgfcurveto{\pgfxy(11.06,77.02)}{\pgfxy(10.02,73.55)}{\pgfxy(10.00,70.00)}\pgfstroke
\pgfputat{\pgfxy(-122.00,37.00)}{\pgfbox[bottom,left]{\fontsize{6.26}{7.51}\selectfont \makebox[0pt][r]{$\sigma_{2i-1}$}}}
\pgfputat{\pgfxy(-102.00,17.00)}{\pgfbox[bottom,left]{\fontsize{6.26}{7.51}\selectfont \makebox[0pt][r]{$\sigma_{2i}$}}}
\pgfmoveto{\pgfxy(-90.00,10.00)}\pgflineto{\pgfxy(-70.00,-10.00)}\pgfstroke
\pgfcircle[fill]{\pgfxy(-90.00,10.00)}{0.55mm}
\pgfcircle[stroke]{\pgfxy(-90.00,10.00)}{0.55mm}
\pgfcircle[fill]{\pgfxy(-70.00,-10.00)}{0.55mm}
\pgfcircle[stroke]{\pgfxy(-70.00,-10.00)}{0.55mm}
\pgfsetdash{{0.60mm}{0.50mm}}{0mm}\pgfsetlinewidth{0.60mm}\pgfmoveto{\pgfxy(-70.00,0.00)}\pgflineto{\pgfxy(-80.00,10.00)}\pgfstroke
\pgfsetdash{}{0mm}\pgfsetlinewidth{0.30mm}\pgfmoveto{\pgfxy(-70.00,-11.00)}\pgfcurveto{\pgfxy(-66.71,-8.93)}{\pgfxy(-63.37,-6.93)}{\pgfxy(-60.00,-5.00)}\pgfcurveto{\pgfxy(-55.85,-2.62)}{\pgfxy(-51.61,-0.33)}{\pgfxy(-47.00,1.00)}\pgfcurveto{\pgfxy(-43.68,1.96)}{\pgfxy(-40.00,3.27)}{\pgfxy(-40.00,6.50)}\pgfcurveto{\pgfxy(-40.00,8.99)}{\pgfxy(-42.38,10.65)}{\pgfxy(-45.00,11.00)}\pgfcurveto{\pgfxy(-48.11,11.42)}{\pgfxy(-51.15,10.30)}{\pgfxy(-54.00,9.00)}\pgfcurveto{\pgfxy(-59.08,6.69)}{\pgfxy(-63.95,3.68)}{\pgfxy(-67.00,-1.00)}\pgfcurveto{\pgfxy(-68.94,-3.98)}{\pgfxy(-69.98,-7.45)}{\pgfxy(-70.00,-11.00)}\pgfstroke
\pgfmoveto{\pgfxy(-90.00,10.00)}\pgfcurveto{\pgfxy(-86.71,12.07)}{\pgfxy(-83.37,14.07)}{\pgfxy(-80.00,16.00)}\pgfcurveto{\pgfxy(-75.85,18.38)}{\pgfxy(-71.61,20.67)}{\pgfxy(-67.00,22.00)}\pgfcurveto{\pgfxy(-63.68,22.96)}{\pgfxy(-60.00,24.27)}{\pgfxy(-60.00,27.50)}\pgfcurveto{\pgfxy(-60.00,29.99)}{\pgfxy(-62.38,31.65)}{\pgfxy(-65.00,32.00)}\pgfcurveto{\pgfxy(-68.11,32.42)}{\pgfxy(-71.15,31.30)}{\pgfxy(-74.00,30.00)}\pgfcurveto{\pgfxy(-79.08,27.69)}{\pgfxy(-83.95,24.68)}{\pgfxy(-87.00,20.00)}\pgfcurveto{\pgfxy(-88.94,17.02)}{\pgfxy(-89.98,13.55)}{\pgfxy(-90.00,10.00)}\pgfstroke
\pgfmoveto{\pgfxy(-120.00,40.00)}\pgflineto{\pgfxy(-110.00,30.00)}\pgfstroke
\pgfputat{\pgfxy(-112.00,27.00)}{\pgfbox[bottom,left]{\fontsize{6.26}{7.51}\selectfont \makebox[0pt][r]{$a^{(i)}$}}}
\pgfmoveto{\pgfxy(-100.00,20.00)}\pgfcurveto{\pgfxy(-96.71,22.07)}{\pgfxy(-93.37,24.07)}{\pgfxy(-90.00,26.00)}\pgfcurveto{\pgfxy(-85.85,28.38)}{\pgfxy(-81.61,30.67)}{\pgfxy(-77.00,32.00)}\pgfcurveto{\pgfxy(-73.68,32.96)}{\pgfxy(-70.00,34.27)}{\pgfxy(-70.00,37.50)}\pgfcurveto{\pgfxy(-70.00,39.99)}{\pgfxy(-72.38,41.65)}{\pgfxy(-75.00,42.00)}\pgfcurveto{\pgfxy(-78.11,42.42)}{\pgfxy(-81.15,41.30)}{\pgfxy(-84.00,40.00)}\pgfcurveto{\pgfxy(-89.08,37.69)}{\pgfxy(-93.95,34.68)}{\pgfxy(-97.00,30.00)}\pgfcurveto{\pgfxy(-98.94,27.02)}{\pgfxy(-99.98,23.55)}{\pgfxy(-100.00,20.00)}\pgfstroke
\pgfputat{\pgfxy(-74.00,37.00)}{\pgfbox[bottom,left]{\fontsize{6.26}{7.51}\selectfont \makebox[0pt][r]{$S_1$}}}
\pgfmoveto{\pgfxy(-100.00,20.00)}\pgflineto{\pgfxy(-90.00,30.00)}\pgfstroke
\pgfcircle[fill]{\pgfxy(-90.00,30.00)}{0.55mm}
\pgfcircle[stroke]{\pgfxy(-90.00,30.00)}{0.55mm}
\pgfputat{\pgfxy(-90.00,32.00)}{\pgfbox[bottom,left]{\fontsize{6.26}{7.51}\selectfont $u_1$}}
\pgfcircle[fill]{\pgfxy(-110.00,30.00)}{0.55mm}
\pgfcircle[stroke]{\pgfxy(-110.00,30.00)}{0.55mm}
\pgfcircle[fill]{\pgfxy(-100.00,40.00)}{0.55mm}
\pgfcircle[stroke]{\pgfxy(-105.00,49.00)}{0.00mm}
\pgfmoveto{\pgfxy(-110.00,50.00)}\pgflineto{\pgfxy(-100.00,40.00)}\pgfstroke
\pgfcircle[fill]{\pgfxy(-70.00,70.00)}{0.55mm}
\pgfcircle[stroke]{\pgfxy(-70.00,70.00)}{0.55mm}
\pgfputat{\pgfxy(-44.00,86.00)}{\pgfbox[bottom,left]{\fontsize{6.26}{7.51}\selectfont \makebox[0pt][r]{$S_{j+1}$}}}
\pgfputat{\pgfxy(-70.00,65.00)}{\pgfbox[bottom,left]{\fontsize{6.26}{7.51}\selectfont $b^{(i)}$}}
\pgfputat{\pgfxy(-96.00,86.00)}{\pgfbox[bottom,left]{\fontsize{6.26}{7.51}\selectfont $S_{j}$}}
\pgfsetdash{{0.60mm}{0.50mm}}{0mm}\pgfsetlinewidth{0.60mm}\pgfmoveto{\pgfxy(-100.00,50.00)}\pgflineto{\pgfxy(-90.00,60.00)}\pgfstroke
\pgfcircle[fill]{\pgfxy(-80.00,60.00)}{0.55mm}
\pgfsetdash{}{0mm}\pgfsetlinewidth{0.30mm}\pgfcircle[stroke]{\pgfxy(-80.00,60.00)}{0.55mm}
\pgfmoveto{\pgfxy(-90.00,70.00)}\pgflineto{\pgfxy(-80.00,60.00)}\pgfstroke
\pgfcircle[fill]{\pgfxy(-90.00,70.00)}{0.55mm}
\pgfcircle[stroke]{\pgfxy(-90.00,70.00)}{0.55mm}
\pgfputat{\pgfxy(-106.00,76.00)}{\pgfbox[bottom,left]{\fontsize{6.26}{7.51}\selectfont $S_{j-1}$}}
\pgfputat{\pgfxy(-126.00,56.00)}{\pgfbox[bottom,left]{\fontsize{6.26}{7.51}\selectfont $S_2$}}
\pgfcircle[fill]{\pgfxy(-100.00,40.00)}{0.55mm}
\pgfcircle[stroke]{\pgfxy(-100.00,40.00)}{0.55mm}
\pgfcircle[fill]{\pgfxy(-110.00,50.00)}{0.55mm}
\pgfcircle[stroke]{\pgfxy(-110.00,50.00)}{0.55mm}
\pgfputat{\pgfxy(-92.00,64.00)}{\pgfbox[bottom,left]{\fontsize{6.26}{7.51}\selectfont \makebox[0pt][r]{$u_{j-1}$}}}
\pgfputat{\pgfxy(-110.00,52.00)}{\pgfbox[bottom,left]{\fontsize{6.26}{7.51}\selectfont \makebox[0pt][r]{$u_2$}}}
\pgfputat{\pgfxy(-80.00,55.00)}{\pgfbox[bottom,left]{\fontsize{6.26}{7.51}\selectfont $v_{j-1}$}}
\pgfmoveto{\pgfxy(-80.00,60.00)}\pgfcurveto{\pgfxy(-83.29,62.07)}{\pgfxy(-86.63,64.07)}{\pgfxy(-90.00,66.00)}\pgfcurveto{\pgfxy(-94.15,68.38)}{\pgfxy(-98.39,70.67)}{\pgfxy(-103.00,72.00)}\pgfcurveto{\pgfxy(-106.32,72.96)}{\pgfxy(-110.00,74.27)}{\pgfxy(-110.00,77.50)}\pgfcurveto{\pgfxy(-110.00,79.99)}{\pgfxy(-107.62,81.65)}{\pgfxy(-105.00,82.00)}\pgfcurveto{\pgfxy(-101.89,82.42)}{\pgfxy(-98.85,81.30)}{\pgfxy(-96.00,80.00)}\pgfcurveto{\pgfxy(-90.92,77.69)}{\pgfxy(-86.05,74.68)}{\pgfxy(-83.00,70.00)}\pgfcurveto{\pgfxy(-81.06,67.02)}{\pgfxy(-80.02,63.55)}{\pgfxy(-80.00,60.00)}\pgfstroke
\pgfmoveto{\pgfxy(-100.00,40.00)}\pgfcurveto{\pgfxy(-103.29,42.07)}{\pgfxy(-106.63,44.07)}{\pgfxy(-110.00,46.00)}\pgfcurveto{\pgfxy(-114.15,48.38)}{\pgfxy(-118.39,50.67)}{\pgfxy(-123.00,52.00)}\pgfcurveto{\pgfxy(-126.32,52.96)}{\pgfxy(-130.00,54.27)}{\pgfxy(-130.00,57.50)}\pgfcurveto{\pgfxy(-130.00,59.99)}{\pgfxy(-127.62,61.65)}{\pgfxy(-125.00,62.00)}\pgfcurveto{\pgfxy(-121.89,62.42)}{\pgfxy(-118.85,61.30)}{\pgfxy(-116.00,60.00)}\pgfcurveto{\pgfxy(-110.92,57.69)}{\pgfxy(-106.05,54.68)}{\pgfxy(-103.00,50.00)}\pgfcurveto{\pgfxy(-101.06,47.02)}{\pgfxy(-100.02,43.55)}{\pgfxy(-100.00,40.00)}\pgfstroke
\pgfmoveto{\pgfxy(-70.00,70.00)}\pgfcurveto{\pgfxy(-73.29,72.07)}{\pgfxy(-76.63,74.07)}{\pgfxy(-80.00,76.00)}\pgfcurveto{\pgfxy(-84.15,78.38)}{\pgfxy(-88.39,80.67)}{\pgfxy(-93.00,82.00)}\pgfcurveto{\pgfxy(-96.32,82.96)}{\pgfxy(-100.00,84.27)}{\pgfxy(-100.00,87.50)}\pgfcurveto{\pgfxy(-100.00,89.99)}{\pgfxy(-97.62,91.65)}{\pgfxy(-95.00,92.00)}\pgfcurveto{\pgfxy(-91.89,92.42)}{\pgfxy(-88.85,91.30)}{\pgfxy(-86.00,90.00)}\pgfcurveto{\pgfxy(-80.92,87.69)}{\pgfxy(-76.05,84.68)}{\pgfxy(-73.00,80.00)}\pgfcurveto{\pgfxy(-71.06,77.02)}{\pgfxy(-70.02,73.55)}{\pgfxy(-70.00,70.00)}\pgfstroke
\pgfmoveto{\pgfxy(-70.00,70.00)}\pgfcurveto{\pgfxy(-66.71,72.07)}{\pgfxy(-63.37,74.07)}{\pgfxy(-60.00,76.00)}\pgfcurveto{\pgfxy(-55.85,78.38)}{\pgfxy(-51.61,80.67)}{\pgfxy(-47.00,82.00)}\pgfcurveto{\pgfxy(-43.68,82.96)}{\pgfxy(-40.00,84.27)}{\pgfxy(-40.00,87.50)}\pgfcurveto{\pgfxy(-40.00,89.99)}{\pgfxy(-42.38,91.65)}{\pgfxy(-45.00,92.00)}\pgfcurveto{\pgfxy(-48.11,92.42)}{\pgfxy(-51.15,91.30)}{\pgfxy(-54.00,90.00)}\pgfcurveto{\pgfxy(-59.08,87.69)}{\pgfxy(-63.95,84.68)}{\pgfxy(-67.00,80.00)}\pgfcurveto{\pgfxy(-68.94,77.02)}{\pgfxy(-69.98,73.55)}{\pgfxy(-70.00,70.00)}\pgfstroke
\pgfsetlinewidth{0.90mm}\pgfmoveto{\pgfxy(-110.00,30.00)}\pgflineto{\pgfxy(-70.00,70.00)}\pgfstroke
\pgfputat{\pgfxy(-101.00,35.00)}{\pgfbox[bottom,left]{\fontsize{6.26}{7.51}\selectfont $v_2$}}
\pgfsetlinewidth{0.30mm}\pgfmoveto{\pgfxy(-90.00,10.00)}\pgflineto{\pgfxy(-110.00,30.00)}\pgfstroke
\pgfcircle[fill]{\pgfxy(-120.00,40.00)}{0.55mm}
\pgfcircle[stroke]{\pgfxy(-120.00,40.00)}{0.55mm}
\pgfcircle[fill]{\pgfxy(-100.00,20.00)}{0.55mm}
\pgfcircle[stroke]{\pgfxy(-100.00,20.00)}{0.55mm}
\pgfsetlinewidth{1.20mm}\pgfmoveto{\pgfxy(-40.00,30.00)}\pgflineto{\pgfxy(-30.00,30.00)}\pgfstroke
\pgfmoveto{\pgfxy(-30.00,30.00)}\pgflineto{\pgfxy(-32.80,30.70)}\pgflineto{\pgfxy(-32.80,29.30)}\pgflineto{\pgfxy(-30.00,30.00)}\pgfclosepath\pgffill
\pgfmoveto{\pgfxy(-30.00,30.00)}\pgflineto{\pgfxy(-32.80,30.70)}\pgflineto{\pgfxy(-32.80,29.30)}\pgflineto{\pgfxy(-30.00,30.00)}\pgfclosepath\pgfstroke
\pgfsetdash{{0.60mm}{0.50mm}}{0mm}\pgfsetlinewidth{0.60mm}\pgfmoveto{\pgfxy(-23.00,30.00)}\pgflineto{\pgfxy(-13.00,30.00)}\pgfstroke
\pgfputat{\pgfxy(-35.00,37.00)}{\pgfbox[bottom,left]{\fontsize{6.26}{7.51}\selectfont \makebox[0pt]{(B2b)}}}
\pgfputat{\pgfxy(-5.00,37.00)}{\pgfbox[bottom,left]{\fontsize{6.26}{7.51}\selectfont \makebox[0pt]{(B2b)}}}
\end{pgfpicture}%
$$
Since (C2a) is produced from the rule (B1) and (B2b), 
then Algorithm~C follows Algorithm~B.
\end{proof}

\section*{Acknowledgements}
The first author's work was supported by the National Research Foundation of Korea(NRF) grant funded by the Korea government(MSIP) (No. 2017R1C1B2008269).




\end{document}